\documentclass[11pt]{article}
\usepackage[latin1]{inputenc}
\usepackage{amsmath,amsthm,amssymb}
\usepackage{amsfonts}
\usepackage{amsmath,amsthm,amssymb,amscd}
\usepackage{latexsym}
\usepackage{color}
\usepackage{graphicx}
\usepackage{mathrsfs}
\usepackage{cite}
\usepackage{color,enumitem,graphicx}
\usepackage[colorlinks=true,urlcolor=black,
citecolor=black,linkcolor=black,linktocpage,pdfpagelabels,
bookmarksnumbered,bookmarksopen]{hyperref}

\makeatletter \@addtoreset{equation}{section} \makeatother

\newtheorem{theorem}{Theorem}[section]

\newtheorem{proposition}{Proposition}[section]
\newtheorem{lemma}{Lemma}[section]
\newtheorem{remark}{Remark}[section]

\allowdisplaybreaks

\begin{document}
\title{New type of solutions for the critical Lane-Emden system}

\author{Wenjing Chen\footnote{Corresponding author.}\ \footnote{E-mail address:\, {\tt wjchen@swu.edu.cn} (W. Chen), {\tt hhuangxiaomeng@126.com} (X. Huang).}\  \ and Xiaomeng Huang\\
\footnotesize  School of Mathematics and Statistics, Southwest University,
Chongqing, 400715, P.R. China}

\date{ }
\maketitle

\begin{abstract}
{In this paper, we consider the critical Lane-Emden system
\begin{align*}
\begin{cases}
-\Delta u=K_1(y)v^p,\quad y\in \mathbb{R}^N,&\\
-\Delta v=K_2(y)u^q,\quad  y\in \mathbb{R}^N,&\\
u,v>0,
\end{cases}
\end{align*}
where $N\geq 5$, $p,q\in (1,\infty)$ with $\frac{1}{p+1}+\frac{1}{q+1}=\frac{N-2}{N}$, $K_1(y)$ and $K_2(y)$ are positive radial potentials. Under suitable conditions on $K_1(y)$ and $K_2(y)$, we construct a new family of solutions to this system, which are centred at points lying on the top and the bottom circles of a cylinder.}



\end{abstract}

\section{Introduction and statement of main result}\label{introduction}

In this work,  we consider the following critical Lane-Emden system 
\begin{align}\label{le}
\begin{cases}
-\Delta u=K_1(y)v^p,\quad   y\in \mathbb{R}^N,&\\
-\Delta v=K_2(y)u^q,\quad   y\in \mathbb{R}^N,&\\
u,v>0,
\end{cases}
\end{align}
where $N\geq 5$, $K_1(y)$ and $K_2(y)$ are positive radial potentials, $p,q\in (1,\infty)$ and $(p, q)$ is a pair of positive numbers lying on the critical hyperbola
\begin{equation}\label{pq}
\frac{1}{p+1}+\frac{1}{q+1}=\frac{N-2}{N}.
\end{equation}
In this paper, without loss of generality, we may assume that $p\leq \frac{N+2}{N-2}\leq q$.

When $u=v$ and $K_1(y)=K_2(y)=K(y)$, system (\ref{le}) is reduced to the following scalar equation
\begin{align}\label{psc}
-\Delta u=K(y)u^{\frac{N+2}{N-2}},\  u>0,\  y\in \mathbb{R}^N,\  u\in D^{1,2}(\mathbb{R}^N),
\end{align}
which can be transformed by the prescribed scalar curvature problem on $\mathbb{S}^N$ through the stereo-graphic projection. There are many works revelent to (\ref{psc}), see \cite{caony, dengly, li, noussaury, yan} and reference therein.
Wei and Yan \cite{weiy} showed that if $K(y)$ is radially symmetric and $K(|y|)$ has a local maximum at $r_0>0$, then (\ref{psc}) has infinitely many non-radial solutions with large number of bubbles lying near the sphere $|y|=r_0$ and the distance between different bubbles can be made arbitrary small.
Subsequently, Guo et al. \cite{guompy} proved a non-degeneracy result for the positive multi-bubbling solutions constructed in \cite{weiy} and used this non-degeneracy result to glue $n$-bubbles, whose centers lie in the circle $|y|=r_0$ in the $(y_3,y_4)$-plane to the $k$-bubbling solution described in \cite{weiy}.
Under more general assumptions on $K(y)$, i.e., $K(y)=K(|y'|,y'')$, where $y=(|y'|,y'')\in \mathbb{R}^2\times \mathbb{R}^{N-2}$, by combining a finite reduction argument and local Pohozaev type of identities, Peng, Wang and Wei\cite{pengww} proved problem (\ref{psc}) has infinitely many solutions.
Recently, Duan, Musso and Wei\cite{duanmw} constructed a family of solutions to problem (\ref{psc}) and the solutions are centred at points lying on the top and the bottom circles of a cylinder.

Consider the standard Lane-Emden system 
\begin{align}\label{le0}
\begin{cases}
-\Delta u=v^p,\quad  \text{ in } \Omega,&\\
-\Delta v=u^q,\quad  \text{ in } \Omega,&\\
u=v=0,       \quad   \text{ on } \partial\Omega,&
\end{cases}
\end{align}
where $\Omega$ is either a smooth bounded domain or the whole space $\mathbb{R}^N$ with $N\geq 3$ and $p,q\in (1,\infty)$. Such
problems have been intensively studied.
For $\Omega=\mathbb{R}^N$, by the concentration compactness principle, Lions \cite{lions} obtained a positive least energy solution to (\ref{le0}) in $X_{p,q}(\mathbb{R}^N):=\dot{W}^{2,\frac{p+1}{p}}(\mathbb{R}^N)\times \dot{W}^{2,\frac{q+1}{q}}(\mathbb{R}^N)$ when $p,q$ satisfy condition (\ref{pq}).
Later, Alvino et al.\cite{alvinolt}, Wang\cite{wang}, Hulshof, van der Vorst\cite{hulshofv} respectively proved the radially symmetry, monotonicity(i.e, decreasing in the radial variable), uniqueness of this positive least energy solution.
Moreover, Frank, Kim and Pistoia\cite{frankkp} deduced that all least energy solutions to (\ref{le0}) satisfying (\ref{pq}) are non-degenerate. For more result about the positive or sign-changing solutions to (\ref{le0}) on $\mathbb{R}^N$, we can refer the
readers to \cite{chenlo, clapps} and references therein.

If $\Omega\subset\mathbb{R}^N$ is a smooth bounded domain, some efforts have been devoted to the study of (\ref{le0}). We
refer the readers to \cite{alvescm, bonheuredr, clementfm, figueiredo, figueiredof, sirakov} and the references therein for the existence, positivity and uniqueness of solutions for (\ref{le0}).
Kim and Moon\cite{kimm} examined asymptotic behavior of solutions of (\ref{le0}) on a smooth bounded convex domain $\Omega$ in $\mathbb{R}^N$ provided that the pair $(p,q)$ satisfies $p\leq q$, and $(p, q)$ in the subcritical region is close to the critical hyperbola.
Based the non-degeneracy result obtained in \cite{frankkp}, Kim and Pistoia\cite{kimp} proved the existence of a one-point blowing-up solution to (\ref{le0}) in general domains and multiple blowing-up solutions to (\ref{le0}) in contractible domains.
Subsequently, Guo, Hu and Peng\cite{guohp} revisited the multiple blowing-up solutions constructed in \cite{kimp} and proved its non-degeneracy.
Recently, Jin and Kim\cite{jink} were interested in (\ref{le0}), when $\Omega$ is a
smooth bounded domain with a small hole, they proved that the system admits a family of positive
solutions that concentrate around the center of the hole as $\epsilon\rightarrow 0$, obtaining a concrete qualitative
description of the solutions as well.
Guo, Liu and Peng\cite{guolpboun} considered the supercritical problem for the Lane-Emden system (\ref{le0}) and  proved that for some suitable bounded smooth domains $\Omega\subset\mathbb{R}^N$, there exist positive solutions with layers concentrating along one or several $k$-dimensional sub-manifolds of $\partial\Omega$ as $(p, q)$ is close to the critical hyperbola.

In contrast to the Dirichlet problem (\ref{le0}), there have been very few results to the Lane-Emden system with Neumann conditions. Using dual method, Salda\~{n}a and Tavares\cite{saldanat} obtained the existence of least energy nodal solutions in the subcritical regime. For critical case, Pistoia, Schiera and Tavares\cite{pistoiast} showed that, under suitable conditions on $(p, q)$, least-energy (sign-changing) solutions exist, and they are classical.
By the Lyapunov-Schmidt reduction argument, Guo and Peng\cite{guop} studied the Lane-Emden system with Neumann conditions, and investigated the existence of multiple concentrated solutions in the unit ball for $(p, q)$ in supercritical regime.

In this paper, we study (\ref{le}). As far as we know, inspired by \cite{weiy}, under suitable conditions on $K_1, K_2$ and the certain range of the exponents $p, q$, Guo, Liu and Peng\cite{guolp} constructed an unbounded sequence of non-radial positive vector solutions, whose energy can be made arbitrarily large. However, no existence result with solutions for (\ref{le}) centring at points lying on the top and the bottom circles of a cylinder  is known in the literature. Therefore, our objective is to  construct such type solution.

We first assume that $N\geq 5$, $K_1$ and $K_2$ are positive and radial functions satisfying the following conditions:

($K_1$) There exists a constant $r_0>0$ such that for $r\in (r_0-\delta,r_0+\delta)$,
\begin{equation*}
K_1(r)=1-c_1|r-r_0|^{m_1}+O(|r-r_0|^{m_1+\theta_1}).
\end{equation*}

($K_2$)There exists a constant $r_0>0$ such that for $r\in (r_0-\delta,r_0+\delta)$,
\begin{equation*}
K_2(r)=1-c_2|r-r_0|^{m_2}+O(|r-r_0|^{m_2+\theta_2}),
\end{equation*}
where $\delta$, $\theta_1$, $\theta_2>0$ are small constants, $c_1$, $c_2>0$, and $m_1, m_2\in [2,N-2)$.
We define $m=\min\{m_1,m_2\}$ and $m$ satisfies
\begin{align}\label{m0}
m\in
\begin{cases}
[2,N-2), & \text{ if } N=5,6,\\[2mm]
\left(\frac{2(N-2)^2}{2N-3},N-2\right), & \text{ if } N\geq 7.
\end{cases}
\end{align}
Moreover, $p$, $q$, $m$ should further satisfy
\begin{align}\label{p+}
&\max\Bigg\{\frac{N+1}{N-2},\frac{N(N-2)}{(N-2)^2-(N-2-m)},
\frac{m+1}{2m}+\frac{\tau}{N-2}+\frac{N+2}{2(N-2)}+\frac{(N-3)(N-2-m)}{2m(N-1)(N-2)}
\Bigg\}
\nonumber\\
&<p\leq \frac{N+2}{N-2}.
\end{align}

\begin{align}\label{q+}
q>\max\Bigg\{\frac{N+2}{N-2}+\frac{N-2-m}{(N-2)^2},
\frac{N+2}{N-2}+\frac{2(N-2-m)}{m(N-1)(N-2)}
\Bigg\}.
\end{align}

\begin{align}\label{m+}
\max\left\{\frac{2N^2-4N+4}{5N-7-2(N-1)\tau}, \frac{2(N-2)^2}{N^2-2N-1}\right\}
<m<\frac{(N-2)[N^2-N+2-2(N-1)\tau]}{N^2-4N+5},
\end{align}
where $\tau$ is defined in (\ref{tau}) below.

\begin{remark}\label{remark1}
We make some explanations for the extra condition on $p$, $q$, $m$.
In (\ref{p+}), the reason for $p>\max\Bigg\{\frac{N+1}{N-2},\frac{N(N-2)}{(N-2)^2-(N-2-m)}\Bigg\}$ is the same to \cite{guolp}.
The condition $p>\frac{m+1}{2m}+\frac{\tau}{N-2}+\frac{N+2}{2(N-2)}+\frac{(N-3)(N-2-m)}{2m(N-1)(N-2)}$ is due to the fact that we want the term $kO\left(r\|\phi_1\|_{*}^2\right)$ controlled by $k\frac{1}{k^{\frac{m(N-2)}{N-2-m}+\frac{N-3}{N-1}+\theta}}$ in Lemma \ref{fh}.
(\ref{q+}) is determined by the compatibility of $\frac{\partial I(W_1,W_2)}{\partial h}$ (see Lemma \ref{iwh}).
(\ref{m+}) is derived by the condition on $p$, which is not only (\ref{p+}), since (\ref{p+}) was simplified.
\end{remark}

Our main result in this paper can be stated as follows:
\begin{theorem}\label{thm1.10}
Suppose that $N\geq 5$ and $p$, $q$, $m$ satisfy (\ref{pq}), (\ref{p+}), (\ref{q+}), (\ref{m0}) and (\ref{m+}) respectively. If $K_i(r)$ satisfies $(K_i)$, then problem (\ref{le}) has infinitely many non-radial solutions.
\end{theorem}

Before we close this introduction, let us outline the main idea in the proof of Theorem \ref{thm1.10}.

For any integer $k$, we denote
\begin{equation}\label{muk}
\mu=k^{\frac{N-2}{N-2-m}}.
\end{equation}
Using the transformation
\begin{equation*}
u(y)\mapsto \mu^{-\frac{N}{q+1}}u\left(\frac{y}{\mu}\right),\quad
v(y)\mapsto \mu^{-\frac{N}{p+1}}v\left(\frac{y}{\mu}\right),
\end{equation*}
system (\ref{le}) becomes
\begin{align}\label{lemu}
\begin{cases}
-\Delta u=K_1\left(\frac{y}{\mu}\right)v^p,\quad  y\in \mathbb{R}^N,&\\
-\Delta v=K_2\left(\frac{y}{\mu}\right)u^q,\quad  y\in \mathbb{R}^N,&\\
u,v>0, \quad
(u,v)\in \dot{W}^{2,\frac{p+1}{p}}(\mathbb{R}^N)\times \dot{W}^{2,\frac{q+1}{q}}(\mathbb{R}^N).
\end{cases}
\end{align}

It is well known that a positive ground state $(U,V)$ to the following system was found in \cite{lions}
\begin{align}\label{limiteq}
\begin{cases}
-\Delta U=|V|^{p-1}V,\text{ in } \mathbb{R}^N,&\\
-\Delta V=|U|^{q-1}U, \text{ in } \mathbb{R}^N,&\\
(U,V)\in \dot{W}^{2,\frac{p+1}{p}}(\mathbb{R}^N)\times \dot{W}^{2,\frac{q+1}{q}}(\mathbb{R}^N),
\end{cases}
\end{align}
where $N\geq 3$ and $(p,q)$ satisfy (\ref{pq}).

In view of \cite{hulshofv} and \cite{wang}, the positive ground state $(U_{0,1}, V_{0,1})$ of (\ref{limiteq}) is unique with $U_{0,1}(0)=1$ and the family of functions

\begin{align*}
(U_{\xi,\lambda}(y),V_{\xi,\lambda}(y))
=\left(\lambda^{\frac{N}{q+1}}U_{0,1}(\lambda(y-\xi)), \lambda^{\frac{N}{p+1}}V_{0,1}(\lambda(y-\xi))\right)
\end{align*}
for any $\lambda>0$, $\xi\in \mathbb{R}^N$ also solves system (\ref{limiteq}).

Define the symmetric Sobolev space:
\begin{align*}
H_s=\Bigg\{&(u,v)\in \dot{W}^{2,\frac{p+1}{p}}(\mathbb{R}^N)\times \dot{W}^{2,\frac{q+1}{q}}(\mathbb{R}^N), u,v \text{ are even in } y_h, h=2, \cdots, N,
\nonumber\\
&\quad u\left(\sqrt{y_1^2+y_2^2}\cos\theta,\sqrt{y_1^2+y_2^2}\sin \theta,y_3,y''\right)
\nonumber\\
&\quad=u\left(\sqrt{y_1^2+y_2^2}\cos\left(\theta+\frac{2j\pi}{k}\right),\sqrt{y_1^2+y_2^2}\sin \left(\theta+\frac{2j\pi}{k}\right),y_3,y''\right),
\nonumber\\
&\quad v\left(\sqrt{y_1^2+y_2^2}\cos\theta,\sqrt{y_1^2+y_2^2}\sin \theta,y_3,y''\right)
\nonumber\\
&\quad=v\left(\sqrt{y_1^2+y_2^2}\cos\left(\theta+\frac{2j\pi}{k}\right),\sqrt{y_1^2+y_2^2}\sin \left(\theta+\frac{2j\pi}{k}\right),y_3,y''\right)
\Bigg\},
\end{align*}
where $y''\in \mathbb{R}^{N-3}$, $\theta=\arctan \frac{y_2}{y_1}$.

For any large integer $k>0$, let
\begin{align*}
\begin{cases}
\overline{x}_j=r\left(\sqrt{1-h^2}\cos\frac{2(j-1)\pi}{k},\sqrt{1-h^2}\sin\frac{2(j-1)\pi}{k},
h,\mathbf{0}\right),\quad j=1, \cdots, k,&\\
\underline{x}_j=r\left(\sqrt{1-h^2}\cos\frac{2(j-1)\pi}{k},\sqrt{1-h^2}\sin\frac{2(j-1)\pi}{k},
-h,\mathbf{0}\right),\quad j=1, \cdots, k,
\end{cases}
\end{align*}
where $\mathbf{0}$ is the zero vector in $\mathbb{R}^{N-3}$ and $h$, $r$ are positive parameters. We define
\begin{align*}
W_1(y):=W_{1,r,h,\Lambda}(y)=\sum_{j=1}^kU_{\overline{x}_j,\Lambda}+\sum_{j=1}^kU_{\underline{x}_j,\Lambda},\\
W_2(y):=W_{2,r,h,\Lambda}(y)=\sum_{j=1}^kV_{\overline{x}_j,\Lambda}+\sum_{j=1}^kV_{\underline{x}_j,\Lambda}.
\end{align*}

In this paper, we assume $N\geq 5$ and $(r,h,\Lambda)\in \wp_k$, where $\wp_k$ is defined by
\begin{align}\label{wpk}
\wp_k=\Bigg\{(r,h,\Lambda)|\ &r\in \left[k^{\frac{N-2}{N-2-m}}-\hat{\sigma},k^{\frac{N-2}{N-2-m}}+\hat{\sigma}\right], h\in \left[\frac{B'}{k^{\frac{N-3}{N-1}}}\left(1-\hat{\sigma}\right),
\frac{B'}{k^{\frac{N-3}{N-1}}}\left(1+\hat{\sigma}\right)\right],
\nonumber\\
&\Lambda\in [\Lambda_0-\hat{\sigma},\Lambda_0+\hat{\sigma}]
\Bigg\}.
\end{align}
Here $\Lambda_0$, $B'$ are the constants in (\ref{lamda0}), (\ref{hb}) and
$\hat{\sigma}$ is a fixed small number  independent of $k$.

We will prove Theorem \ref{thm1.10} by proving the following result:
\begin{theorem}\label{thm1.1}
Let $N\geq 5$ and $p$, $q$, $m$ satisfy (\ref{pq}), (\ref{p+}), (\ref{q+}), (\ref{m0}) and (\ref{m+}) respectively. Suppose that $K_i(|y|)$ satisfies $(K_i)$ for $i=1,2$. Then there exists a large integer $k_0$, such that for each integer $k\geq k_0$, problem (\ref{lemu}) has a solution $(u_k,v_k)$ of the form
\begin{align*}
(u_k,v_k)=\left(W_{1,r_k,h_k,\Lambda_k}(y)+\phi_1,W_{2,r_k,h_k,\Lambda_k}(y)+\phi_2\right),
\end{align*}
where $(\phi_1,\phi_2)\in H_s$ and $(r_k,h_k,\Lambda_k)\in \wp_k$,
\begin{align*}
\|(\phi_1,\phi_2)\|_{*}=o(1) \ \text{ as } k\rightarrow \infty.
\end{align*}
\end{theorem}

We will use the techniques in the singularly perturbed elliptic problems to prove Theorem \ref{thm1.1}. Inspired by \cite{weiy}, we use $k$, the number of the bumps of the solutions, as the parameter in the construction of spike solutions for (\ref{lemu}).

The paper is organized as follows. In Section \ref{Finite-dimensional reduction}, we will carry out the reduction.
Section \ref{existence} is devoted to prove Theorem \ref{thm1.1}, we obtain the existence of solution to problem \eqref{lemu} by reduction method. In Section \ref{appendix}, we will give some useful lemmas, propositions and the details for the energy of approximate solution expansion.

\section{Finite-dimensional reduction}\label{Finite-dimensional reduction}

According to the decay properties of the asymptotic solutions in (\ref{limiteq}), we define the following norms for $p>\frac{N}{N-2}$,
\begin{align}\label{norm1}
\|u\|_{*}=\sup_{y\in \mathbb{R}^N}\left(\sum_{j=1}^k\left[\frac{1}{(1+|y-\overline{x}_j|)^{\frac{N-2}{2}+\tau}}
+\frac{1}{(1+|y-\underline{x}_j|)^{\frac{N-2}{2}+\tau}}
\right]\right)^{-1}|u(y)|,
\end{align}
and
\begin{align}\label{norm2}
\|f\|_{**}=\sup_{y\in \mathbb{R}^N}\left(\sum_{j=1}^k\left[\frac{1}{(1+|y-\overline{x}_j|)^{\frac{N+2}{2}+\tau}}
+\frac{1}{(1+|y-\underline{x}_j|)^{\frac{N+2}{2}+\tau}}
\right]\right)^{-1}|f(y)|,
\end{align}
where
\begin{align}\label{tau}
\tau\geq \frac{N-2-m}{N-2}+\frac{2}{p-1}-\frac{N-2}{2}.
\end{align}
Denote also $\|(u,v)\|_{*}=\|u\|_{*}+\|v\|_{*}$ and $\|(u,v)\|_{**}=\|u\|_{**}+\|v\|_{**}$.

For $j=1,\cdots, k$, we divide $\mathbb{R}^N$ into $k$ parts:
\begin{align*}
\Omega_j:=\Bigg\{&y=(y_1,y_2,y_3,y'')\in \mathbb{R}^3\times \mathbb{R}^{N-3}:
\nonumber\\
&\qquad \left\langle \frac{(y_1,y_2)}{|(y_1,y_2)|},\left(\cos\frac{2(j-1)\pi}{k},\sin\frac{2(j-1)\pi}{k}\right)
\right\rangle_{\mathbb{R}^2}
\geq \cos \frac{\pi}{k}
\Bigg\},
\end{align*}
where $\langle,\rangle_{\mathbb{R}^2}$ denote the dot product in ${\mathbb{R}^2}$. Set
\begin{align*}
\Omega_j^+:=\{y: y=(y_1,y_2,y_3,y'')\in \Omega_j, y_3\geq 0\},
\end{align*}
and
\begin{align*}
\Omega_j^-:=\{y: y=(y_1,y_2,y_3,y'')\in \Omega_j, y_3\leq 0\}.
\end{align*}
Then
\begin{align*}
\mathbb{R}^N=\cup_{j=1}^k\Omega_j,\quad
\Omega_j=\Omega_j^+\cup \Omega_j^-
\end{align*}
and for $i\neq j$,
\begin{align*}
\Omega_j\cap \Omega_i=\emptyset,\quad
\Omega_j^+\cap \Omega_j^-=\emptyset.
\end{align*}
Define
\begin{align*}
\overline{Z}_{1j}=\frac{\partial U_{\overline{x}_j,\Lambda}}{\partial r},\quad
\ \ \ \overline{Z}_{2j}=\frac{\partial U_{\overline{x}_j,\Lambda}}{\partial h},\quad
\overline{Z}_{3j}=\frac{\partial U_{\overline{x}_j,\Lambda}}{\partial \Lambda},
\end{align*}

\begin{align*}
\underline{Z}_{1j}=\frac{\partial U_{\underline{x}_j,\Lambda}}{\partial r},\quad
\underline{Z}_{2j}=\frac{\partial U_{\underline{x}_j,\Lambda}}{\partial h},\quad
\underline{Z}_{3j}=\frac{\partial U_{\underline{x}_j,\Lambda}}{\partial \Lambda},
\end{align*}

\begin{align*}
\overline{Y}_{1j}=\frac{\partial V_{\overline{x}_j,\Lambda}}{\partial r},\quad
\overline{Y}_{2j}=\frac{\partial V_{\overline{x}_j,\Lambda}}{\partial h},\quad
\overline{Y}_{3j}=\frac{\partial V_{\overline{x}_j,\Lambda}}{\partial \Lambda},
\end{align*}

\begin{align*}
\underline{Y}_{1j}=\frac{\partial V_{\underline{x}_j,\Lambda}}{\partial r},\quad
\underline{Y}_{2j}=\frac{\partial V_{\underline{x}_j,\Lambda}}{\partial h},\quad
\underline{Y}_{3j}=\frac{\partial V_{\underline{x}_j,\Lambda}}{\partial \Lambda},
\end{align*}
for $j=1, \cdots, k$.

We first consider the following linear problem
\begin{equation}\label{linearp}
\begin{cases}
-\Delta\phi_1-pK_1\left(\frac{|y|}{\mu}\right)W_2^{p-1}\phi_2=f_{1}
+\sum_{l=1}^3\sum_{j=1}^kc_l\left(V_{\overline{x}_j,\Lambda}^{p-1}\overline{Y}_{lj}
+V_{\underline{x}_j,\Lambda}^{p-1}\underline{Y}_{lj}
\right),\  \text{ in } \mathbb{R}^N, \\
-\Delta\phi_2-qK_2\left(\frac{|y|}{\mu}\right)W_1^{q-1}\phi_1=f_{2}
+\sum_{l=1}^3\sum_{j=1}^kc_l\left(U_{\overline{x}_j,\Lambda}^{q-1}\overline{Z}_{lj}
+U_{\underline{x}_j,\Lambda}^{q-1}\underline{Z}_{lj}
\right),\  \text{ in } \mathbb{R}^N, \\
(\phi_1, \phi_2)\in \mathbb{E},
\end{cases}
\end{equation}
for some constants $c_l$ for $l=1, 2, 3$ and $\mathbb{E}$ is defined by
\begin{align}\label{E}
\mathbb{E}:=\Bigg\{(\phi_1, \phi_2)\in H_s,
&\int_{\mathbb{R}^N}V_{\overline{x}_j,\Lambda}^{p-1}\overline{Y}_{lj}\phi_1=0 \text{ and }
\int_{\mathbb{R}^N}V_{\underline{x}_j,\Lambda}^{p-1}\underline{Y}_{lj}\phi_1=0,
\nonumber\\
&\int_{\mathbb{R}^N}U_{\overline{x}_j,\Lambda}^{q-1}\overline{Z}_{lj}\phi_2=0
\text{ and }
\int_{\mathbb{R}^N}U_{\underline{x}_j,\Lambda}^{q-1}\underline{Z}_{lj}\phi_2=0,
j=1, \cdots, k, \ l=1, 2, 3
\Bigg\}.
\end{align}
Define
\begin{align}\label{lk}
L_k(\phi_1, \phi_2)=\left(-\Delta \phi_1-pK_1\left(\frac{|y|}{\mu}\right)W_2^{p-1}\phi_2,
-\Delta \phi_2-qK_2\left(\frac{|y|}{\mu}\right)W_1^{q-1}\phi_1\right).
\end{align}

\begin{lemma}\label{lpL}
Assume that $(\phi_{1,k}, \phi_{2,k})$ solves (\ref{linearp}) for $(f_1,f_2)=(f_{1,k},f_{2,k})$.
If $\|(f_{1,k},f_{2,k})\|_{**}$ goes to zero as $k$ goes to infinity, so does $\|(\phi_{1,k}, \phi_{2,k})\|_{*}$.
\end{lemma}
\begin{proof}
Assume by contradiction, suppose that there exists a sequence $(r_k,h_k,\Lambda_k)\in \wp_k$ and $(\phi_{1,k}, \phi_{2,k})$ satisfies (\ref{linearp}) with $(f_1,f_2)=(f_{1,k},f_{2,k})$, $r=r_k$, $h=h_k$, $\Lambda=\Lambda_k$. Without loss of generality, we assume that $\|(f_{1,k},f_{2,k})\|_{**}\rightarrow 0$ and $\|(\phi_{1,k}, \phi_{2,k})\|_{*}=1$. For convenience, we drop the subscript $k$.

From (\ref{linearp}), we have
\begin{align}\label{lpLphi1}
\phi_1(y)=\int_{\mathbb{R}^N}\frac{1}{|y-z|^{N-2}}\left(pK_1\left(\frac{|z|}{\mu}\right)W_2^{p-1}\phi_2(z)+f_{1}(z)
+\sum_{l=1}^3\sum_{j=1}^kc_l\left(V_{\overline{x}_j,\Lambda}^{p-1}\overline{Y}_{lj}
+V_{\underline{x}_j,\Lambda}^{p-1}\underline{Y}_{lj}
\right)\right)dz,
\end{align}
\begin{align}\label{lpLphi2}
\phi_2(y)=\int_{\mathbb{R}^N}\frac{1}{|y-z|^{N-2}}\left(qK_2\left(\frac{|z|}{\mu}\right)W_1^{q-1}\phi_1(z)+f_{2}(z)
+\sum_{l=1}^3\sum_{j=1}^kc_l\left(U_{\overline{x}_j,\Lambda}^{q-1}\overline{Z}_{lj}
+U_{\underline{x}_j,\Lambda}^{q-1}\underline{Z}_{lj}
\right)\right)dz.
\end{align}
We mainly deal with $\phi_1$, since $\phi_2$ can be estimated in the same way. It holds
\begin{align}\label{lpLphi11}
\phi_1(y)=&p\int_{\mathbb{R}^N}\frac{K_1\left(\frac{|z|}{\mu}\right)W_2^{p-1}\phi_2(z)}{|y-z|^{N-2}}dz
+\int_{\mathbb{R}^N}\frac{f_{1}(z)}{|y-z|^{N-2}}dz
\nonumber\\
&+\int_{\mathbb{R}^N}\frac{1}{|y-z|^{N-2}}\sum_{l=1}^3\sum_{j=1}^kc_l\left(V_{\overline{x}_j,\Lambda}^{p-1}\overline{Y}_{lj}
+V_{\underline{x}_j,\Lambda}^{p-1}\underline{Y}_{lj}
\right)dz
:=M_1+M_2+M_3.
\end{align}
We first estimate $M_1$. By \cite[Lemma B.5]{duanmw}, there exists some $\theta>0$ such that
\begin{align*}
M_1\leq &C\|\phi_2\|_{*}\int_{\mathbb{R}^N}\frac{K_1\left(\frac{|z|}{\mu}\right)}{|y-z|^{N-2}}W_2^{p-1}
\left(\sum_{j=1}^k\left[\frac{1}{(1+|y-\overline{x}_j|)^{\frac{N-2}{2}+\tau}}
+\frac{1}{(1+|y-\underline{x}_j|)^{\frac{N-2}{2}+\tau}}
\right]\right)dz
\nonumber\\
\leq &C\|\phi_2\|_{*}\int_{\mathbb{R}^N}\frac{1}{|y-z|^{N-2}}
\left(\sum_{j=1}^k\left[\frac{1}{(1+|y-\overline{x}_j|)^{N-2}}
+\frac{1}{(1+|y-\underline{x}_j|)^{N-2}}
\right]\right)^{p-1}
\nonumber\\
&\times\left(\sum_{j=1}^k\left[\frac{1}{(1+|y-\overline{x}_j|)^{\frac{N-2}{2}+\tau}}
+\frac{1}{(1+|y-\underline{x}_j|)^{\frac{N-2}{2}+\tau}}
\right]\right)dz
\nonumber\\
\leq &C\|\phi_2\|_{*}\sum_{j=1}^k\left[\frac{1}{(1+|y-\overline{x}_j|)^{\frac{N-2}{2}+\tau+\theta}}
+\frac{1}{(1+|y-\underline{x}_j|)^{\frac{N-2}{2}+\tau+\theta}}
\right],
\end{align*}
where $p>\max\left\{\frac{N+1}{N-2},\frac{N(N-2)}{(N-2)^2-(N-2-m)}\right\}$, this condition is similar to \cite{guolp}.

Moreover,
\begin{align*}
M_2=&\int_{\mathbb{R}^N}\frac{f_{1}(z)}{|y-z|^{N-2}}dz
\nonumber\\
\leq &C
\|f_{1}(z)\|_{**}\int_{\mathbb{R}^N}\frac{1}{|y-z|^{N-2}}
\left(\sum_{j=1}^k\left[\frac{1}{(1+|y-\overline{x}_j|)^{\frac{N+2}{2}+\tau}}
+\frac{1}{(1+|y-\underline{x}_j|)^{\frac{N+2}{2}+\tau}}
\right]\right)dz
\nonumber\\
\leq &C
\|f_{1}(z)\|_{**}
\sum_{j=1}^k\left[\frac{1}{(1+|y-\overline{x}_j|)^{\frac{N-2}{2}+\tau}}
+\frac{1}{(1+|y-\underline{x}_j|)^{\frac{N-2}{2}+\tau}}
\right]
\end{align*}
and since
\begin{align*}
|\overline{Z}_{1j}|\leq \frac{C}{(1+|y-\overline{x}_j|)^{N-1}},\
|\overline{Z}_{2j}|\leq \frac{Cr}{(1+|y-\overline{x}_j|)^{N-1}},\
|\overline{Z}_{3j}|\leq \frac{C}{(1+|y-\overline{x}_j|)^{N-2}},\
\end{align*}
\begin{align*}
|\overline{Y}_{1j}|\leq \frac{C}{(1+|y-\underline{x}_j|)^{N-1}},\
|\overline{Y}_{2j}|\leq \frac{Cr}{(1+|y-\underline{x}_j|)^{N-1}},\
|\overline{Y}_{3j}|\leq \frac{C}{(1+|y-\underline{x}_j|)^{N-2}},\
\end{align*}
we obtain that for $l=1,2,3$,
\begin{align*}
\left|\int_{\mathbb{R}^N}\frac{1}{|y-z|^{N-2}}
V_{\overline{x}_j,\Lambda}^{p-1}\overline{Y}_{lj}dz\right|
\leq &C
\sum_{j=1}^k\int_{\mathbb{R}^N}\frac{1}{|y-z|^{N-2}}
\frac{1+r\delta_{l2}}{(1+|z-\overline{x}_j|)^{p(N-2)}}dz
\nonumber\\
\leq &C
\sum_{j=1}^k\frac{1+r\delta_{l2}}{(1+|y-\overline{x}_j|)^{\frac{N-2}{2}+\tau}}.
\end{align*}
Similarly, we have
\begin{align*}
&\left|\int_{\mathbb{R}^N}\frac{1}{|y-z|^{N-2}}
V_{\underline{x}_j,\Lambda}^{p-1}\underline{Y}_{lj}dz\right|
\leq C
\sum_{j=1}^k\frac{1+r\delta_{l2}}{(1+|y-\underline{x}_j|)^{\frac{N-2}{2}+\tau}}.
\end{align*}

Next, we estimate $c_i$ for $i=1,2,3$. Multiplying the two equation in (\ref{linearp}) by $\overline{Y}_{l1}$ and $\overline{Z}_{l1}$ respectively. Then
\begin{align*}
&\sum_{i=1}^3\sum_{j=1}^kc_i\langle (V_{\overline{x}_j,\Lambda}^{p-1}\overline{Y}_{ij}
+V_{\underline{x}_j,\Lambda}^{p-1}\underline{Y}_{ij}, U_{\overline{x}_j,\Lambda}^{q-1}\overline{Z}_{ij}
+U_{\underline{x}_j,\Lambda}^{q-1}\underline{Z}_{ij} ),(\overline{Y}_{l1}, \overline{Z}_{l1}) \rangle
\nonumber\\
=&\langle L_k(\phi_1, \phi_2), (\overline{Y}_{l1}, \overline{Z}_{l1})\rangle
-\langle (f_1, f_2),(\overline{Y}_{l1}, \overline{Z}_{l1}) \rangle.
\end{align*}
For $l=1,2,3$, by \cite[Lemma B.3]{duanmw},
\begin{align*}
&\left|\langle (f_1, f_2),(\overline{Y}_{l1}, \overline{Z}_{l1}) \rangle\right|
\nonumber\\
\leq
&C\|(f_1, f_2)\|_{**}\int_{\mathbb{R}^N}\frac{1+r\delta_{l2}}{(1+|y-\overline{x}_1|)^{N-2}}
\sum_{j=1}^k\left[\frac{1}{(1+|y-\overline{x}_j|)^{\frac{N+2}{2}+\tau}}
+\frac{1}{(1+|y-\underline{x}_j|)^{\frac{N+2}{2}+\tau}}
\right]
\nonumber\\
\leq
&C(1+r\delta_{l2})\|(f_1, f_2)\|_{**}.
\end{align*}
Because of
\begin{align*}
\langle L_k(\phi_1, \phi_2), (\overline{Y}_{l1}, \overline{Z}_{l1})\rangle
=&\langle L_k(\overline{Y}_{l1}, \overline{Z}_{l1}), (\phi_1, \phi_2)\rangle
\nonumber\\
=&\int_{\mathbb{R}^N}p\left(1-K_1\left(\frac{|y|}{\mu}\right)\right)W_2^{p-1}\phi_2\overline{Y}_{l1}
+p\left(V_{\overline{x}_1,\Lambda}^{p-1}-W_2^{p-1}\right)\phi_2\overline{Y}_{l1}
\nonumber\\
&+\int_{\mathbb{R}^N}q\left(1-K_2\left(\frac{|y|}{\mu}\right)\right)W_1^{q-1}\phi_1\overline{Z}_{l1}
+q\left(U_{\overline{x}_1,\Lambda}^{q-1}-W_1^{q-1}\right)\phi_1\overline{Z}_{l1}
\nonumber\\
:=&J_1+J_2.
\end{align*}
We mainly estimate the term $J_2$, since $J_1$ can be got in the same way.
In fact
\begin{align*}
J_2=\int_{\mathbb{R}^N}q\left(1-K_2\left(\frac{|y|}{\mu}\right)\right)W_1^{q-1}\phi_1\overline{Z}_{l1}
+q\left(U_{\overline{x}_1,\Lambda}^{q-1}-W_1^{q-1}\right)\phi_1\overline{Z}_{l1}
:=J_{21}+J_{22}.
\end{align*}
For the term $J_{21}$, we have
\begin{align*}
J_{21}\leq &C\|\phi_1\|_{*}
\int_{\mathbb{R}^N}\left|1-K_2\left(\frac{|y|}{\mu}\right)\right|W_1^{q-1}\overline{Z}_{l1}
\sum_{j=1}^k\left[\frac{1}{(1+|y-\overline{x}_j|)^{\frac{N-2}{2}+\tau}}
+\frac{1}{(1+|y-\underline{x}_j|)^{\frac{N-2}{2}+\tau}}
\right]
\nonumber\\
=&C\|\phi_1\|_{*}\left\{\int_{||y|-\mu r_0|\leq \sqrt{\mu}}+\int_{||y|-\mu r_0|\geq \sqrt{\mu}}\right\}
\left|1-K_2\left(\frac{|y|}{\mu}\right)\right|W_1^{q-1}\overline{Z}_{l1}
\nonumber\\
&\times \sum_{j=1}^k\left[\frac{1}{(1+|y-\overline{x}_j|)^{\frac{N-2}{2}+\tau}}
+\frac{1}{(1+|y-\underline{x}_j|)^{\frac{N-2}{2}+\tau}}
\right],
\end{align*}
where
\begin{align*}
&\int_{||y|-\mu r_0|\leq \sqrt{\mu}}\left|1-K_2\left(\frac{|y|}{\mu}\right)\right|W_1^{q-1}\overline{Z}_{l1}
\sum_{j=1}^k\left[\frac{1}{(1+|y-\overline{x}_j|)^{\frac{N-2}{2}+\tau}}
+\frac{1}{(1+|y-\underline{x}_j|)^{\frac{N-2}{2}+\tau}}
\right]
\nonumber\\
\leq &\frac{C}{\sqrt{\mu}}\int_{\mathbb{R}^N}W_1^{q-1}\overline{Z}_{l1}
\sum_{j=1}^k\left[\frac{1}{(1+|y-\overline{x}_j|)^{\frac{N-2}{2}+\tau}}
+\frac{1}{(1+|y-\underline{x}_j|)^{\frac{N-2}{2}+\tau}}
\right]
\nonumber\\
\leq &\frac{C}{\sqrt{\mu}}\int_{\mathbb{R}^N}W_1^{q-1}\frac{1+r\delta_{l2}}{(1+|y-\overline{x}_1|)^{N-2}}
\sum_{j=1}^k\left[\frac{1}{(1+|y-\overline{x}_j|)^{\frac{N-2}{2}+\tau}}
+\frac{1}{(1+|y-\underline{x}_j|)^{\frac{N-2}{2}+\tau}}
\right]
\nonumber\\
\leq &\frac{C(1+r\delta_{l2})}{\sqrt{\mu}}.
\end{align*}
When $||y|-\mu r_0|\geq \sqrt{\mu}$,
\begin{align*}
|y-\overline{x}_1|
\geq ||y|-\mu r_0|-||\overline{x}_1|-\mu r_0|
\geq \sqrt{\mu}-\frac{1}{\mu^{\overline{\theta}}}
\geq \frac{1}{2}\sqrt{\mu},
\end{align*}
then
\begin{align*}
&\int_{||y|-\mu r_0|\geq \sqrt{\mu}}\left|1-K_2\left(\frac{|y|}{\mu}\right)\right|W_1^{q-1}\overline{Z}_{l1}
\sum_{j=1}^k\left[\frac{1}{(1+|y-\overline{x}_j|)^{\frac{N-2}{2}+\tau}}
+\frac{1}{(1+|y-\underline{x}_j|)^{\frac{N-2}{2}+\tau}}
\right]
\nonumber\\
\leq &\frac{C}{{\mu}^{\sigma}}\int_{\mathbb{R}^N}W_1^{q-1}\frac{1+r\delta_{l2}}{(1+|y-\overline{x}_1|)^{N-2}}
\sum_{j=1}^k\left[\frac{1}{(1+|y-\overline{x}_j|)^{\frac{N-2}{2}+\tau-2\sigma}}
+\frac{1}{(1+|y-\underline{x}_j|)^{\frac{N-2}{2}+\tau-2\sigma}}
\right]
\nonumber\\
\leq &\frac{C(1+r\delta_{l2})}{{\mu}^{\sigma}}.
\end{align*}
Therefore,
\begin{align*}
J_{21}=O\left(\frac{1+r\delta_{l2}}{{\mu}^{\sigma}} \|\phi_1\|_{*}\right) \ \text{ for some } \sigma>0.
\end{align*}
As to $J_{22}$. For $q-1\leq 1$,
\begin{align*}
\left|U_{\overline{x}_1,\Lambda}^{q-1}-W_1^{q-1}\right|
\leq \left(\sum_{j=2}^kU_{\overline{x}_j,\Lambda}+\sum_{j=1}^kU_{\underline{x}_j,\Lambda}\right)^{q-1},
\end{align*}
then we have
\begin{align*}
|J_{22}|\leq &C\|\phi_1\|_{*}
\int_{\mathbb{R}^N}\left(\sum_{j=2}^kU_{\overline{x}_j,\Lambda}+\sum_{j=1}^kU_{\underline{x}_j,\Lambda}\right)^{q-1}
\frac{1+r\delta_{l2}}{(1+|y-\overline{x}_1|)^{N-2}}
\nonumber\\
&\times \sum_{j=1}^k\left[\frac{1}{(1+|y-\overline{x}_j|)^{\frac{N-2}{2}+\tau}}
+\frac{1}{(1+|y-\underline{x}_j|)^{\frac{N-2}{2}+\tau}}
\right]
\nonumber\\
\leq &C\|\phi_1\|_{*}\Bigg\{\int_{\mathbb{R}^N}\sum_{j=2}^kU_{\overline{x}_j,\Lambda}^{q-1}
\frac{1+r\delta_{l2}}{(1+|y-\overline{x}_1|)^{N-2+\frac{N-2}{2}+\tau}}
\nonumber\\
&+\int_{\mathbb{R}^N}\sum_{j=2}^kU_{\overline{x}_j,\Lambda}^{q-1}
\frac{1+r\delta_{l2}}{(1+|y-\overline{x}_1|)^{N-2}}\sum_{j=1}^k\frac{1}{(1+|y-\underline{x}_j|)^{\frac{N-2}{2}+\tau}}
\nonumber\\
&+\int_{\mathbb{R}^N}\left(\sum_{j=2}^k\frac{1}{(1+|y-\overline{x}_j|)^{\frac{N-2}{2}+\tau}}\right)^q
\frac{1+r\delta_{l2}}{(1+|y-\overline{x}_1|)^{N-2}}
\nonumber\\
&+\int_{\mathbb{R}^N}\left(\sum_{j=1}^kU_{\underline{x}_j,\Lambda}\right)^{q-1}
\frac{1+r\delta_{l2}}{(1+|y-\overline{x}_1|)^{N-2}}
\sum_{j=1}^k\left[\frac{1}{(1+|y-\overline{x}_j|)^{\frac{N-2}{2}+\tau}}
+\frac{1}{(1+|y-\underline{x}_j|)^{\frac{N-2}{2}+\tau}}
\right]
\Bigg\}.
\end{align*}
Since $p,q>\frac{N+1}{N-2}$, we have $q(N-2)+\frac{N-2}{2}+\tau>N+1+\frac{N-2}{2}+\tau$. Then there exists $\tilde{\theta}>0$ such that for $\frac{N-2-m}{N-2}\leq \alpha\leq q(N-2)-\frac{N+2}{2}+\tau-\tilde{\theta}$, we have
\begin{align*}
q(N-2)+\frac{N-2}{2}+\tau-\alpha\geq N+\tilde{\theta}.
\end{align*}
Thus,
\begin{align*}
&\int_{\mathbb{R}^N}\sum_{j=2}^kU_{\overline{x}_j,\Lambda}^{q-1}
\frac{1+r\delta_{l2}}{(1+|y-\overline{x}_1|)^{N-2+\frac{N-2}{2}+\tau}}
\nonumber\\
\leq &C\int_{\mathbb{R}^N}\sum_{j=2}^k\frac{1}{(1+|y-\overline{x}_j|)^{(N-2)(q-1)}}
\frac{1+r\delta_{l2}}{(1+|y-\overline{x}_1|)^{N-2+\frac{N-2}{2}+\tau}}
\nonumber\\
\leq &C\sum_{j=2}^k\frac{1+r\delta_{l2}}{|\overline{x}_1-\overline{x}_j|^{\alpha}}
\Bigg\{\int_{\mathbb{R}^N}\frac{1}{(1+|y-\overline{x}_1|)^{(N-2)q+\frac{N-2}{2}+\tau-\alpha}}
+\int_{\mathbb{R}^N}\frac{1}{(1+|y-\overline{x}_j|)^{(N-2)q+\frac{N-2}{2}+\tau-\alpha}}
\Bigg\}
\nonumber\\
\leq &C\sum_{j=2}^k\frac{1+r\delta_{l2}}{|\overline{x}_1-\overline{x}_j|^{\alpha}}
\int_{\mathbb{R}^N}\left(\frac{1}{(1+|y-\overline{x}_1|)^{N+\tilde{\theta}}}
+\frac{1}{(1+|y-\overline{x}_j|)^{N+\tilde{\theta}}}\right).
\end{align*}
Moreover
\begin{align*}
&\int_{\mathbb{R}^N}\sum_{j=2}^kU_{\overline{x}_j,\Lambda}^{q-1}
\frac{1+r\delta_{l2}}{(1+|y-\overline{x}_1|)^{N-2}}\sum_{j=1}^k\frac{1}{(1+|y-\underline{x}_j|)^{\frac{N-2}{2}+\tau}}
\nonumber\\
\leq &C\int_{\mathbb{R}^N}\sum_{j=2}^k\left(\frac{1}{(1+|y-\overline{x}_j|)}\right)^{(N-2)(q-1)}
\frac{1+r\delta_{l2}}{(1+|y-\overline{x}_1|)^{N-2}}\sum_{j=1}^k\frac{1}{(1+|y-\underline{x}_j|)^{\frac{N-2}{2}+\tau}}
\nonumber\\
\leq &C\int_{\mathbb{R}^N}\sum_{j=2}^k\left(\frac{1}{(1+|y-\overline{x}_j|)}\right)^{(N-2)(q-1)}
\sum_{j=1}^k\frac{(1+r\delta_{l2})}{|\overline{x}_1-\underline{x}_j|^{\alpha}}
\nonumber\\
&\times\Bigg\{\frac{1}{(1+|y-\overline{x}_1|)^{N-2+\frac{N-2}{2}+\tau-\alpha}}+
\frac{1}{(1+|y-\underline{x}_j|)^{N-2+\frac{N-2}{2}+\tau-\alpha}}\Bigg\}
\nonumber\\
\leq &C\int_{\mathbb{R}^N}\sum_{j=2}^k\frac{1}{|\overline{x}_1-\overline{x}_j|^{\beta}}
\Bigg\{\frac{1}{(1+|y-\overline{x}_1|)^{q(N-2)+\frac{N-2}{2}+\tau-\alpha-\beta}}+
\frac{1}{(1+|y-\overline{x}_j|)^{q(N-2)+\frac{N-2}{2}+\tau-\alpha-\beta}}
\Bigg\}
\nonumber\\
&\times \sum_{i=1}^k\frac{(1+r\delta_{l2})}{|\overline{x}_1-\underline{x}_i|^{\alpha}}
+C\int_{\mathbb{R}^N}\sum_{j=2}^k\frac{1}{|\overline{x}_j-\underline{x}_i|^{\beta}}
\Bigg\{\frac{1}{(1+|y-\overline{x}_j|)^{q(N-2)+\frac{N-2}{2}+\tau-\alpha-\beta}}
\nonumber\\
&+
\frac{1}{(1+|y-\underline{x}_i|)^{q(N-2)+\frac{N-2}{2}+\tau-\alpha-\beta}}
\Bigg\}
\sum_{i=1}^k\frac{(1+r\delta_{l2})}{|\overline{x}_1-\underline{x}_i|^{\alpha}}
\nonumber\\
=&O\left(\frac{1+r\delta_{l2}}{\mu^{\sigma}}\right),
\end{align*}
where $\alpha$, $\beta$ satisfy $0<\alpha<\frac{N-2}{2}+\tau$, $0<\beta<\min\{(N-2)(q-1), N-2+\frac{N-2}{2}+\tau-\alpha\}$ and
\begin{align*}
q(N-2)+\frac{N-2}{2}+\tau-\alpha-\beta>N.
\end{align*}
Similarly, for $\gamma\in [\frac{N-2-m}{N-2}, q(\frac{N-2}{2}+\tau)-2]$, we get
\begin{align*}
&\int_{\mathbb{R}^N}\left(\sum_{j=2}^k\frac{1}{(1+|y-\overline{x}_j|)^{\frac{N-2}{2}+\tau}}\right)^q
\frac{1+r\delta_{l2}}{(1+|y-\overline{x}_1|)^{N-2}}
\nonumber\\
\leq &C\int_{\mathbb{R}^N}\sum_{j=2}^k\frac{1+r\delta_{l2}}{|\overline{x}_1-\overline{x}_j|^{\gamma}}
\Bigg\{\frac{1}{(1+|y-\overline{x}_1|)^{N-2+q(\frac{N-2}{2}+\tau)-\gamma}}
+\frac{1}{(1+|y-\overline{x}_j|)^{N-2+q(\frac{N-2}{2}+\tau)-\gamma}}\Bigg\}
\nonumber\\
=&O\left(\frac{1+r\delta_{l2}}{\mu^{\sigma}}\right),
\end{align*}
and by the same way, we have
\begin{align*}
&\int_{\mathbb{R}^N}\left(\sum_{j=1}^kU_{\underline{x}_j,\Lambda}\right)^{q-1}
\frac{1+r\delta_{l2}}{(1+|y-\overline{x}_1|)^{N-2}}
\sum_{j=1}^k\left[\frac{1}{(1+|y-\overline{x}_j|)^{\frac{N-2}{2}+\tau}}
+\frac{1}{(1+|y-\underline{x}_j|)^{\frac{N-2}{2}+\tau}}
\right]
=O\left(\frac{1+r\delta_{l2}}{\mu^{\sigma}}\right).
\end{align*}
Above all,
\begin{align*}
|J_{22}|=O\left(\frac{1+r\delta_{l2}}{\mu^{\sigma}}\|\phi_1\|_{*}\right) \ \text{ for } q\leq 2.
\end{align*}
For $q\geq 2$, we have
\begin{align*}
\left|U_{\overline{x}_1,\Lambda}^{q-1}-W_1^{q-1}\right|
\leq & U_{\overline{x}_1,\Lambda}^{q-2}\left(\sum_{j=2}^kU_{\overline{x}_j,\Lambda}+\sum_{j=1}^kU_{\underline{x}_j,\Lambda}\right)
+U_{\overline{x}_1,\Lambda}\left(\sum_{j=2}^kU_{\overline{x}_j,\Lambda}+\sum_{j=1}^kU_{\underline{x}_j,\Lambda}\right)^{q-2}
\nonumber\\
&+\left(\sum_{j=2}^kU_{\overline{x}_j,\Lambda}+\sum_{j=1}^kU_{\underline{x}_j,\Lambda}\right)^{q-1}.
\end{align*}
Similar to the case $q\leq 2$, we obtain that
\begin{align*}
|J_{22}|=O\left(\frac{1+r\delta_{l2}}{\mu^{\sigma}}\|\phi_1\|_{*}\right) \ \text{ for } q-1\geq 1.
\end{align*}
Therefore,
\begin{align*}
|J_{22}|=O\left(\frac{1+r\delta_{l2}}{\mu^{\sigma}}\|\phi_1\|_{*}\right).
\end{align*}
Moreover,
\begin{align*}
|J_{2}|=O\left(\frac{1+r\delta_{l2}}{\mu^{\sigma}}\|\phi_1\|_{*}\right).
\end{align*}
Similarly,
\begin{align*}
|J_{1}|=O\left(\frac{1+r\delta_{l2}}{\mu^{\sigma}}\|\phi_2\|_{*}\right).
\end{align*}
In conclusion,
\begin{align*}
\left|\langle L_k(\phi_1, \phi_2), (\overline{Y}_{l1}, \overline{Z}_{l1})\rangle\right|
=O\left(\frac{1+r\delta_{l2}}{\mu^{\sigma}}\|(\phi_1, \phi_2)\|_{*}\right).
\end{align*}
On the other hand, since
\begin{align*}
\int_{\mathbb{R}^N}U_{\overline{x}_1,\Lambda}^{q-1}\overline{Z}_{l1}\overline{Z}_{i1}=
\begin{cases}
0, & \text{if }l\neq i,\\
\bar{c}_l(1+\delta_{l2}r^2), & \text{if }l=i,
\end{cases}
\end{align*}
and
\begin{align*}
\int_{\mathbb{R}^N}U_{\underline{x}_1,\Lambda}^{q-1}\underline{Z}_{l1}\underline{Z}_{i1}=
\begin{cases}
0, & \text{if }l\neq i,\\
\underline{c}_l(1+\delta_{l2}r^2), & \text{if }l=i,
\end{cases}
\end{align*}
for some constant $\overline{c}_l>0$  and  $\underline{c}_l>0$.
We get
\begin{align*}
&\sum_{j=1}^k\langle (V_{\overline{x}_j,\Lambda}^{p-1}\overline{Y}_{ij}
+V_{\underline{x}_j,\Lambda}^{p-1}\underline{Y}_{ij}, U_{\overline{x}_j,\Lambda}^{q-1}\overline{Z}_{ij}
+U_{\underline{x}_j,\Lambda}^{q-1}\underline{Z}_{ij} ),(\overline{Y}_{l1}, \overline{Z}_{l1}) \rangle
\nonumber\\
=&\left(\bar{c}_l(1+\delta_{l2}r^2)+o(1),\underline{c}_l(1+\delta_{l2}r^2)+o(1)\right)\  \text{as } k\rightarrow \infty.
\end{align*}
Then for $l=1,2,3$,
\begin{align*}
c_l=\frac{1+\delta_{l2}r}{1+\delta_{l2}r^2}O\left(\frac{1}{\mu^{\sigma}}\|(\phi_1, \phi_2)\|_{**}
+\|(f_1, f_2)\|_{*}
\right)
=o(1)\  \text{as } k\rightarrow \infty.
\end{align*}
Therefore,
\begin{align*}
M_3=O\left(\sum_{j=1}^k\left[\frac{1}{(1+|y-\overline{x}_j|)^{\frac{N-2}{2}+\tau}}
+\frac{1}{(1+|y-\underline{x}_j|)^{\frac{N-2}{2}+\tau}}
\right]\left(\frac{1}{\mu^{\sigma}}\|(\phi_1, \phi_2)\|_{*}+\|(f_1, f_2)\|_{**}\right)\right).
\end{align*}
Then by (\ref{lpLphi1}) and (\ref{lpLphi2}),
\begin{align*}
\|(\phi_1, \phi_2)\|_{*}\leq &
\Bigg\{\|(f_1, f_2)\|_{**}\sum_{j=1}^k\left[\frac{1}{(1+|y-\overline{x}_j|)^{\frac{N-2}{2}+\tau}}
+\frac{1}{(1+|y-\underline{x}_j|)^{\frac{N-2}{2}+\tau}}
\right]
\nonumber\\
&+\sum_{j=1}^k\left[\frac{1}{(1+|y-\overline{x}_j|)^{\frac{N-2}{2}+\tau+\theta}}
+\frac{1}{(1+|y-\underline{x}_j|)^{\frac{N-2}{2}+\tau+\theta}}
\right]
\Bigg\}.
\end{align*}
In terms of the assumption that $\|(\phi_1, \phi_2)\|_{*}=1$, there exist $R, a>0$ such that for some $j$,
\begin{align}\label{phi12}
\|(\phi_1, \phi_2)\|_{L^{\infty}B_R(\overline{x}_j)}\geq a>0.
\end{align}
Let $(\widetilde{\phi}_1(y), \widetilde{\phi}_2(y)):=(\phi_1(y-\overline{x}_j), \phi_2(y-\overline{x}_j))\rightarrow (u,v)$ satisfies the following equation
\begin{align*}
\begin{cases}
-\Delta u-pV_{0,\Lambda}^{p-1}v=0\\
-\Delta v-qU_{0,\Lambda}^{q-1}u=0
\end{cases}
\text{ for some } \lambda\in [L_1,L_2].
\end{align*}
Since $\phi_i(i=1,2)$ is even in $y_d$, $d=2,4,\cdots, N$, we know that $(u,v)$ is even in $y_d$, $d=2,4,\cdots, N$.
Then $(u,v)$ must be a linear combination of the functions
\begin{align*}
\left(\frac{\partial U_{0,\Lambda}}{\partial y_1}, \frac{\partial V_{0,\Lambda}}{\partial y_1}\right),\ \
\left(\frac{\partial U_{0,\Lambda}}{\partial y_3}, \frac{\partial V_{0,\Lambda}}{\partial y_3}\right),\ \
\left(y\cdot \nabla U_{0,\Lambda}+(N-2)U_{0,\Lambda}, y\cdot \nabla V_{0,\Lambda}+(N-2)V_{0,\Lambda}\right).
\end{align*}
According to the fact that for $l=1,2,3$ and fixed $j$,
\begin{align*}
\int_{\mathbb{R}^N}U_{\overline{x}_j,\Lambda}^{q-1}\overline{Z}_{lj}\widetilde{\phi}_2
=0,
\end{align*}
we have that for $l=1$,
\begin{align}\label{phi2}
\sqrt{1-h^2}\int_{\mathbb{R}^N}U_{0,\Lambda}^{q-1}\frac{\partial U_{0,\Lambda}}{\partial y_1}\widetilde{\phi}_2
+h\int_{\mathbb{R}^N}U_{0,\Lambda}^{q-1}\frac{\partial U_{0,\Lambda}}{\partial y_3}\widetilde{\phi}_2
=0,
\end{align}
for $l=2$,
\begin{align}\label{phi2-}
\sqrt{1-h^2}\int_{\mathbb{R}^N}U_{0,\Lambda}^{q-1}\frac{\partial U_{0,\Lambda}}{\partial y_1}\widetilde{\phi}_2
-h\int_{\mathbb{R}^N}U_{0,\Lambda}^{q-1}\frac{\partial U_{0,\Lambda}}{\partial y_3}\widetilde{\phi}_2
=0,
\end{align}
and for $l=3$,
\begin{align}\label{phi3}
\int_{\mathbb{R}^N}U_{0,\Lambda}^{q-1}\frac{\partial U_{0,\Lambda}}{\partial \Lambda}\widetilde{\phi}_2=0.
\end{align}
By taking limit, we have
\begin{align*}
\int_{\mathbb{R}^N}U_{0,\Lambda}^{q-1}\frac{\partial U_{0,\Lambda}}{\partial y_1}u
=\int_{\mathbb{R}^N}U_{0,\Lambda}^{q-1}\frac{\partial U_{0,\Lambda}}{\partial y_3}u
=\int_{\mathbb{R}^N}U_{0,\Lambda}^{q-1}\left(y\cdot \nabla U_{0,\Lambda}+(N-2)U_{0,\Lambda}\right)u
=0.
\end{align*}
Then $u=0$. And by the same token, $v=0$, which is a contradiction to (\ref{phi12}). That concludes the proof.
\end{proof}

Based on the Lemma \ref{lpL}, we have the following existence, uniqueness results related to (\ref{linearp}) and the following estimates of $(\phi_1, \phi_2)$ and $c_i$ with $i=1,2,3$.
\begin{proposition}\label{lpLc}
There exist $k_0 > 0$ and a constant $C>0$ such that for all $k\geq k_0$ and all $(f_1,f_2)\in L^{\infty}(\mathbb{R}^N)\times L^{\infty}(\mathbb{R}^N)$, problem (\ref{linearp}) has a unique solution $(\phi_1,\phi_2)= \mathbb{L}_k(f_1,f_2)$. Moreover, there hold that
\begin{align*}
\|(\phi_1, \phi_2)\|_{*}\leq C\|(f_1, f_2)\|_{**}, \quad |c_i|\leq \frac{C}{1+\delta_{l2}r}\|(f_1, f_2)\|_{**}.
\end{align*}
\end{proposition}

We next consider the following nonlinear problem
\begin{equation}\label{nlp}
\begin{cases}
-\Delta(W_1+\phi_1)=K_1\left(\frac{|y|}{\mu}\right)(W_2+\phi_2)^p
+\sum_{l=1}^3\sum_{j=1}^kc_l\left(V_{\overline{x}_j,\Lambda}^{p-1}\overline{Y}_{lj}
+V_{\underline{x}_j,\Lambda}^{p-1}\underline{Y}_{lj}
\right)  \text{ in } \mathbb{R}^N, \\
-\Delta(W_2+\phi_2)=K_2\left(\frac{|y|}{\mu}\right)(W_1+\phi_1)^q
+\sum_{l=1}^3\sum_{j=1}^kc_l\left(U_{\overline{x}_j,\Lambda}^{q-1}\overline{Z}_{lj}
+U_{\underline{x}_j,\Lambda}^{q-1}\underline{Z}_{lj}
\right)  \text{ in } \mathbb{R}^N, \\
(\phi_1, \phi_2)\in \mathbb{E},
\end{cases}
\end{equation}
which is equivalent to the following problem
\begin{equation}\label{nlp1}
\begin{cases}
L_k(\phi_1, \phi_2)=R_k+N_k(\phi_1, \phi_2)+\sum_{l=1}^3\sum_{j=1}^kc_l\langle V_{\overline{x}_j,\Lambda}^{p-1}\overline{Y}_{lj}
+V_{\underline{x}_j,\Lambda}^{p-1}\underline{Y}_{lj}, U_{\overline{x}_j,\Lambda}^{q-1}\overline{Z}_{lj}
+U_{\underline{x}_j,\Lambda}^{q-1}\underline{Z}_{lj}\rangle, \\
(\phi_1, \phi_2)\in \mathbb{E},
\end{cases}
\end{equation}
where $L_k$ is defined in (\ref{lk}),
$N_k(\phi_1, \phi_2)=(N_{1,k}(\phi_2), N_{2,k}(\phi_1))$ with
\begin{align*}
N_{1,k}(\phi_2)=K_1\left(\frac{|y|}{\mu}\right)\left((W_2+\phi_2)^p-W_2^p-pW_2^{p-1}\phi_2
\right),
\end{align*}
\begin{align*}
N_{2,k}(\phi_1)=K_2\left(\frac{|y|}{\mu}\right)\left((W_1+\phi_1)^q-W_1^q-qW_1^{q-1}\phi_1
\right),
\end{align*}
and
\begin{align*}
R_k=(R_{1,k},R_{2,k})
=\left(K_1\left(\frac{|y|}{\mu}\right)W_2^p-\sum_{j=1}^k(V_{\overline{x}_j,\Lambda}^p+V_{\underline{x}_j,\Lambda}^p),
K_2\left(\frac{|y|}{\mu}\right)W_1^q-\sum_{j=1}^k(U_{\overline{x}_j,\Lambda}^q+U_{\underline{x}_j,\Lambda}^q) \right).
\end{align*}

In the following part, we first give the estimates of $N_k(\phi_1, \phi_2)$ and $R_k$.
\begin{lemma}\label{nk}
Suppose that $N\geq 5$. There exists $C>0$ such that
\begin{align*}
\|N_k(\phi_1, \phi_2)\|_{**}\leq C\|(\phi_1, \phi_2)\|_{*}^{\min\{p,2\}}.
\end{align*}
\end{lemma}
\begin{proof}
By the definition of $N_k(\phi_1, \phi_2)$, we mainly estimate the term $N_{1,k}(\phi_2)$, since the other term $N_{2,k}(\phi_1)$ can be estimated in the same way. According to the fact that
\begin{equation*}
\left|N_{1,k}(\phi_2)\right|\leq
\begin{cases}
C|\phi_2|^p,\ &\text{ if } 1<p\leq 2;\\
CW_2^{p-2}\phi_2^2+C|\phi_2|^p,\ &\text{ if } p>2,
\end{cases}
\end{equation*}
we have that for $1<p\leq 2$, by H\"{o}lder inequalities,
\begin{align*}
\left|N_{1,k}(\phi_2)\right|\leq &C\|\phi_2\|_{*}^p
\left(\sum_{j=1}^k\frac{1}{(1+|y-\overline{x}_j|)^{\frac{N-2}{2}+\tau}}
+\sum_{j=1}^k\frac{1}{(1+|y-\underline{x}_j|)^{\frac{N-2}{2}+\tau}}\right)^p
\nonumber\\
\leq &C\|\phi_2\|_{*}^p
\left\{\left(\sum_{j=1}^k\frac{1}{(1+|y-\overline{x}_j|)^{\frac{N-2}{2}+\tau}}
\right)^p
+\left(\sum_{j=1}^k\frac{1}{(1+|y-\underline{x}_j|)^{\frac{N-2}{2}+\tau}}\right)^p
\right\}
\nonumber\\
\leq &C\|\phi_2\|_{*}^p\left(\sum_{j=1}^k\frac{1}{(1+|y-\overline{x}_j|)^{\frac{N+2}{2}+\tau}}\right)
\left(\sum_{j=1}^k\frac{1}{(1+|y-\overline{x}_j|)^{\tilde{\tau}}}\right)^{p-1}
\nonumber\\
&+C\|\phi_2\|_{*}^p\left(\sum_{j=1}^k\frac{1}{(1+|y-\underline{x}_j|)^{\frac{N+2}{2}+\tau}}\right)
\left(\sum_{j=1}^k\frac{1}{(1+|y-\underline{x}_j|)^{\tilde{\tau}}}\right)^{p-1}
\nonumber\\
\leq &C\|\phi_2\|_{*}^p
\left(\sum_{j=1}^k\frac{1}{(1+|y-\overline{x}_j|)^{\frac{N+2}{2}+\tau}}
+\sum_{j=1}^k\frac{1}{(1+|y-\underline{x}_j|)^{\frac{N+2}{2}+\tau}}\right),
\end{align*}
where $\tilde{\tau}=\frac{(\frac{N-2}{2}+\tau)p-\frac{N+2}{2}-\tau}{p-1}\geq\frac{N-2-m}{N-2}$ and
$\tau\geq \max\{\frac{N-2-m}{N-2}+\frac{2}{p-1}-\frac{N-2}{2},\frac{N-2-m}{N-2}\}=\frac{N-2-m}{N-2}+\frac{2}{p-1}-\frac{N-2}{2}$.

For $p>2$, we have
\begin{align*}
\left|W_2^{P-2}\phi_2^2\right|\leq &C
\|\phi_2\|_{*}^2\left(\sum_{j=1}^k\frac{1}{(1+|y-\overline{x}_j|)^{\frac{N-2}{2}+\tau}}
+\sum_{j=1}^k\frac{1}{(1+|y-\underline{x}_j|)^{\frac{N-2}{2}+\tau}}\right)^2W_2^{p-2}
\nonumber\\
\leq &C
\|\phi_2\|_{*}^2\left(\sum_{j=1}^k\frac{1}{(1+|y-\overline{x}_j|)^{\frac{N-2}{2}+\tau}}
+\sum_{j=1}^k\frac{1}{(1+|y-\underline{x}_j|)^{\frac{N-2}{2}+\tau}}\right)^2
\nonumber\\
&\times \left(\sum_{j=1}^k\frac{1}{(1+|y-\overline{x}_j|)^{N-2}}
+\sum_{j=1}^k\frac{1}{(1+|y-\underline{x}_j|)^{N-2}}\right)^{p-2}
\nonumber\\
\leq &C
\|\phi_2\|_{*}^2\left(\sum_{j=1}^k\frac{1}{(1+|y-\overline{x}_j|)^{\frac{N-2}{2}+\tau}}
+\sum_{j=1}^k\frac{1}{(1+|y-\underline{x}_j|)^{\frac{N-2}{2}+\tau}}\right)^p
\nonumber\\
\leq &C\|\phi_2\|_{*}^2
\left(\sum_{j=1}^k\frac{1}{(1+|y-\overline{x}_j|)^{\frac{N+2}{2}+\tau}}
+\sum_{j=1}^k\frac{1}{(1+|y-\underline{x}_j|)^{\frac{N+2}{2}+\tau}}\right).
\end{align*}
Thus
\begin{align*}
\left|N_{1,k}(\phi_2)\right|\leq C\|\phi_2\|_{*}^{\min\{p,2\}}.
\end{align*}
Thereforex
\begin{align*}
\|N_k(\phi_1, \phi_2)\|_{**}\leq C\|(\phi_1, \phi_2)\|_{*}^{\min\{p,2\}}.
\end{align*}
\end{proof}

\begin{lemma}\label{rk}
Assume that $||x_1|-\mu|<\frac{1}{\mu^{\hat{\theta}}}$, where  $\hat{\theta}>0$ is a fixed small constant. If $N\geq5$, then there is a constant $\beta$ satisfying
\begin{align*}
&\max\left\{\frac{m}{2}+\frac{(N-3)(N-2-m)}{(N-1)(N-2)},\frac{m+1}{2}+\frac{(N-3)(N-2-m)}{2(N-1)(N-2)}\right\}
<\beta
\nonumber\\
&<\min\left\{m, p(N-2)-\frac{N+2}{2}-\tau\right\},
\end{align*}
such that
\begin{align*}
\|R_k\|_{**} \leq C\max\left\{\left(\frac{k}{\mu}\right)^{p(N-2)-\frac{N+2}{2}-\tau},\left(\frac{1}{\mu}\right)^{\beta}\right\}.
\end{align*}
\end{lemma}
\begin{proof}
Since $p\leq q$, we only have to estimate $\|R_{1,k}\|_{**}$. By the definition of $R_{1,k}$, we have
\begin{align*}
R_{1,k}=&K_1\left(\frac{|y|}{\mu}\right)W_2^p-\sum_{j=1}^k\left(V_{\overline{x}_j,\Lambda}^p+V_{\underline{x}_j,\Lambda}^p\right)
\nonumber\\
=&K_1\left(\frac{|y|}{\mu}\right)\left(W_2^p-\sum_{j=1}^k\left(V_{\overline{x}_j,\Lambda}^p+V_{\underline{x}_j,\Lambda}^p\right)\right)
+\sum_{j=1}^k\left\{K_1\left(\frac{|y|}{\mu}\right)-1\right\}\left(V_{\overline{x}_j,\Lambda}^p+V_{\underline{x}_j,\Lambda}^p\right)
\nonumber\\
:=&G_1+G_2.
\end{align*}
By symmetry, we assume that $y\in \Omega^+_1$,
\begin{align}\label{j1}
G_1=&K_1\left(\frac{|y|}{\mu}\right)\left(W_2^p-\sum_{j=1}^k\left(V_{\overline{x}_j,\Lambda}^p+V_{\underline{x}_j,\Lambda}^p\right)\right)
\nonumber\\
=&K_1\left(\frac{|y|}{\mu}\right)\left(\left(\sum_{j=1}^kV_{\overline{x}_j,\Lambda}+\sum_{j=1}^kV_{\underline{x}_j,\Lambda}\right)^p
-\sum_{j=1}^k\left(V_{\overline{x}_j,\Lambda}^p+V_{\underline{x}_j,\Lambda}^p\right)\right)
\nonumber\\
\leq &CK_1\left(\frac{|y|}{\mu}\right)
\left\{V_{\overline{x}_1,\Lambda}^{p-1}\left(\sum_{j=2}^kV_{\overline{x}_j,\Lambda}+\sum_{j=1}^kV_{\underline{x}_j,\Lambda}\right)
+\left(\sum_{j=2}^kV_{\overline{x}_j,\Lambda}+\sum_{j=1}^kV_{\underline{x}_j,\Lambda}\right)^p
\right\}
\nonumber\\
\leq &C\Bigg\{\frac{1}{(1+|y-\overline{x}_1|)^{(N-2)(p-1)}}
\left(\sum_{j=2}^k\frac{1}{(1+|y-\overline{x}_j|)^{N-2}}+\sum_{j=1}^k\frac{1}{(1+|y-\underline{x}_j|)^{N-2}}\right)
\nonumber\\
&+\left(\sum_{j=2}^k\frac{1}{(1+|y-\overline{x}_j|)^{N-2}}\right)^p
\Bigg\}
:=G_{11}+G_{12}+G_{13}.
\end{align}
By \cite[Lemma B.1]{duanmw}, for any $1<\alpha_1<N-2$ and $y\in \Omega^+_1$,
\begin{align*}
\sum_{j=2}^k\frac{1}{(1+|y-\overline{x}_j|)^{N-2}}\leq \frac{C}{(1+|y-\overline{x}_1|)^{N-2-\alpha_1}}\sum_{j=2}^k\frac{1}{|\overline{x}_1-\overline{x}_j|^{\alpha_1}}.
\end{align*}
We choose $\alpha_1$ satisfying $\alpha_1=p(N-2)-\frac{N+2}{2}-\tau<\frac{N+2}{2}-\tau<N-2$, then
\begin{align}\label{j11}
G_{11}\leq
&\frac{C}{(1+|y-\overline{x}_1|)^{p(N-2)-\alpha_1}}\sum_{j=2}^k\frac{1}{|\overline{x}_1-\overline{x}_j|^{\alpha_1}}
\nonumber\\
\leq &
\frac{C}{(1+|y-\overline{x}_1|)^{\frac{N+2}{2}+\tau}}\left(\frac{k}{r\sqrt{1-h^2}}\right)^{\alpha_1}
\nonumber\\
\leq &
\frac{C}{(1+|y-\overline{x}_1|)^{\frac{N+2}{2}+\tau}}
\left(\frac{k}{r}\right)^{p(N-2)-\frac{N+2}{2}-\tau}.
\end{align}
By $|y-\underline{x}_j|\geq |y-\overline{x}_j|$ for $y\in \Omega^+_1$, we get
\begin{align*}
G_{12}=&\frac{C}{(1+|y-\overline{x}_1|)^{(N-2)(p-1)}}
\left(\frac{1}{(1+|y-\underline{x}_1|)^{N-2}}+\sum_{j=2}^k\frac{1}{(1+|y-\underline{x}_j|)^{N-2}}\right)
\nonumber\\
\leq &
\frac{C}{(1+|y-\overline{x}_1|)^{(N-2)(p-1)}}
\left(\frac{1}{(1+|y-\underline{x}_1|)^{N-2}}+\sum_{j=2}^k\frac{1}{(1+|y-\overline{x}_j|)^{N-2}}\right).
\end{align*}
Choose $\alpha_1$ satisfying $\alpha_1=p(N-2)-\frac{N+2}{2}-\tau<\frac{N+2}{2}-\tau<N-2$, it follows from $hr>\frac{cr}{k}$ that
\begin{align*}
&\frac{C}{(1+|y-\overline{x}_1|)^{(N-2)(p-1)}}\frac{1}{(1+|y-\underline{x}_1|)^{N-2}}
\nonumber\\
\leq
&\left[\frac{C}{(1+|y-\overline{x}_1|)^{p(N-2)-\alpha_1}}+\frac{C}{(1+|y-\underline{x}_1|)^{p(N-2)-\alpha_1}}\right]\frac{1}{|\overline{x}_1-\underline{x}_1|^{\alpha_1}}
\nonumber\\
\leq
&\frac{C}{(1+|y-\overline{x}_1|)^{\frac{N+2}{2}+\tau}}\frac{1}{(hr)^{\alpha_1}}
\leq \frac{C}{(1+|y-\overline{x}_1|)^{\frac{N+2}{2}+\tau}}\left(\frac{k}{r}\right)^{\alpha_1}.
\end{align*}
Therefore,
\begin{align}\label{j12}
G_{12}\leq \frac{C}{(1+|y-\overline{x}_1|)^{\frac{N+2}{2}+\tau}}\left(\frac{k}{r}\right)^{\alpha_1}.
\end{align}
For $y\in \Omega^+_1$, we have $|y-\overline{x}_j|\geq |y-\overline{x}_1|$. Then
\begin{align*}
\sum_{j=2}^k\frac{1}{(1+|y-\overline{x}_j|)^{N-2}}
\leq
&\sum_{j=2}^k\frac{1}{(1+|y-\overline{x}_1|)^{\frac{N-2}{2}}}\frac{1}{(1+|y-\overline{x}_j|)^{\frac{N-2}{2}}}
\nonumber\\
\leq
&\sum_{j=2}^k\frac{1}{|\overline{x}_1-\overline{x}_j|^{\frac{N-2}{2}-[\frac{1}{p}(\frac{N+2}{2}+\tau)-\frac{N-2}{2}]}}
\frac{1}{(1+|y-\overline{x}_1|)^{\frac{N-2}{2}+[\frac{1}{p}(\frac{N+2}{2}+\tau)-\frac{N-2}{2}]}}
\nonumber\\
\leq
&\sum_{j=2}^k\frac{1}{|\overline{x}_1-\overline{x}_j|^{N-2-\frac{1}{p}(\frac{N+2}{2}+\tau)}}
\frac{1}{(1+|y-\overline{x}_1|)^{\frac{1}{p}(\frac{N+2}{2}+\tau)}}.
\end{align*}
By the above inequality, we get
\begin{align}\label{j13}
G_{13}\leq \frac{C}{(1+|y-\overline{x}_1|)^{\frac{N+2}{2}+\tau}}\left(\frac{k}{r}\right)^{p(N-2)-\frac{N+2}{2}-\tau}
\leq \frac{C}{(1+|y-\overline{x}_1|)^{\frac{N+2}{2}+\tau}}\left(\frac{k}{r}\right)^{\alpha_1}.
\end{align}
From (\ref{j1})-(\ref{j13}), we have
\begin{align*}
\|G_{1}\|_{**}=O\left(\left(\frac{k}{r}\right)^{\alpha_1}\right).
\end{align*}

Next, we estimate $G_{2}$. In fact, for $y\in \Omega^+_1$,
\begin{align}\label{j2}
G_{2}\leq& 2\sum_{j=1}^k\left\{K_1\left(\frac{|y|}{\mu}\right)-1\right\}V_{\overline{x}_j,\Lambda}^p
\nonumber\\
=&2\left\{K_1\left(\frac{|y|}{\mu}\right)-1\right\}V_{\overline{x}_1,\Lambda}^p
+2\sum_{j=2}^k\left\{K_1\left(\frac{|y|}{\mu}\right)-1\right\}V_{\overline{x}_j,\Lambda}^p
:=G_{21}+G_{22}.
\end{align}
If $\left|\frac{|y|}{\mu}-1\right|\geq \delta_1$ with $\delta>\delta_1>0$ and $y\in \Omega^+_1$, then
\begin{align*}
|y-\overline{x}_1|\geq ||y|-\mu|-|\mu-|\overline{x}_1||\geq \frac{1}{2}\delta_1\mu.
\end{align*}
Moreover,
\begin{align}\label{j211}
\left\{K_1\left(\frac{|y|}{\mu}\right)-1\right\}V_{\overline{x}_1,\Lambda}^p
\leq &
\frac{C}{(1+|y-\overline{x}_1|)^{p(N-2)}}
\nonumber\\
\leq &
\frac{C}{(1+|y-\overline{x}_1|)^{\frac{N+2}{2}+\tau}}\left(\frac{1}{\mu}\right)^{p(N-2)-\frac{N+2}{2}-\tau}
\leq
\frac{C}{(1+|y-\overline{x}_1|)^{\frac{N+2}{2}+\tau}}\left(\frac{1}{\mu}\right)^{\alpha_1}.
\end{align}
If $\left|\frac{|y|}{\mu}-1\right|\leq \delta_1$, we obtain
\begin{align*}
K_1\left(\frac{|y|}{\mu}\right)-1 \leq
C\left|\frac{|y|}{\mu}-1\right|^{m_1}
\leq \frac{C}{\mu^{m_1}}\left[||y|-|\overline{x}_1||^{m_1}+||\overline{x}_1|-\mu|^{m_1}\right]
\leq \frac{C}{\mu^{m_1}}\left[||y|-|\overline{x}_1||^{m_1}+\frac{1}{k^{\bar{\theta}m_1 }}\right].
\end{align*}
As a result, we get that for $\beta<m_1$,
\begin{align}\label{j212}
\left\{K_1\left(\frac{|y|}{\mu}\right)-1\right\}V_{\overline{x}_1,\Lambda}^p
\leq &\frac{C}{\mu^{m_1}}\left[||y|-|\overline{x}_1||^{m_1}+\frac{1}{k^{\bar{\theta}m_1 }}\right]
\frac{1}{(1+|y-\overline{x}_1|)^{p(N-2)}}
\nonumber\\
\leq &\frac{C}{\mu^{\beta}}\left[\frac{||y|-|\overline{x}_1||^{\beta}}{(1+|y-\overline{x}_1|)^{p(N-2)}}
+\frac{1}{\mu^{m_1-\beta}}\frac{1}{k^{\bar{\theta}m_1 }}\frac{1}{(1+|y-\overline{x}_1|)^{p(N-2)}}\right]
\nonumber\\
\leq &\frac{C}{\mu^{\beta}}\Bigg[\frac{1}{(1+|y-\overline{x}_1|)^{\frac{N+2}{2}+\tau}}
\frac{1}{(1+|y-\overline{x}_1|)^{p(N-2)-\frac{N+2}{2}-\tau-\beta}}
\nonumber\\
&+\frac{1}{(1+|y-\overline{x}_1|)^{p(N-2)}}
\Bigg]
\nonumber\\
\leq &\frac{C}{\mu^{\beta}}\frac{1}{(1+|y-\overline{x}_1|)^{\frac{N+2}{2}+\tau}}.
\end{align}
It follows from (\ref{j211})-(\ref{j212}) that
\begin{align}\label{j21}
G_{21}\leq C\max\left\{\left(\frac{1}{\mu}\right)^{\alpha_1}, \frac{1}{\mu^{\beta}}\right\}\frac{1}{(1+|y-\overline{x}_1|)^{\frac{N+2}{2}+\tau}}.
\end{align}
As to $G_{22}$. For $y\in \Omega^+_1$, $j=2,\cdots, k$,
\begin{align*}
|\overline{x}_1-\overline{x}_j|\leq |y-\overline{x}_1|+|y-\overline{x}_j|\leq 2|y-\overline{x}_j|.
\end{align*}
Then
\begin{align}\label{j22}
J_{22}\leq &C\frac{1}{(1+|y-\overline{x}_1|)^{\frac{N+2}{2}}}\sum_{j=2}^k\frac{1}{(1+|y-\overline{x}_j|)^{p(N-2)-\frac{N+2}{2}}}
\nonumber\\
\leq &
\frac{C}{(1+|y-\overline{x}_1|)^{\frac{N+2}{2}+\tau}}\sum_{j=2}^k
\frac{1}{|\overline{x}_1-\overline{x}_j|^{p(N-2)-\frac{N+2}{2}-\tau}}
\nonumber\\
\leq &
\frac{C}{(1+|y-\overline{x}_1|)^{\frac{N+2}{2}+\tau}}\left(\frac{k}{\mu}\right)^{p(N-2)-\frac{N+2}{2}-\tau}.
\end{align}
In virtue of (\ref{j21})-(\ref{j22}), we get
\begin{align*}
\|G_{2}\|_{**}=O\left(\max\left\{\left(\frac{k}{\mu}\right)^{\alpha_1}, \frac{1}{\mu^{\beta}}\right\}\right).
\end{align*}
In conclusion,
\begin{align*}
\|R_{1,k}\|_{**}=O\left(\max\left\{\left(\frac{k}{\mu}\right)^{\alpha_1}, \frac{1}{\mu^{\beta}}\right\}\right).
\end{align*}
\end{proof}

\begin{proposition}\label{cmphi}
Suppose that $K_i(|y|)$ satisfies $(K_i)$ with $i=1,2$ and $N\geq 5$, $(r,h,\Lambda)\in \wp_k$. There exists $k_0$ such that for all $k\geq k_0$, problem (\ref{nlp}) has a unique solution $(\phi_1, \phi_2)$ which satisfies
\begin{align*}
\|(\phi_1, \phi_2)\|_{*}\leq C\max\left\{\left(\frac{k}{\mu}\right)^{\alpha_1}, \frac{1}{\mu^{\beta}}\right\}.
\end{align*}
Moreover,
\begin{align*}
|c_l|\leq \frac{C}{1+\delta_{l2}r}\max\left\{\left(\frac{k}{\mu}\right)^{\alpha_1}, \frac{1}{\mu^{\beta}}\right\}.
\end{align*}
\end{proposition}
\begin{proof}
Let us define
\begin{align*}
B=\left\{(v_1,v_2): (v_1,v_2)\in E, \|(v_1,v_2)\|_{*}\leq C\max\left\{\left(\frac{k}{\mu}\right)^{\alpha_1}, \frac{1}{\mu^{\beta}}\right\}   \right\}.
\end{align*}
Then, (\ref{nlp1}) is equivalent to
\begin{align*}
(\phi_1,\phi_2)=A(\phi_1,\phi_2):=L_k(N_k(\phi_1, \phi_2))+L_k(R_k).
\end{align*}
We will prove that $A$ is a contraction map from $B$ to $B$. In fact,
\begin{align*}
\|A(\phi_1,\phi_2)\|_{*}
\leq &C\|N_k(\phi_1, \phi_2)\|_{**}+\|R_k\|_{**}
\nonumber\\
\leq &C\|(\phi_1, \phi_2)\|_{*}^{\min\{p,2\}}+C\max\left\{\left(\frac{k}{\mu}\right)^{\alpha_1}, \frac{1}{\mu^{\beta}}\right\}
\leq C\max\left\{\left(\frac{k}{\mu}\right)^{\alpha_1}, \frac{1}{\mu^{\beta}}\right\}.
\end{align*}
Thus, $A$ maps $B$ to $B$.

On the other hand, for any $(\phi_1, \phi_2),(\tilde{\phi}_1, \tilde{\phi}_2)\in B$, we have
\begin{align*}
\|A(\phi_1,\phi_2)-A(\tilde{\phi}_1,\tilde{\phi}_2)\|_{*}
\leq \|N_k(\phi_1, \phi_2)-N_k(\tilde{\phi}_1, \tilde{\phi}_2)\|_{**}.
\end{align*}
Since
\begin{align*}
\left|N_{1,k}(\phi_2)-N_{1,k}(\tilde{\phi}_2)\right|
=&\left|K_1\left(\frac{|y|}{\mu}\right)\left[(W_2+\phi_2)^p-pW_2^{p-1}\phi_2
-(W_2+\tilde{\phi}_2)^p+pW_2^{p-1}\tilde{\phi}_2
\right]
\right|
\nonumber\\
\leq &C\left|(W_2+\phi_2)^p-(W_2+\tilde{\phi}_2)^p-pW_2^{p-1}(\phi_2-\tilde{\phi}_2)\right|
\nonumber\\
\leq &C(\|\phi_2\|_{*}^{p-1}+\|\tilde{\phi}_2\|_{*}^{p-1})\|\phi_2-\tilde{\phi}_2\|_{*}
\nonumber\\
&\times \left(\sum_{j=1}^k\frac{1}{(1+|y-\overline{x}_j|)^{\frac{N-2}{2}+\tau}}
+\sum_{j=1}^k\frac{1}{(1+|y-\underline{x}_j|)^{\frac{N-2}{2}+\tau}}\right)^p
\nonumber\\
\leq &C(\|\phi_2\|_{*}^{p-1}+\|\tilde{\phi}_2\|_{*}^{p-1})\|\phi_2-\tilde{\phi}_2\|_{*}
\nonumber\\
&\times \left(\sum_{j=1}^k\frac{1}{(1+|y-\overline{x}_j|)^{\frac{N+2}{2}+\tau}}
+\sum_{j=1}^k\frac{1}{(1+|y-\underline{x}_j|)^{\frac{N+2}{2}+\tau}}\right).
\end{align*}
Similarly, we obtain
\begin{align*}
\left|N_{2,k}(\phi_1)-N_{2,k}(\tilde{\phi}_1)\right|
\leq &C(\|\phi_1\|_{*}^{p-1}+\|\tilde{\phi}_1\|_{*}^{p-1})\|\phi_1-\tilde{\phi}_1\|_{*}
\nonumber\\
&\times \left(\sum_{j=1}^k\frac{1}{(1+|y-\overline{x}_j|)^{\frac{N+2}{2}+\tau}}
+\sum_{j=1}^k\frac{1}{(1+|y-\underline{x}_j|)^{\frac{N+2}{2}+\tau}}\right).
\end{align*}
Above all,
\begin{align*}
\|A(\phi_1,\phi_2)-A(\tilde{\phi}_1,\tilde{\phi}_2)\|_{*}
\leq o(1)\|(\phi_1, \phi_2)-(\tilde{\phi}_1, \tilde{\phi}_2)\|_{*}.
\end{align*}
Thus, $A$ is a contraction map.
It follows from the contraction mapping theorem that there is a unique $(\phi_1, \phi_2)\in B$, such that
\begin{align*}
(\phi_1,\phi_2)=A(\phi_1,\phi_2).
\end{align*}
Then from Proposition \ref{lpLc}, we can get the estimates of $(\phi_1,\phi_2)$ and $c_l$.
\end{proof}

\section{Proof of the existence result}\label{existence}

Define
\begin{align*}
F(r,h,\Lambda):=I(W_1+\phi_1,W_2+\phi_2),
\end{align*}
and
\begin{align*}
I(u,v)=\int_{\mathbb{R}^N}\nabla u\nabla v
-\frac{1}{p+1}\int_{\mathbb{R}^N}K_1\left(\frac{|y|}{\mu}\right)|v|^{p+1}
-\frac{1}{q+1}\int_{\mathbb{R}^N}K_2\left(\frac{|y|}{\mu}\right)|u|^{q+1}.
\end{align*}

\begin{proposition}\label{energyex}
Suppose that $K_i(|y|)$ satisfies $(K_i)$ and $N\geq 5$, $(r,h,\Lambda)\in \wp_k$,
we have the following expansion
\begin{align*}
F(r,h,\Lambda)=&I(W_1,W_2)+kO\left(\frac{1}{k^{\frac{m(N-2)}{N-2-m}+\frac{2(N-3)}{N-1}+\theta}}\right)
\nonumber\\
=&kA_1-\frac{k}{\Lambda^{N-2}}\left[\frac{B_4k^{N-2}}{(r\sqrt{1-h^2})^{N-2}}
-\frac{B_5k}{r^{N-2}h^{N-3}\sqrt{1-h^2}}\right]
\nonumber\\
&+k\left[\frac{A_2}{\mu^{m_2}\Lambda^{m_2}}+\frac{\bar{A}_2}{\mu^{m_1}\Lambda^{m_1}}\right]
+k\left[\frac{A_3}{\mu^{m_2}\Lambda^{m_2-2}}+\frac{\bar{A}_3}{\mu^{m_1}\Lambda^{m_1-2}}\right](\mu-r)^2
\nonumber\\
&+k\frac{C(r,\Lambda)}{\mu^m}(\mu-r)^{\theta+2}+k\frac{C(r,\Lambda)}{\mu^{m+\theta}}
+kO\left(\frac{1}{k^{\frac{m(N-2)}{N-2-m}+\frac{2(N-3)}{N-1}+\theta}}\right),
\end{align*}
where $\theta>0$ is a fixed constant and  $A_1, A_2, A_3, \bar{A}_2, \bar{A}_3, B_4, B_5$ are positive constants and defined in Proposition \ref{iw}.
\end{proposition}
\begin{proof}
Since
\begin{align*}
\langle (I'_u(W_1+\phi_1,W_2+\phi_2),I'_v(W_1+\phi_1,W_2+\phi_2)), (\phi_1, \phi_2)\rangle=0, \ \forall (\phi_1, \phi_2)\in B,
\end{align*}
there are $t,s\in (0, 1)$ such that
\begin{align*}
F(r,h,\Lambda)=&I(W_1,W_2)-\frac{1}{2}\langle D^2I(W_1+t\phi_1,W_2+s\phi_2)(\phi_1, \phi_2),(\phi_1, \phi_2)\rangle
\nonumber\\
=&I(W_1,W_2)-\frac{1}{2}\int_{\mathbb{R}^N}\Bigg\{2\nabla \phi_1 \nabla \phi_2
-qK_2\left(\frac{|y|}{\mu}\right)(W_1+t\phi_1)^{q-1}\phi_1^2
\nonumber\\
&-pK_1\left(\frac{|y|}{\mu}\right)(W_2+s\phi_2)^{p-1}\phi_2^2
\Bigg\}
\nonumber\\
=&I(W_1,W_2)+\frac{1}{2}\int_{\mathbb{R}^N}\left[qK_2\left(\frac{|y|}{\mu}\right)(W_1+t\phi_1)^{q-1}\phi_1^2
+pK_1\left(\frac{|y|}{\mu}\right)(W_2+s\phi_2)^{p-1}\phi_2^2\right]
\nonumber\\
&-\frac{1}{2}\int_{\mathbb{R}^N}\left[N_{1,k}(\phi_2)+R_{1,k}\right]\phi_2+pK_1\left(\frac{|y|}{\mu}\right)W_2^{p-1}\phi_2^2
\nonumber\\
&-\frac{1}{2}\int_{\mathbb{R}^N}\left[N_{2,k}(\phi_1)+R_{2,k}\right]\phi_1+qK_2\left(\frac{|y|}{\mu}\right)W_1^{q-1}\phi_1^2
\nonumber\\
=&I(W_1,W_2)
+\frac{1}{2}\int_{\mathbb{R}^N}qK_2\left(\frac{|y|}{\mu}\right)\left[(W_1+t\phi_1)^{q-1}-W_1^{q-1}\right]\phi_1^2
-\left[N_{2,k}(\phi_1)+R_{2,k}\right]\phi_1
\nonumber\\
&+\frac{1}{2}\int_{\mathbb{R}^N}pK_1\left(\frac{|y|}{\mu}\right)\left[(W_2+s\phi_2)^{p-1}-W_2^{p-1}\right]\phi_2^2
-\left[N_{1,k}(\phi_2)+R_{1,k}\right]\phi_2.
\end{align*}
One has
\begin{align*}
&\int_{\mathbb{R}^N}\left[N_{2,k}(\phi_1)+R_{2,k}\right]\phi_1+\left[N_{1,k}(\phi_2)+R_{1,k}\right]\phi_2
\nonumber\\
\leq &C\left(\|N_k(\phi_1, \phi_2)\|_{**}+\|R_k\|_{**}\right)\|(\phi_1, \phi_2)\|_{*}\int_{\mathbb{R}^N}
\left(\sum_{j=1}^k\frac{1}{(1+|y-\overline{x}_j|)^{\frac{N-2}{2}+\tau}}
+\sum_{j=1}^k\frac{1}{(1+|y-\underline{x}_j|)^{\frac{N-2}{2}+\tau}}\right)
\nonumber\\
&\times \left(\sum_{j=1}^k\frac{1}{(1+|y-\overline{x}_j|)^{\frac{N+2}{2}+\tau}}
+\sum_{j=1}^k\frac{1}{(1+|y-\underline{x}_j|)^{\frac{N+2}{2}+\tau}}\right).
\end{align*}
In view of \cite{weiy}, we have
\begin{align*}
\sum_{j=1}^k\frac{1}{(1+|y-\overline{x}_j|)^{\frac{N-2}{2}+\tau}}
\sum_{i=1}^k\frac{1}{(1+|y-\overline{x}_i|)^{\frac{N+2}{2}+\tau}}
\leq C\sum_{j=1}^k\frac{1}{(1+|y-\overline{x}_j|)^{N+\tau}},
\end{align*}
and
\begin{align*}
\sum_{j=1}^k\frac{1}{(1+|y-\underline{x}_j|)^{\frac{N-2}{2}+\tau}}
\sum_{i=1}^k\frac{1}{(1+|y-\underline{x}_i|)^{\frac{N+2}{2}+\tau}}
\leq C\sum_{j=1}^k\frac{1}{(1+|y-\underline{x}_j|)^{N+\tau}}.
\end{align*}
Moreover, for $y\in \Omega_1^+$, there exists $\bar{\eta}$ satisfying $\bar{\eta}\leq 4\tau-\frac{2(N-2)}{N-2-m}$ such that
\begin{align*}
\frac{1}{(1+|y-\overline{x}_1|)^{\frac{N-2}{2}+\tau}}
\frac{1}{(1+|y-\underline{x}_1|)^{\frac{N+2}{2}+\tau}}
\leq &\frac{C}{(1+|y-\overline{x}_1|)^{N+\frac{1}{2}\bar{\eta}}}
\frac{1}{|\overline{x}_1-\underline{x}_1|^{2\tau-\frac{1}{2}\bar{\eta}}}
\nonumber\\
\leq &\frac{C}{(1+|y-\overline{x}_1|)^{N+\frac{1}{2}\bar{\eta}}},
\end{align*}
and similarly,
\begin{align*}
\frac{1}{(1+|y-\overline{x}_1|)^{\frac{N-2}{2}+\tau}}
\frac{1}{(1+|y-\underline{x}_j|)^{\frac{N+2}{2}+\tau}}
\leq \frac{C}{(1+|y-\overline{x}_1|)^{N+\frac{1}{2}\bar{\eta}}},
\end{align*}
\begin{align*}
\frac{1}{(1+|y-\overline{x}_j|)^{\frac{N-2}{2}+\tau}}
\frac{1}{(1+|y-\underline{x}_i|)^{\frac{N+2}{2}+\tau}}
\leq \frac{C}{(1+|y-\overline{x}_1|)^{N+\frac{1}{2}\bar{\eta}}}.
\end{align*}
Using the above inequalities, we get
\begin{align*}
&\int_{\mathbb{R}^N}\sum_{j=1}^k\frac{1}{(1+|y-\overline{x}_j|)^{\frac{N-2}{2}+\tau}}
\sum_{i=1}^k\frac{1}{(1+|y-\underline{x}_i|)^{\frac{N+2}{2}+\tau}}
\nonumber\\
=&\int_{\mathbb{R}^N}\sum_{j=1}^k\frac{1}{(1+|y-\underline{x}_j|)^{\frac{N-2}{2}+\tau}}
\sum_{i=1}^k\frac{1}{(1+|y-\overline{x}_i|)^{\frac{N+2}{2}+\tau}}
=O\left(k\right).
\end{align*}
Thus,
\begin{align}\label{fd1}
\int_{\mathbb{R}^N}\left[N_{2,k}(\phi_1)+R_{2,k}\right]\phi_1+\left[N_{1,k}(\phi_2)+R_{1,k}\right]\phi_2
=kO\left(\left(\|N_k(\phi_1, \phi_2)\|_{**}+\|R_k\|_{**}\right)\|(\phi_1, \phi_2)\|_{*}\right).
\end{align}
On the other hand,
\begin{align*}
&\int_{\mathbb{R}^N}K_2\left(\frac{|y|}{\mu}\right)(W_1+t\phi_1)^{q-1}\phi_1^2
\nonumber\\
\leq &C\int_{\mathbb{R}^N}\phi_1^{q+1}
\leq C\|\phi_1\|_{*}^{q+1}
\int_{\mathbb{R}^N}
\left(\sum_{j=1}^k\frac{1}{(1+|y-\overline{x}_j|)^{\frac{N-2}{2}+\tau}}
+\sum_{j=1}^k\frac{1}{(1+|y-\underline{x}_j|)^{\frac{N-2}{2}+\tau}}\right)^{q+1}
\nonumber\\
\leq &C\|\phi_1\|_{*}^{q+1}
\int_{\mathbb{R}^N}\left(\sum_{j=1}^k\frac{1}{(1+|y-\overline{x}_j|)^{\frac{N-2}{2}+\tau}}\right)^{q+1}
+\left(\sum_{j=1}^k\frac{1}{(1+|y-\underline{x}_j|)^{\frac{N-2}{2}+\tau}}\right)^{q+1},
\end{align*}
where from \cite[Lemma B.1]{duanmw}, for $y\in \Omega_1^+$ and $\check{\eta}\leq 2\tau-\frac{2(N-2)}{N-2-m}$,
\begin{align*}
\sum_{j=2}^k\frac{1}{(1+|y-\overline{x}_j|)^{\frac{N-2}{2}+\tau}}
\leq &\sum_{j=2}^k\frac{1}{(1+|y-\overline{x}_1|)^{\frac{N-2}{4}+\frac{1}{2}\tau}}
\frac{1}{(1+|y-\overline{x}_j|)^{\frac{N-2}{4}+\frac{1}{2}\tau}}
\nonumber\\
\leq &C\frac{1}{(1+|y-\overline{x}_1|)^{\frac{N-2}{2}+\frac{1}{2}\check{\eta}}}
\sum_{j=2}^k\frac{1}{|\overline{x}_1-\overline{x}_j|^{\tau-\frac{1}{2}\check{\eta}}}
\nonumber\\
\leq &C\frac{1}{(1+|y-\overline{x}_1|)^{\frac{N-2}{2}+\frac{1}{2}\check{\eta}}}.
\end{align*}
Then for $y\in \Omega_1^+$,
\begin{align*}
\left(\sum_{j=1}^k\frac{1}{(1+|y-\overline{x}_j|)^{\frac{N-2}{2}+\tau}}
+\sum_{j=1}^k\frac{1}{(1+|y-\underline{x}_j|)^{\frac{N-2}{2}+\tau}}\right)^{q+1}
\leq &\frac{C}{(1+|y-\overline{x}_1|)^{(\frac{N-2}{2}+\frac{1}{2}\check{\eta})(q+1)}}
\nonumber\\
\leq &\frac{C}{(1+|y-\overline{x}_1|)^{N+\frac{1}{2}\check{\eta}(q+1)}}.
\end{align*}
Moreover,
\begin{align*}
\int_{\mathbb{R}^N}\left(\sum_{j=1}^k\frac{1}{(1+|y-\overline{x}_j|)^{\frac{N-2}{2}+\tau}}
\right)^{q+1}
\leq k\int_{\Omega_1^+}\frac{C}{(1+|y-\overline{x}_1|)^{N+\frac{1}{2}\check{\eta}(q+1)}}\leq Ck.
\end{align*}
In an analogy way,
\begin{align*}
\int_{\mathbb{R}^N}\left(\sum_{j=1}^k\frac{1}{(1+|y-\underline{x}_j|)^{\frac{N-2}{2}+\tau}}
\right)^{q+1}\leq Ck.
\end{align*}
Then we have
\begin{align}\label{fd2}
\int_{\mathbb{R}^N}K_2\left(\frac{|y|}{\mu}\right)(W_1+t\phi_1)^{q-1}\phi_1^2
=kO\left(\|\phi_1\|_{*}^{q+1}\right).
\end{align}
Similar results can be obtained as follows
\begin{align}\label{fd3}
\int_{\mathbb{R}^N}K_1\left(\frac{|y|}{\mu}\right)(W_2+s\phi_2)^{p-1}\phi_2^2
=kO\left(\|\phi_2\|_{*}^{p+1}\right).
\end{align}
From (\ref{fd1})-(\ref{fd3}), we get
\begin{align*}
F(r,h,\Lambda)=&I(W_1,W_2)
+k\left\{O\left(\left(\|N_k(\phi_1, \phi_2)\|_{**}+\|R_k\|_{**}\right)\|(\phi_1, \phi_2)\|_{*}+\|\phi_1\|_{*}^{q+1}
+\|\phi_2\|_{*}^{p+1}\right)
\right\}
\nonumber\\
=&I(W_1,W_2)+kO\left(\|(\phi_1, \phi_2)\|_{*}^2\right).
\end{align*}
Because of
\begin{align*}
&\max\left\{\frac{m}{2}+\frac{(N-3)(N-2-m)}{(N-1)(N-2)},\frac{m+1}{2}+\frac{(N-3)(N-2-m)}{2(N-1)(N-2)}\right\}
<\beta
\nonumber\\
&<\min\left\{m, p(N-2)-\frac{N+2}{2}-\tau\right\}
\end{align*}
and
\begin{align*}
p>\frac{N}{N-2}+\frac{\tau}{N-2}+\frac{(N-3)(N-2-m)}{m(N-1)(N-2)},
\end{align*}
we see that
\begin{align*}
k\|(\phi_1, \phi_2)\|_{*}^2\leq Ck\frac{1}{k^{\frac{m(N-2)}{N-2-m}+\frac{2(N-3)}{N-1}+\theta}}.
\end{align*}
Then combing with Proposition \ref{iw}, we conclude the proof.
\end{proof}

\begin{proposition}\label{flamda}
We have
\begin{align*}
\frac{\partial F(r,h,\Lambda)}{\partial \Lambda}
=&\frac{k(N-2)}{\Lambda^{N-1}}\left[\frac{B_4k^{N-2}}{(r\sqrt{1-h^2})^{N-2}}
+\frac{B_5k}{r^{N-2}h^{N-3}\sqrt{1-h^2}}\right]
\nonumber\\
&-k\left[\frac{A_2m_2}{\mu^{m_2}\Lambda^{m_2+1}}+\frac{\bar{A}_2m_1}{\mu^{m_1}\Lambda^{m_1+1}}\right]
-k\left[\frac{A_3(m_2-2)}{\mu^{m_2}\Lambda^{m_2-1}}+\frac{\bar{A}_3(m_1-2)}{\mu^{m_1}\Lambda^{m_1-1}}\right](\mu-r)^2
\nonumber\\
&+kO\left(\frac{1}{k^{\frac{m(N-2)}{N-2-m}+\theta}}\right),
\end{align*}
where $B_4$, $B_5$, $A_2$, $\bar{A}_2$, $A_3$, $\bar{A}_3$, $m$, $\theta$ are positive constants defined in Proposition \ref{iw}.
\end{proposition}
\begin{proof}
It holds
\begin{align*}
\frac{\partial F(r,h,\Lambda)}{\partial \Lambda}
=&\left\langle I'_{1}(W_1+\phi_1,W_2+\phi_2), \frac{\partial W_1}{\partial \Lambda}+\frac{\partial \phi_1}{\partial \Lambda}\right\rangle
+\left\langle I'_{2}(W_1+\phi_1,W_2+\phi_2), \frac{\partial W_2}{\partial \Lambda}+\frac{\partial \phi_2}{\partial \Lambda}\right\rangle
\nonumber\\
=&\left\langle I'_{1}(W_1+\phi_1,W_2+\phi_2), \frac{\partial W_1}{\partial \Lambda}\right\rangle
+\sum_{l=1}^3\sum_{j=1}^k\left\langle U_{\overline{x}_j, \Lambda}^{q-1}\overline{Z}_{lj}+U_{\underline{x}_j, \Lambda}^{q-1}\underline{Z}_{lj},\frac{\partial \phi_1}{\partial \Lambda}\right\rangle
\nonumber\\
&+\left\langle I'_{1}(W_1+\phi_1,W_2+\phi_2), \frac{\partial W_2}{\partial \Lambda}\right\rangle
+\sum_{l=1}^3\sum_{j=1}^k\left\langle V_{\overline{x}_j, \Lambda}^{p-1}\overline{Y}_{lj}+V_{\underline{x}_j, \Lambda}^{p-1}\underline{Y}_{lj},\frac{\partial \phi_2}{\partial \Lambda}\right\rangle.
\end{align*}
Since
\begin{align*}
\left\langle U_{\overline{x}_j, \Lambda}^{q-1}\overline{Z}_{lj}, \frac{\partial \phi_1}{\partial \Lambda}\right\rangle
=-\left\langle \frac{\partial U_{\overline{x}_j, \Lambda}^{q-1}\overline{Z}_{lj}}{\partial \Lambda}, \phi_1\right\rangle,
\end{align*}
we obtain
\begin{align*}
\left|\sum_{j=1}^kc_l\left\langle U_{\overline{x}_j, \Lambda}^{q-1}\overline{Z}_{lj}, \frac{\partial \phi_1}{\partial \Lambda}\right\rangle
\right|
\leq &C|c_l|\|\phi_1\|_{*}\int_{\mathbb{R}^N}\sum_{j=1}^k\frac{1}{(1+|y-\overline{x}_j|)^{q(N-2)}}
\nonumber\\
&\sum_{j=1}^k\Bigg\{\frac{1}{(1+|y-\overline{x}_j|)^{\frac{N-2}{2}+\tau}}
+\frac{1}{(1+|y-\underline{x}_j|)^{\frac{N-2}{2}+\tau}}
\Bigg\}.
\end{align*}
For $y\in \Omega_1^+$, by \cite[Lemma B.1]{duanmw}, there exists $\hat{\eta}$ satisfying $\frac{2(N-2)}{N-2-m}\leq \hat{\eta}<q(N-2)$ such that
\begin{align*}
\sum_{j=2}^k\frac{1}{(1+|y-\overline{x}_j|)^{q(N-2)}}
\leq &\sum_{j=2}^k\frac{1}{(1+|y-\overline{x}_1|)^{\frac{N-2}{2}q}}\frac{1}{(1+|y-\overline{x}_j|)^{\frac{N-2}{2}q}}
\nonumber\\
\leq &C\frac{1}{(1+|y-\overline{x}_1|)^{q(N-2)-\frac{1}{2}\hat{\eta}}}
\sum_{j=2}^k\frac{1}{|\overline{x}_1-\overline{x}_j|^{\frac{1}{2}\hat{\eta}}}
\leq \frac{C}{(1+|y-\overline{x}_1|)^{\frac{N+2}{2}}}.
\end{align*}
Similarly,
\begin{align*}
\sum_{j=2}^k\frac{1}{(1+|y-\overline{x}_j|)^{\frac{N-2}{2}+\tau}}
\leq \frac{C}{(1+|y-\overline{x}_1|)^{\frac{N-2}{2}+\frac{1}{2}\bar{\eta}_1}}, \ \text{ for } y\in \Omega_1^+,
\end{align*}
and
\begin{align*}
\sum_{j=1}^k\frac{1}{(1+|y-\underline{x}_j|)^{\frac{N-2}{2}+\tau}}
\leq \frac{C}{(1+|y-\underline{x}_1|)^{\frac{N-2}{2}+\frac{1}{2}\bar{\eta}_2}}, \ \text{ for } y\in \Omega_1^+.
\end{align*}
It follows from Proposition \ref{lpLc} that
\begin{align*}
\left|\sum_{j=1}^kc_l\left\langle U_{\overline{x}_j, \Lambda}^{q-1}\overline{Z}_{lj}, \frac{\partial \phi_1}{\partial \Lambda}\right\rangle
\right|
\leq Ck|c_l|\|\phi_1\|_{*}
\leq \frac{Ck}{1+\delta_{l2}r}\|\phi_1\|_{*}^2.
\end{align*}
We have
\begin{align*}
\left\langle I'_{1}(W_1+\phi_1,W_2+\phi_2), \frac{\partial W_1}{\partial \Lambda}\right\rangle
=&\int_{\mathbb{R}^N}D(W_1+\phi_1)D\frac{\partial W_1}{\partial \Lambda}
-\int_{\mathbb{R}^N}K_2\left(\frac{|y|}{\mu}\right)(W_1+\phi_1)^{q}\frac{\partial W_1}{\partial \Lambda}
\nonumber\\
=&\int_{\mathbb{R}^N}DW_1D\frac{\partial W_1}{\partial \Lambda}
-\int_{\mathbb{R}^N}K_2\left(\frac{|y|}{\mu}\right)(W_1+\phi_1)^{q}\frac{\partial W_1}{\partial \Lambda}
\nonumber\\
=&\int_{\mathbb{R}^N}DW_1D\frac{\partial W_1}{\partial \Lambda}
-\Bigg\{\int_{\mathbb{R}^N}K_2\left(\frac{|y|}{\mu}\right)W_1^{q}\frac{\partial W_1}{\partial \Lambda}
\nonumber\\
&+q\int_{\mathbb{R}^N}K_2\left(\frac{|y|}{\mu}\right)W_1^{q-1}\frac{\partial W_1}{\partial \Lambda}\phi_1
+O\left(\int_{\mathbb{R}^N}\phi_1^{q+1}\right)
\Bigg\}.
\end{align*}
Since $(\phi_1,\phi_2)\in \mathbb{E}$, we have
\begin{align*}
\int_{\mathbb{R}^N}K_2\left(\frac{|y|}{\mu}\right)W_1^{q-1}\frac{\partial W_1}{\partial \Lambda}\phi_1
=&\int_{\mathbb{R}^N}K_2\left(\frac{|y|}{\mu}\right)\left[W_1^{q-1}\frac{\partial W_1}{\partial \Lambda}
-\sum_{j=1}^k\left(U_{\overline{x}_j, \Lambda}^{q-1}\overline{Z}_{3j}
+U_{\underline{x}_j, \Lambda}^{q-1}\underline{Z}_{3j}
\right)
\right]\phi_1
\nonumber\\
&+\sum_{j=1}^k\int_{\mathbb{R}^N}\left(K_2\left(\frac{|y|}{\mu}\right)-1\right)
\left(U_{\overline{x}_j, \Lambda}^{q-1}\overline{Z}_{3j}
+U_{\underline{x}_j, \Lambda}^{q-1}\underline{Z}_{3j}
\right)\phi_1
\nonumber\\
=&2k\int_{\Omega_1^+}K_2\left(\frac{|y|}{\mu}\right)\left[W_1^{q-1}\frac{\partial W_1}{\partial \Lambda}
-\sum_{j=1}^k\left(U_{\overline{x}_j, \Lambda}^{q-1}\overline{Z}_{3j}
+U_{\underline{x}_j, \Lambda}^{q-1}\underline{Z}_{3j}
\right)
\right]\phi_1
\nonumber\\
&+k\int_{\mathbb{R}^N}\left(K_2\left(\frac{|y|}{\mu}\right)-1\right)
\left(U_{\overline{x}_1, \Lambda}^{q-1}\overline{Z}_{31}
+U_{\underline{x}_1, \Lambda}^{q-1}\underline{Z}_{31}
\right)\phi_1,
\end{align*}
where
\begin{align*}
&\left|\int_{\Omega_1^+}K_2\left(\frac{|y|}{\mu}\right)\left[W_1^{q-1}\frac{\partial W_1}{\partial \Lambda}
-\sum_{j=1}^k\left(U_{\overline{x}_j, \Lambda}^{q-1}\overline{Z}_{3j}
+U_{\underline{x}_j, \Lambda}^{q-1}\underline{Z}_{3j}
\right)
\right]\phi_1
\right|
\nonumber\\
\leq &C\int_{\Omega_1^+}\left(U_{\overline{x}_1, \Lambda}^{q-1}\left(\sum_{j=2}^kU_{\overline{x}_j, \Lambda}+\sum_{j=1}^kU_{\underline{x}_j, \Lambda}\right)
+\left(\sum_{j=2}^kU_{\overline{x}_j, \Lambda}^q+\sum_{j=1}^kU_{\underline{x}_j, \Lambda}^q\right)
\right)|\phi_1|
\nonumber\\
\leq &C\|\phi_1\|_{*}\int_{\Omega_1^+}\left(U_{\overline{x}_1, \Lambda}^{q-1}\left(\sum_{j=2}^kU_{\overline{x}_j, \Lambda}+\sum_{j=1}^kU_{\underline{x}_j, \Lambda}\right)
+\left(\sum_{j=2}^kU_{\overline{x}_j, \Lambda}^q+\sum_{j=1}^kU_{\underline{x}_j, \Lambda}^q\right)
\right)
\nonumber\\
&\sum_{j=1}^k\Bigg\{\frac{1}{(1+|y-\overline{x}_j|)^{\frac{N-2}{2}+\tau}}
+\frac{1}{(1+|y-\underline{x}_j|)^{\frac{N-2}{2}+\tau}}
\Bigg\}
\nonumber\\
\leq &C\left(\frac{k}{\mu}\right)^{\alpha}\|\phi_1\|_{*}.
\end{align*}
In fact, since
\begin{align*}
&\int_{\Omega_1^+}U_{\overline{x}_1, \Lambda}^{q-1}\left(\sum_{j=2}^kU_{\overline{x}_j, \Lambda}+\sum_{j=1}^kU_{\underline{x}_j, \Lambda}\right)
\sum_{j=1}^k\Bigg\{\frac{1}{(1+|y-\overline{x}_j|)^{\frac{N-2}{2}+\tau}}
+\frac{1}{(1+|y-\underline{x}_j|)^{\frac{N-2}{2}+\tau}}
\Bigg\}
\nonumber\\
=&\int_{\Omega_1^+}U_{\overline{x}_1, \Lambda}^{q-1}\left(\sum_{j=2}^kU_{\overline{x}_j, \Lambda}+\sum_{j=1}^kU_{\underline{x}_j, \Lambda}\right)
\frac{1}{(1+|y-\overline{x}_1|)^{\frac{N-2}{2}+\tau}}
\nonumber\\
&+\int_{\Omega_1^+}U_{\overline{x}_1, \Lambda}^{q-1}\left(\sum_{j=2}^kU_{\overline{x}_j, \Lambda}+\sum_{j=1}^kU_{\underline{x}_j, \Lambda}\right)
\Bigg\{\sum_{j=2}^k\frac{1}{(1+|y-\overline{x}_j|)^{\frac{N-2}{2}+\tau}}
+\sum_{j=1}^k\frac{1}{(1+|y-\underline{x}_j|)^{\frac{N-2}{2}+\tau}}
\Bigg\}
\nonumber\\
\leq &\int_{\Omega_1^+}\frac{C}{(1+|y-\overline{x}_1|)^{q(N-2)-\alpha+\frac{N-2}{2}+\tau}}\left(\frac{k}{\mu}\right)^{\alpha}
\nonumber\\
&+\int_{\Omega_1^+}\frac{C}{(1+|y-\overline{x}_1|)^{(q-1)(N-2)-\alpha}}\left(\frac{k}{\mu}\right)^{\alpha}
\Bigg\{\sum_{j=2}^k\frac{1}{(1+|y-\overline{x}_j|)^{\frac{N-2}{2}+\tau}}
+\sum_{j=1}^k\frac{1}{(1+|y-\underline{x}_j|)^{\frac{N-2}{2}+\tau}}
\Bigg\}
\nonumber\\
\leq &\int_{\Omega_1^+}\frac{C}{(1+|y-\overline{x}_1|)^{q(N-2)-\alpha+\frac{N-2}{2}+\tau}}\left(\frac{k}{\mu}\right)^{\alpha}
+\int_{\Omega_1^+}\frac{C}{(1+|y-\overline{x}_1|)^{q(N-2)+\frac{N-2}{2}+\tau-\alpha-\beta}}\left(\frac{k}{\mu}\right)^{\alpha+\tilde{\alpha}}
\nonumber\\
=&O\left(\left(\frac{k}{\mu}\right)^{\alpha}\right),
\end{align*}
where $\alpha\in (1, N-2)$ and $\tilde{\alpha}\in (1,\frac{N-2}{2}+\tau)$.

In an analogy way, we get
\begin{align*}
\int_{\Omega_1^+}\left(\sum_{j=2}^kU_{\overline{x}_j, \Lambda}^q+\sum_{j=1}^kU_{\underline{x}_j, \Lambda}^q\right)
\sum_{j=1}^k\Bigg\{\frac{1}{(1+|y-\overline{x}_j|)^{\frac{N-2}{2}+\tau}}
+\frac{1}{(1+|y-\underline{x}_j|)^{\frac{N-2}{2}+\tau}}
\Bigg\}
=O\left(\left(\frac{k}{\mu}\right)^{\alpha}\right).
\end{align*}
Moreover, similar to the proof in Proposition \ref{lpL}, we have
\begin{align*}
&\left|\int_{\mathbb{R}^N}\left(K_2\left(\frac{|y|}{\mu}\right)-1\right)
\left(U_{\overline{x}_1, \Lambda}^{q-1}\overline{Z}_{31}
+U_{\underline{x}_1, \Lambda}^{q-1}\underline{Z}_{31}
\right)\phi_1\right|
\nonumber\\
\leq &\left|\int_{||y|-\mu|\leq \sqrt{\mu}}\left(K_2\left(\frac{|y|}{\mu}\right)-1\right)
\left(U_{\overline{x}_1, \Lambda}^{q-1}\overline{Z}_{31}
+U_{\underline{x}_1, \Lambda}^{q-1}\underline{Z}_{31}
\right)\phi_1\right|
\nonumber\\
&+\left|\int_{||y|-\mu|\geq \sqrt{\mu}}\left(K_2\left(\frac{|y|}{\mu}\right)-1\right)
\left(U_{\overline{x}_1, \Lambda}^{q-1}\overline{Z}_{31}
+U_{\underline{x}_1, \Lambda}^{q-1}\underline{Z}_{31}
\right)\phi_1\right|
\leq \frac{C}{\mu^{\frac{m}{2}}}\|\phi_1\|_{*}.
\end{align*}
Thus, we have proved that let $\alpha=\frac{N-2}{2}$,
\begin{align*}
\frac{\partial F(r,h,\Lambda)}{\partial \Lambda}
=&\frac{\partial I(W_1,W_2)}{\partial \Lambda}
+kO\left(\frac{1}{\mu^{\sigma}}\|\phi_1\|_{*}\right)
+kO\left(\left(\frac{k}{\mu}\right)^{\alpha}\|\phi_1\|_{*}\right)
+kO\left(\frac{1}{1+\delta_{l2}r}\|\phi_1\|_{*}^2\right)
\nonumber\\
=&\frac{\partial I(W_1,W_2)}{\partial \Lambda}
+kO\left(\frac{1}{k^{\frac{m(N-2)}{N-2-m}+\theta}}\right).
\end{align*}
\end{proof}

\begin{proposition}\label{fh}
We have
\begin{align*}
\frac{\partial F(r,h,\Lambda)}{\partial h}
=&-\frac{k}{\Lambda^{N-2}}\left[\frac{(N-2)B_4k^{N-2}h}{r^{N-2}(\sqrt{1-h^2})^{N}}
-\frac{(N-3)B_5k}{r^{N-2}h^{N-2}\sqrt{1-h^2}}\right]
\nonumber\\
&+kO\left(\frac{1}{k^{\frac{m(N-2)}{N-2-m}+\frac{N-3}{N-1}+\theta}}\right),
\end{align*}
where $B_4$, $B_5$, $m$, $\theta$ are positive constants defined in Proposition \ref{iw}.
\end{proposition}
\begin{proof}
By the definition of $F(r,h,\Lambda)$, we obtain
\begin{align}\label{fh1}
\frac{\partial F(r,h,\Lambda)}{\partial h}
=&\left\langle I'_1(W_1+\phi_1, W_2+\phi_2), \frac{\partial (W_1+\phi_1)}{\partial h} \right\rangle
+\left\langle I'_2(W_1+\phi_1, W_2+\phi_2), \frac{\partial (W_2+\phi_2)}{\partial h} \right\rangle
\nonumber\\
=&\left\langle I'_1(W_1+\phi_1, W_2+\phi_2), \frac{\partial W_1}{\partial h} \right\rangle
+\left\langle I'_1(W_1+\phi_1, W_2+\phi_2), \frac{\partial \phi_1}{\partial h} \right\rangle
\nonumber\\
&+\left\langle I'_2(W_1+\phi_1, W_2+\phi_2), \frac{\partial W_2}{\partial h} \right\rangle
+\left\langle I'_2(W_1+\phi_1, W_2+\phi_2), \frac{\partial \phi_2}{\partial h} \right\rangle
\nonumber\\
=&\left\langle I'_1(W_1+\phi_1, W_2+\phi_2), \frac{\partial W_1}{\partial h} \right\rangle
+\sum_{l=1}^k\sum_{j=1}^kc_l\left\langle U_{\overline{x}_j,\Lambda}^{q-1}\overline{Z}_{lj}+U_{\underline{x}_j,\Lambda}^{q-1}\underline{Z}_{lj},
\frac{\partial \phi_1}{\partial h} \right\rangle
\nonumber\\
&+\left\langle I'_2(W_1+\phi_1, W_2+\phi_2), \frac{\partial W_2}{\partial h} \right\rangle
+\sum_{l=1}^k\sum_{j=1}^kc_l\left\langle V_{\overline{x}_j,\Lambda}^{p-1}\overline{Y}_{lj}+V_{\underline{x}_j,\Lambda}^{p-1}\underline{Y}_{lj},
\frac{\partial \phi_2}{\partial h} \right\rangle.
\end{align}
Since
\begin{align*}
\int_{\mathbb{R}^N}U_{\overline{x}_j,\Lambda}^{q-1}\overline{Z}_{lj}\phi_1=0,\
\int_{\mathbb{R}^N}U_{\underline{x}_j,\Lambda}^{q-1}\underline{Z}_{lj}\phi_1=0,\
\int_{\mathbb{R}^N}V_{\overline{x}_j,\Lambda}^{p-1}\overline{Y}_{lj}\phi_2=0,\
\int_{\mathbb{R}^N}V_{\underline{x}_j,\Lambda}^{p-1}\underline{Y}_{lj}\phi_2=0,\
\end{align*}
we have
\begin{align*}
\left\langle U_{\overline{x}_j,\Lambda}^{q-1}\overline{Z}_{lj},
\frac{\partial \phi_1}{\partial h} \right\rangle
=-\left\langle \frac{\partial \left(U_{\overline{x}_j,\Lambda}^{q-1}\overline{Z}_{lj}\right)}{\partial h},
\phi_1 \right\rangle,\
\left\langle U_{\underline{x}_j,\Lambda}^{q-1}\underline{Z}_{lj},
\frac{\partial \phi_1}{\partial h} \right\rangle
=-\left\langle \frac{\partial \left(U_{\underline{x}_j,\Lambda}^{q-1}\underline{Z}_{lj}\right)}{\partial h},
\phi_1 \right\rangle, \\
\left\langle V_{\overline{x}_j,\Lambda}^{p-1}\overline{Y}_{lj},
\frac{\partial \phi_2}{\partial h} \right\rangle
=-\left\langle \frac{\partial \left(V_{\overline{x}_j,\Lambda}^{p-1}\overline{Y}_{lj}\right)}{\partial h},
\phi_2 \right\rangle,\
\left\langle V_{\underline{x}_j,\Lambda}^{p-1}\underline{Y}_{lj},
\frac{\partial \phi_2}{\partial h}  \right\rangle
=-\left\langle \frac{\partial \left(V_{\underline{x}_j,\Lambda}^{p-1}\underline{Y}_{lj}\right)}{\partial h},
\phi_2 \right\rangle.
\end{align*}
Then for $l=1,2,3$,
\begin{align}\label{fh2}
&\sum_{j=1}^kc_l\left\langle U_{\overline{x}_j,\Lambda}^{q-1}\overline{Z}_{lj}+U_{\underline{x}_j,\Lambda}^{q-1}\underline{Z}_{lj},
\frac{\partial \phi_1}{\partial h} \right\rangle
\nonumber\\
\leq &C|c_l|\|\phi_1\|_{*}\sum_{j=1}^k\int_{\mathbb{R}^N}\frac{\partial \left(U_{\overline{x}_j,\Lambda}^{q-1}\overline{Z}_{lj}\right)}{\partial h}
\left(\sum_{j=1}^k\frac{1}{(1+|y-\overline{x}_j|)^{\frac{N-2}{2}+\tau}}
+\sum_{j=1}^k\frac{1}{(1+|y-\underline{x}_j|)^{\frac{N-2}{2}+\tau}}\right)
\nonumber\\
\leq &C|c_l|\|\phi_1\|_{*}\sum_{j=1}^k\int_{\mathbb{R}^N}\frac{r(1+\delta_{l2}r)}{(1+|y-\overline{x}_j|)^{q(N-2)+1}}
\left(\sum_{j=1}^k\frac{1}{(1+|y-\overline{x}_j|)^{\frac{N-2}{2}+\tau}}
+\sum_{j=1}^k\frac{1}{(1+|y-\underline{x}_j|)^{\frac{N-2}{2}+\tau}}\right)
\nonumber\\
\leq &Cr|c_l|\|\phi_1\|_{*}.
\end{align}
Moreover, when $q\geq \frac{N+2}{N-2}$, because of $(\phi_1,\phi_2)\in \mathbb{E}$, we have
\begin{align}\label{fh3}
&\left\langle I'_1(W_1+\phi_1, W_2+\phi_2), \frac{\partial W_1}{\partial h} \right\rangle
\nonumber\\
=&\int_{\mathbb{R}^N}D(W_1+\phi_1)D\frac{\partial W_1}{\partial h}
-\int_{\mathbb{R}^N}K_2\left(\frac{|y|}{\mu}\right)(W_1+\phi_1)^{q}\frac{\partial W_1}{\partial h}
\nonumber\\
=&\int_{\mathbb{R}^N}DW_1D\frac{\partial W_1}{\partial h}
-\int_{\mathbb{R}^N}K_2\left(\frac{|y|}{\mu}\right)(W_1+\phi_1)^{q}\frac{\partial W_1}{\partial h}
\nonumber\\
=&\frac{\partial I(W_1,0)}{\partial h}
+q\int_{\mathbb{R}^N}K_2\left(\frac{|y|}{\mu}\right)W_1^{q-1}\frac{\partial W_1}{\partial h}\phi_1
+O\left(\int_{\mathbb{R}^N}\left|\phi_1\right|^{q+1}\right).
\end{align}
One has
\begin{align*}
\int_{\mathbb{R}^N}K_2\left(\frac{|y|}{\mu}\right)W_1^{q-1}\frac{\partial W_1}{\partial h}\phi_1
=&\int_{\mathbb{R}^N}K_2\left(\frac{|y|}{\mu}\right)\left(W_1^{q-1}\frac{\partial W_1}{\partial h}
-\sum_{j=1}^k\left(U_{\overline{x}_j,\Lambda}^{q-1}\overline{Z}_{2j}+U_{\underline{x}_j,\Lambda}^{q-1}\underline{Z}_{2j}
\right)
\right)\phi_1
\nonumber\\
&+\sum_{j=1}^k\int_{\mathbb{R}^N}\left\{K_2\left(\frac{|y|}{\mu}\right)-1\right\}
\left(U_{\overline{x}_j,\Lambda}^{q-1}\overline{Z}_{2j}+U_{\underline{x}_j,\Lambda}^{q-1}\underline{Z}_{2j}
\right)\phi_1
\nonumber\\
=&2k\int_{\Omega_1^+}K_2\left(\frac{|y|}{\mu}\right)\left(W_1^{q-1}\frac{\partial W_1}{\partial h}
-\sum_{j=1}^k\left(U_{\overline{x}_j,\Lambda}^{q-1}\overline{Z}_{2j}+U_{\underline{x}_j,\Lambda}^{q-1}\underline{Z}_{2j}
\right)
\right)\phi_1
\nonumber\\
&+2k\int_{\mathbb{R}^N}\left\{K_2\left(\frac{|y|}{\mu}\right)-1\right\}
U_{\overline{x}_1,\Lambda}^{q-1}\overline{Z}_{21}\phi_1,
\end{align*}
where
\begin{align*}
&\int_{\Omega_1^+}K_2\left(\frac{|y|}{\mu}\right)\left(W_1^{q-1}\frac{\partial W_1}{\partial h}
-\sum_{j=1}^k\left(U_{\overline{x}_j,\Lambda}^{q-1}\overline{Z}_{2j}+U_{\underline{x}_j,\Lambda}^{q-1}\underline{Z}_{2j}
\right)
\right)\phi_1
\nonumber\\
\leq &C\int_{\Omega_1^+}\left[U_{\overline{x}_1,\Lambda}^{q-1}\left(\sum_{j=2}^k\overline{Z}_{2j}+\sum_{j=1}^k\underline{Z}_{2j}\right)
+\left(\sum_{j=2}^kU_{\overline{x}_j,\Lambda}^{q-1}\overline{Z}_{2j}+\sum_{j=1}^kU_{\underline{x}_j,\Lambda}^{q-1}\underline{Z}_{2j}\right)
\right]\phi_1
\nonumber\\
\leq &C\left(\frac{k}{\mu}\right)^{p(N-2)-\frac{N+2}{2}-\tau}\int_{\Omega_1^+}
\frac{r}{(1+|y-\overline{x}_1|)^{(q-p)(N-2)+\frac{N}{2}+2+\tau}}\phi_1
\nonumber\\
\leq &C\left(\frac{k}{\mu}\right)^{p(N-2)-\frac{N+2}{2}-\tau}\|\phi_1\|_{*}
\int_{\Omega_1^+}
\frac{r}{(1+|y-\overline{x}_1|)^{(q-p)(N-2)+\frac{N}{2}+2+\tau}}
\nonumber\\
&\times \sum_{j=1}^k\left(\frac{1}{(1+|y-\overline{x}_j|)^{\frac{N-2}{2}+\tau}}
+\frac{1}{(1+|y-\underline{x}_j|)^{\frac{N-2}{2}+\tau}}
\right)
\nonumber\\
\leq &Cr\left(\frac{k}{\mu}\right)^{p(N-2)-\frac{N+2}{2}-\tau}\|\phi_1\|_{*}.
\end{align*}
In an analogy way, we get
\begin{align*}
\int_{\mathbb{R}^N}\left\{K_2\left(\frac{|y|}{\mu}\right)-1\right\}
U_{\overline{x}_1,\Lambda}^{q-1}\overline{Z}_{21}\phi_1
=O\left(r\left(\frac{k}{\mu}\right)^{p(N-2)-\frac{N+2}{2}-\tau}\|\phi_1\|_{*}\right).
\end{align*}
Thus,
\begin{align}\label{fh4}
\int_{\mathbb{R}^N}K_2\left(\frac{|y|}{\mu}\right)W_1^{q-1}\frac{\partial W_1}{\partial h}\phi_1
=O\left(rk\left(\frac{k}{\mu}\right)^{p(N-2)-\frac{N+2}{2}-\tau}\|\phi_1\|_{*}\right).
\end{align}
Combining (\ref{fh3}) and (\ref{fh4}), we find
\begin{align}\label{fh5}
\left\langle I'_1(W_1+\phi_1, W_2+\phi_2), \frac{\partial W_1}{\partial h} \right\rangle
=\frac{\partial I(W_1,0)}{\partial h}
+O\left(rk\left(\frac{k}{\mu}\right)^{p(N-2)-\frac{N+2}{2}-\tau}\|\phi_1\|_{*}\right).
\end{align}
$\left\langle I'_1(W_1+\phi_1, W_2+\phi_2), \frac{\partial W_2}{\partial h} \right\rangle$ can be proved in the same way. Consequently,
\begin{align*}
\frac{\partial F(r,h,\Lambda)}{\partial h}
=\frac{\partial I(W_1,W_2)}{\partial h}
+kO\left(r\left(\frac{k}{\mu}\right)^{p(N-2)-\frac{N+2}{2}-\tau}\|\phi_1\|_{*}\right)
=\frac{\partial I(W_1,W_2)}{\partial h}
+kO\left(r\|\phi_1\|_{*}^2\right).
\end{align*}
It follows from
$p\geq \frac{N+2}{2(N-2)}+\frac{m+1}{2m}+\frac{(N-3)(N-2-m)}{2m(N-1)(N-2)}+\frac{N-2-m}{2m(N-2)}+\frac{\tau}{N-2}$
, $\beta\geq \frac{m+1}{2}+\frac{(N-3)(N-2-m)}{2(N-1)(N-2)}$ and Proposition \ref{iwh} that
\begin{align*}
\frac{\partial F(r,h,\Lambda)}{\partial h}
=\frac{\partial I(W_1,W_2)}{\partial h}
+kO\left(\frac{1}{k^{\frac{m(N-2)}{N-2-m}+\frac{N-3}{N-1}+\theta}}\right).
\end{align*}
\end{proof}

Similar to \cite{duanmw}, we have
\begin{align*}
&\frac{B_4k^{N-2}}{(r\sqrt{1-h^2})^{N-2}}+\frac{B_5k}{r^{N-2}h^{N-3}\sqrt{1-h^2}}
\nonumber\\
=&\frac{B_4}{k^{\frac{m(N-2)}{N-2-m}}}
+\frac{B_6}{k^{\frac{m(N-2)}{N-2-m}+\frac{2(N-3)}{N-1}}}
+\frac{C(r,h)}{k^{\frac{m(N-2)}{N-2-m}+\theta}}
+\frac{B_7}{k^{\frac{m(N-2)}{N-2-m}+\frac{2(N-3)}{N-1}}}(1-\mathcal{H}^{-1}h)^2
\nonumber\\
&+O\left(\frac{1}{k^{\frac{m(N-2)}{N-2-m}+\frac{2(N-3)}{N-1}}}\right)(1-\mathcal{H}^{-1}h)^3,
\end{align*}
where
\begin{align}\label{b6}
B_6=\frac{(N-2)B_4B'^2}{2},
\end{align}
\begin{align}\label{b7}
B_7=\frac{N-2}{2}\left[B_4B'^2+\frac{(N-3)B_5}{B'^{N-3}}\right],
\end{align}
\begin{align}\label{hb}
\mathcal{H}=\frac{B'}{k^{\frac{N-3}{N-1}}}, \text{ with } B'=\left(\frac{(N-3)B_5}{(N-2)B_4}\right)^{\frac{1}{N-1}}.
\end{align}
Then we rewrite the expansion of the energy functional as shown below
\begin{align}\label{newf}
F(r,h,\Lambda)
=&kA_1-\frac{k}{\Lambda^{N-2}}\left[\frac{B_4}{k^{\frac{m(N-2)}{N-2-m}}}
+\frac{B_6}{k^{\frac{m(N-2)}{N-2-m}+\frac{2(N-3)}{N-1}}}
+\frac{B_7}{k^{\frac{m(N-2)}{N-2-m}+\frac{2(N-3)}{N-1}}}(1-\mathcal{H}^{-1}h)^2
\right]
\nonumber\\
&+k\left[\frac{A_2}{k^{\frac{m_2(N-2)}{N-2-m}}\Lambda^{m_2}}
+\frac{\bar{A}_2}{k^{\frac{m_1(N-2)}{N-2-m}}\Lambda^{m_1}}\right]
+k\left[\frac{A_3}{k^{\frac{m_2(N-2)}{N-2-m}}\Lambda^{m_2-2}}
+\frac{\bar{A}_3}{k^{\frac{m_1(N-2)}{N-2-m}}\Lambda^{m_1-2}}\right](\mu-r)^2
\nonumber\\
&+k\frac{C(r,\Lambda)}{k^{\frac{m(N-2)}{N-2-m}}}(\mu-r)^{\theta+2}
+kO\left(\frac{1}{k^{\frac{m(N-2)}{N-2-m}+\frac{2(N-3)}{N-1}+\theta}}\right)
+kO\left(\frac{1}{k^{\frac{m(N-2)}{N-2-m}+\frac{2(N-3)}{N-1}}}\right)(1-\mathcal{H}^{-1}h)^3,
\end{align}
\begin{align}\label{newflamda}
\frac{\partial F(r,h,\Lambda)}{\partial \Lambda}
=&k\left[\frac{(N-2)B_4}{\Lambda^{N-1}k^{\frac{m(N-2)}{N-2-m}}}
-\frac{m_2A_2}{\Lambda^{m_2+1}k^{\frac{m_2(N-2)}{N-2-m}}}
-\frac{m_1\bar{A}_2}{\Lambda^{m_1+1}k^{\frac{m_1(N-2)}{N-2-m}}}
\right]
\nonumber\\
&-k\left[\frac{A_3(m_2-2)}{\mu^{m_2}\Lambda^{m_2-1}}+\frac{\bar{A}_3(m_1-2)}{\mu^{m_1}\Lambda^{m_1-1}}\right](\mu-r)^2
+kO\left(\frac{1}{k^{\frac{m(N-2)}{N-2-m}}}(\mu-r)^{\theta+2}\right),
\end{align}
and
\begin{align}\label{newfh}
\frac{\partial F(r,h,\Lambda)}{\partial h}
=&\frac{k}{\Lambda^{N-2}}\left[\frac{2B_7}{k^{\frac{m(N-2)}{N-2-m}+\frac{N-3}{N-1}}}(1-\mathcal{H}^{-1}h)
\right]
+kO\left(\frac{1}{k^{\frac{m(N-2)}{N-2-m}+\frac{N-3}{N-1}}}\right)(1-\mathcal{H}^{-1}h)^2
\nonumber\\
&+kO\left(\frac{1}{k^{\frac{m(N-2)}{N-2-m}+\frac{N-3}{N-1}+\theta}}\right).
\end{align}

If $m=m_1<m_2$, we set $\Lambda_0$ to be the solution of
\begin{align*}
\frac{-m\bar{A}_2}{\Lambda_0^{m+1}}+\frac{(N-2)B_4}{\Lambda_0^{N-1}}=0,
\end{align*}
then
\begin{align*}
\Lambda_0=\left(\frac{(N-2)B_4}{m\bar{A}_2}\right)^{\frac{1}{N-2-m}}.
\end{align*}

If $m=m_2<m_1$, we set $\Lambda_0$ to be the solution of
\begin{align*}
\frac{-m{A}_2}{\Lambda_0^{m+1}}+\frac{(N-2)B_4}{\Lambda_0^{N-1}}=0,
\end{align*}
then
\begin{align*}
\Lambda_0=\left(\frac{(N-2)B_4}{m{A}_2}\right)^{\frac{1}{N-2-m}}.
\end{align*}

If $m=m_1=m_2$, we set $\Lambda_0$ to be the solution of
\begin{align*}
\frac{-m(\bar{A}_2+{A}_2)}{\Lambda_0^{m+1}}+\frac{(N-2)B_4}{\Lambda_0^{N-1}}=0,
\end{align*}
then
\begin{align*}
\Lambda_0=\left(\frac{(N-2)B_4}{m(\bar{A}_2+{A}_2)}\right)^{\frac{1}{N-2-m}}.
\end{align*}
Therefore,
\begin{align}\label{lamda0}
\Lambda_0=
\begin{cases}
\left(\frac{(N-2)B_4}{m\bar{A}_2}\right)^{\frac{1}{N-2-m}}, \ &m=m_1<m_2,\\
\left(\frac{(N-2)B_4}{m{A}_2}\right)^{\frac{1}{N-2-m}}, \ &m=m_2<m_1,\\
\left(\frac{(N-2)B_4}{m(\bar{A}_2+{A}_2)}\right)^{\frac{1}{N-2-m}}, \ &m=m_1=m_2.
\end{cases}
\end{align}

Now we consider the case of $m=m_1<m_2$ and define
\begin{align*}
\bar{F}(r,h,\Lambda)=-F(r,h,\Lambda),\ (r,h,\Lambda)\in S_k,
\end{align*}
and
\begin{align*}
t_1=k\left[-A_1-\left(\frac{\bar{A}_2}{\Lambda_0^m}-\frac{B_4}{\Lambda_0^{N-2}}\right)k^{\frac{m(N-2)}{N-2-m}}
-k^{\frac{m(N-2)}{N-2-m}+\frac{5}{2}\vartheta_2}\right],\
t_2=k(-A_1+\vartheta_1),
\end{align*}
where $\vartheta_i>0$ small, $i=1,2$ and
\begin{align*}
S_k=\Bigg\{(r,h,\Lambda)|\ &r\in \left[k^{\frac{N-2}{N-2-m}}-\frac{1}{k^{\vartheta_2}},k^{\frac{N-2}{N-2-m}}+\frac{1}{k^{\vartheta_2}}\right],\
\Lambda\in \left[\Lambda_0-\frac{1}{k^{\frac{3}{2}\vartheta_2}},\Lambda_0+\frac{1}{k^{\frac{3}{2}\vartheta_2}}\right],
\nonumber\\
&h\in \left[\frac{B'}{k^{\frac{N-3}{N-1}}}\left(1-\frac{1}{k^{\vartheta_2}}\right),
\frac{B'}{k^{\frac{N-3}{N-1}}}\left(1+\frac{1}{k^{\vartheta_2}}\right)\right]
\Bigg\}.
\end{align*}

Let
\begin{align}\label{barf}
\bar{F}^{\alpha}(r,h,\Lambda)=\left\{(r,h,\Lambda)\in S_k, \bar{F}(r,h,\Lambda)\leq \alpha\right\}.
\end{align}
Consider the following gradient flow system
\begin{align*}
\begin{cases}
\frac{\partial r}{\partial t}=-\bar{F}_r(r,h,\Lambda),\ &t>0;\\
\frac{\partial h}{\partial t}=-\bar{F}_h(r,h,\Lambda),\ &t>0;\\
\frac{\partial \Lambda}{\partial t}=-\bar{F}_{\Lambda}(r,h,\Lambda),\ &t>0.
\end{cases}
\end{align*}
Then
\begin{proposition}\label{barf}
The flow $(r(t), h(t), \Lambda(t))$ does not leave $S_k$ before it reaches $\bar{F}^{t_1}$.
\end{proposition}
\begin{proof}
{\it Case I. $|r-\mu|=\frac{1}{k^{\vartheta_2}}$, $|1-\mathcal{H}^{-1}h|\leq \frac{1}{k^{\vartheta_2}}$,
$|\Lambda-\Lambda_0|\leq \frac{1}{k^{\frac{3}{2}\vartheta_2}}$}, then
\begin{align*}
\frac{B_4}{\Lambda^{N-2}}-\frac{\bar{A}_2}{\Lambda^{m}}
=\left(\frac{B_4}{\Lambda_0^{N-2}}-\frac{\bar{A}_2}{\Lambda_0^{m}}\right)+O\left(|\Lambda-\Lambda_0|^2\right)
=\left(\frac{B_4}{\Lambda_0^{N-2}}-\frac{\bar{A}_2}{\Lambda_0^{m}}\right)+O\left(\frac{1}{k^{3\vartheta_2}}\right).
\end{align*}
In view of (\ref{newf}),
\begin{align}\label{ct1}
\bar{F}(r,h,\Lambda)
=&-kA_1+k\left[\frac{B_4}{\Lambda_0^{N-2}k^{\frac{m(N-2)}{N-2-m}}}
-\frac{\bar{A}_2}{k^{\frac{m(N-2)}{N-2-m}}\Lambda_0^{m}}
\right]
-k\frac{\bar{A}_3}{k^{\frac{m(N-2)}{N-2-m}}\Lambda_0^{m-2}}
+kO\left(\frac{1}{k^{\frac{m(N-2)}{N-2-m}+\frac{5}{2}\vartheta_2}}\right)
\nonumber\\
=&-kA_1+\frac{k}{k^{\frac{m(N-2)}{N-2-m}}\Lambda_0^{m-2}}\left(\frac{B_4}{\Lambda_0^{N-2-m}}
-\bar{A}_2\right)
-k\frac{\bar{A}_3}{k^{\frac{m(N-2)}{N-2-m}}\Lambda_0^{m-2}}
+kO\left(\frac{1}{k^{\frac{m(N-2)}{N-2-m}+\frac{5}{2}\vartheta_2}}\right)
\nonumber\\
<&t_1.
\end{align}

{\it Case II. $|1-\mathcal{H}^{-1}h|=\frac{1}{k^{\vartheta_2}}$, $|r-\mu|\leq \frac{1}{k^{\vartheta_2}}$,
$|\Lambda-\Lambda_0|\leq \frac{1}{k^{\frac{3}{2}\vartheta_2}}$}.

If $1-\mathcal{H}^{-1}h=\frac{1}{k^{\vartheta_2}}$, if follows from (\ref{newfh}) that
\begin{align*}
\frac{\partial \bar{F}(r,h,\Lambda)}{\partial h}
=&-\frac{k}{\Lambda^{N-2}}\left[\frac{2B_7}{k^{\frac{m(N-2)}{N-2-m}+\frac{N-3}{N-1}+\vartheta_2}}
\right]
+kO\left(\frac{1}{k^{\frac{m(N-2)}{N-2-m}+\frac{N-3}{N-1}+2\vartheta_2}}\right)
<0.
\end{align*}
So, the flow does not leave $S_k$.

If $1-\mathcal{H}^{-1}h=-\frac{1}{k^{\vartheta_2}}$, we have
\begin{align*}
\frac{\partial \bar{F}(r,h,\Lambda)}{\partial h}
=&\frac{k}{\Lambda^{N-2}}\left[\frac{2B_7}{k^{\frac{m(N-2)}{N-2-m}+\frac{N-3}{N-1}+\vartheta_2}}
\right]
+kO\left(\frac{1}{k^{\frac{m(N-2)}{N-2-m}+\frac{N-3}{N-1}+2\vartheta_2}}\right)
>0.
\end{align*}
So, the flow does not leave $S_k$.

{\it Case III. $|\Lambda-\Lambda_0|=\frac{1}{k^{\frac{3}{2}\vartheta_2}}$, $|r-\mu|\leq \frac{1}{k^{\vartheta_2}}$, $|1-\mathcal{H}^{-1}h|\leq \frac{1}{k^{\vartheta_2}}$}.

For $\Lambda-\Lambda_0=\frac{1}{k^{\frac{3}{2}\vartheta_2}}$, if follows from (\ref{newflamda}) that there exists a constant $\gamma_1$ such that
\begin{align*}
\frac{\partial \bar{F}(r,h,\Lambda)}{\partial \Lambda}
=&-k\left[\frac{(N-2)B_4}{\Lambda^{N-1}k^{\frac{m(N-2)}{N-2-m}}}
-\frac{m_1\bar{A}_2}{\Lambda^{m_1+1}k^{\frac{m_1(N-2)}{N-2-m}}}
\right]
+k\left[\frac{\bar{A}_3(m_1-2)}{\mu^{m_1}\Lambda^{m_1-1}}\right](\mu-r)^2
\nonumber\\
&+kO\left(\frac{1}{k^{\frac{m(N-2)}{N-2-m}}}(\mu-r)^{\theta+2}\right)
\nonumber\\
=&k\left[\gamma_1\frac{1}{k^{\frac{m(N-2)}{N-2-m}+\frac{3}{2}\vartheta_2}}
+O\left(\frac{1}{k^{\frac{m(N-2)}{N-2-m}+2\vartheta_2}}\right)
\right]
>0.
\end{align*}

Similarly, for $\Lambda-\Lambda_0=-\frac{1}{k^{\frac{3}{2}\vartheta_2}}$, if follows from (\ref{newflamda}) that
\begin{align*}
\frac{\partial \bar{F}(r,h,\Lambda)}{\partial \Lambda}
=k\left[-\gamma_1\frac{1}{k^{\frac{m(N-2)}{N-2-m}+\frac{3}{2}\vartheta_2}}
+O\left(\frac{1}{k^{\frac{m(N-2)}{N-2-m}+2\vartheta_2}}\right)
\right]
<0.
\end{align*}
Therefore, the flow does not leave $S_k$ when $|\Lambda-\Lambda_0|=\frac{1}{k^{\frac{3}{2}\vartheta_2}}$. Combining above results, that conclude the proof.
\end{proof}

Based on the above Proposition \ref{barf}, we are ready to prove the Theorem \ref{thm1.1}.

\noindent
{\it Proof of Theorem \ref{thm1.1}:} We only need to prove that the function $\bar{F}(r,h,\Lambda)$ has a critical point in $S_k$, and thus ${F}(r,h,\Lambda)$ has a critical point in $S_k$.
Define
\begin{align*}
\Gamma=\Bigg\{&g(r,h,\Lambda)=(g_1(r,h,\Lambda),g_2(r,h,\Lambda),g_3(r,h,\Lambda))\in S_k,
(r,h,\Lambda)\in S_k,
\nonumber\\
&g(r,h,\Lambda)=(r,h,\Lambda) \text{ if } |\mu-r|=\frac{1}{k^{\vartheta_2}}
\Bigg\}.
\end{align*}
Let
\begin{align*}
c=\inf_{g\in \Gamma}\max_{(r,h,\Lambda)\in S_k}\bar{F}(r,h,\Lambda).
\end{align*}
Similar to \cite{weiy}, if we prove that
\begin{align}\label{crit1}
t_1<c<t_2,
\end{align}
and
\begin{align}\label{crit2}
\sup_{|\mu-r|=\frac{1}{k^{\vartheta_2}}}\bar{F}(r,h,\Lambda)<t_1, \ \forall g\in \Gamma.
\end{align}
Then $c$ is a critical value of $\bar{F}(r,h,\Lambda)$ and can be achieved by some $(r,h,\Lambda)\in S_k$.

In fact, by the definition of $c$ and $t_2$, we get
\begin{align*}
c<t_2.
\end{align*}
For $\forall g=(g_1,g_2,g_3)\in \Gamma$, we find if $|\mu-r|=\frac{1}{k^{\vartheta_2}}$,
$g_1(r,h,\Lambda)=r$.

Define
\begin{align*}
\tilde{g}_1(r)=g_1(r,h,\Lambda_0),
\end{align*}
then $\tilde{g}_1(r)=r$ when $|\mu-r|=\frac{1}{k^{\vartheta_2}}$. \par
Furthermore, we deduce that there exists
$\Bar{r}\in \left(\mu-\frac{1}{k^{\vartheta_2}},\mu+\frac{1}{k^{\vartheta_2}}\right)$ such that
$\tilde{g}_1(\bar{r})=\mu$.

Let $\tilde{g}_2(h)=g_2(\bar{r},h,\Lambda_0)$, then when
$|\mu-r|=\frac{1}{k^{\vartheta_2}}$ and $\bar{h}\in \left(\frac{B'}{k^{\frac{N-3}{N-1}}}\left(1-\frac{1}{k^{\vartheta_2}}\right),
\frac{B'}{k^{\frac{N-3}{N-1}}}\left(1+\frac{1}{k^{\vartheta_2}}\right)\right)$, w have
\begin{align*}
\tilde{g}_2(\bar{h})=\bar{h}.
\end{align*}
Set $\bar{\Lambda}=g_3(\bar{r},\bar{h},\Lambda_0)$, then
\begin{align*}
\max_{(r,h,\Lambda)\in S_k}\bar{F}(g(r,h,\Lambda))
\geq &\bar{F}(g(r,h,\Lambda_0))
\nonumber\\
=&\bar{F}(r,\bar{h},\bar{\Lambda})
\nonumber\\
=&-kA_1+\frac{k}{\bar{\Lambda}^{N-2}}\left[\frac{B_4}{k^{\frac{m(N-2)}{N-2-m}}}
+\frac{B_6}{k^{\frac{m(N-2)}{N-2-m}+\frac{2(N-3)}{N-1}}}
\right]
-k\frac{\bar{A}_2}{k^{\frac{m_1(N-2)}{N-2-m}}\bar{\Lambda}^{m_1}}
\nonumber\\
&+kO\left(\frac{1}{k^{\frac{m(N-2)}{N-2-m}+\frac{2(N-3)}{N-1}+\theta}}\right)
>t_1.
\end{align*}
Consequently, (\ref{crit1}) holds. We next prove (\ref{crit2}). For $\forall g\in \Gamma$ and $\forall \bar{r}$ with $|\bar{r}-\mu|=\frac{1}{k^{\vartheta_2}}$, we have
\begin{align*}
g(\bar{r},\bar{h},\Lambda)=(\bar{r},\bar{h},\tilde{\Lambda}) \ \text{ for some } \tilde{\Lambda}.
\end{align*}
In terms of (\ref{ct1}), we have
\begin{align*}
\bar{F}(g(r,h,\Lambda))=\bar{F}(\bar{r},\bar{h},\tilde{\Lambda})<t_1.
\end{align*}
That concludes proof.

\section{Appendix}\label{appendix}

Now we first give the following lemmas which are useful in the previous estimates.

\begin{lemma}\label{a1}
As $k\rightarrow \infty$, it holds
\begin{align*}
\int_{\mathbb{R}^N}U_{\overline{x}_1,\Lambda}^qU_{\overline{x}_i,\Lambda}
=\frac{B_0}{\Lambda^{N-2}|\overline{x}_1-\overline{x}_i|^{N-2}}
+O\left(\frac{1}{|\overline{x}_1-\overline{x}_i|^{N-1}}\right),
\end{align*}
and
\begin{align*}
\int_{\mathbb{R}^N}U_{\overline{x}_1,\Lambda}^qU_{\underline{x}_i,\Lambda}
=\frac{B_0}{\Lambda^{N-2}|\overline{x}_1-\underline{x}_i|^{N-2}}
+O\left(\frac{1}{|\overline{x}_1-\underline{x}_i|^{N-1}}\right),
\end{align*}
where $B_0=\int_{\mathbb{R}^N}U_{0,1}^q(z)dz$.
\end{lemma}
\begin{proof}
Let $\overline{d}_i=|\overline{x}_1-\overline{x}_i|$, $\underline{d}_i=|\overline{x}_1-\underline{x}_i|$ for $i=1, \cdots, k$, then
\begin{align}\label{a10}
\int_{\mathbb{R}^N}U_{\overline{x}_1,\Lambda}^qU_{\overline{x}_i,\Lambda}
=&\int_{\mathbb{R}^N}\Lambda^NU_{0,1}^q(\Lambda(y-\overline{x}_1))U_{0,1}(\Lambda(y-\overline{x}_i))
\nonumber\\
=&\left\{\int_{B_{\frac{\overline{d}_i}{4}}(\overline{x}_1)}+\int_{\mathbb{R}^N\setminus B_{\frac{\overline{d}_i}{4}}(\overline{x}_1)}\right\}
\Lambda^NU_{0,1}^q(\Lambda(y-\overline{x}_1))U_{0,1}(\Lambda(y-\overline{x}_i)),
\end{align}
where by the sharp asymptotic behaviour of $U_{0,1}$ and $U'_{0,1}$, we get
\begin{align}\label{a11}
&\int_{B_{\frac{\overline{d}_i}{4}}(\overline{x}_1)}
\Lambda^NU_{0,1}^q(\Lambda(y-\overline{x}_1))U_{0,1}(\Lambda(y-\overline{x}_i))
\nonumber\\
=&\int_{B_{\frac{\Lambda\overline{d}_i}{4}}(0)}
U_{0,1}^q(z)U_{0,1}(z+\Lambda(\overline{x}_1-\overline{x}_i))
\nonumber\\
=&\int_{B_{\frac{\Lambda\overline{d}_i}{4}}(0)}
U_{0,1}^q(z)\left[U_{0,1}(\Lambda(\overline{x}_1-\overline{x}_i))
+O\left(U'_{0,1}(\Lambda(\overline{x}_1-\overline{x}_i))|z|\right)\right]
\nonumber\\
=&\int_{\mathbb{R}^N}U_{0,1}^q(z)U_{0,1}(\Lambda(\overline{x}_1-\overline{x}_i))
+O\left(\int_{\mathbb{R}^N\setminus _{\frac{\Lambda\overline{d}_i}{4}}(0)}U_{0,1}^q(z)U_{0,1}(\Lambda(\overline{x}_1-\overline{x}_i))\right)
\nonumber\\
&+O\left(\int_{B_{\frac{\Lambda\overline{d}_i}{4}}(0)}
U_{0,1}^q(z)U'_{0,1}(\Lambda(\overline{x}_1-\overline{x}_i))|z|\right)
\nonumber\\
=&\frac{B_0}{\Lambda^{N-2}|\overline{x}_1-\overline{x}_i|^{N-2}}
+O\left(\frac{1}{|\overline{x}_1-\overline{x}_i|^{N}}\right)
+O\left(\int_{B_{\frac{\Lambda\overline{d}_i}{4}}(0)}
U_{0,1}^q(z)U'_{0,1}(\Lambda(\overline{x}_1-\overline{x}_i))|z|\right)
\nonumber\\
=&\frac{B_0}{\Lambda^{N-2}|\overline{x}_1-\overline{x}_i|^{N-2}}
+O\left(\frac{1}{|\overline{x}_1-\overline{x}_i|^{N-1}}\right).
\end{align}
Moreover, for $y\in \mathbb{R}^N\setminus B_{\frac{\overline{d}_i}{4}}(\overline{x}_1)$, $|y-\overline{x}_1|\geq \frac{1}{4}|\overline{x}_1-\overline{x}_i|$, then
\begin{align}\label{a12}
&\int_{\mathbb{R}^N\setminus B_{\frac{\overline{d}_i}{4}}(\overline{x}_1)}
\Lambda^NU_{0,1}^q(\Lambda(y-\overline{x}_1))U_{0,1}(\Lambda(y-\overline{x}_i))
\nonumber\\
\leq &C\Lambda^N\int_{\mathbb{R}^N\setminus B_{\frac{\overline{d}_i}{4}}(\overline{x}_1)}
\frac{1}{(1+\Lambda|y-\overline{x}_1|)^{q(N-2)}}\frac{1}{(1+\Lambda|y-\overline{x}_i|)^{N-2}}
\nonumber\\
\leq &C\Lambda^{N-2}\int_{\mathbb{R}^N\setminus B_{\frac{\overline{d}_i}{4}}(\overline{x}_1)}
\frac{1}{|\overline{x}_1-\overline{x}_i|^2}
\frac{1}{(1+\Lambda|y-\overline{x}_1|)^{q(N-2)-2}}\frac{1}{(1+\Lambda|y-\overline{x}_i|)^{N-2}}
\nonumber\\
\leq &C\Lambda^{N-2}\int_{\mathbb{R}^N\setminus B_{\frac{\overline{d}_i}{4}}(\overline{x}_1)}
\frac{1}{|\overline{x}_1-\overline{x}_i|^2}
\frac{1}{|\overline{x}_1-\overline{x}_i|^{N-2-\varepsilon_0}}
\Bigg\{\frac{1}{(1+\Lambda|y-\overline{x}_1|)^{q(N-2)+\varepsilon_0-2}}
\nonumber\\
&+\frac{1}{(1+\Lambda|y-\overline{x}_i|)^{q(N-2)+\varepsilon_0-2}}\Bigg\}
\nonumber\\
\leq &C\int_{\mathbb{R}^N\setminus B_{\frac{\overline{d}_i}{4}}(\overline{x}_1)}
\frac{1}{|\overline{x}_1-\overline{x}_i|^{N-\varepsilon_0}}
\frac{1}{(1+\Lambda|y-\overline{x}_1|)^{q(N-2)+\varepsilon_0-2}}
\nonumber\\
&+C\frac{1}{|\overline{x}_1-\overline{x}_i|^{N-\varepsilon_0}}
\Bigg\{\int_{(\mathbb{R}^N\setminus B_{\frac{\overline{d}_i}{4}}(\overline{x}_1))\cap B_{\frac{\overline{d}_i}{4}}(\overline{x}_i)}
+\int_{(\mathbb{R}^N\setminus B_{\frac{\overline{d}_i}{4}}(\overline{x}_1))\setminus B_{\frac{\overline{d}_i}{4}}(\overline{x}_i)}\Bigg\}
\frac{1}{(1+\Lambda|y-\overline{x}_i|)^{q(N-2)+\varepsilon_0-2}}
\nonumber\\
=&O\left(\frac{1}{|\overline{x}_1-\overline{x}_i|^{N-\varepsilon_0}}\right)
+O\left(\frac{1}{|\overline{x}_1-\overline{x}_i|^{N}}\right)
=O\left(\frac{1}{|\overline{x}_1-\overline{x}_i|^{N-\varepsilon_0}}\right).
\end{align}
In conclusion, we prove the first estimate and the second estimate can be proved in the same way.
\end{proof}

\begin{lemma}\label{a2}
Suppose that $N\geq 5$ and $\tau\in (0,2)$, $y=(y_1,\cdots, y_N)$. Then there is a small $\sigma>0$, such that that,
\begin{align*}
\int_{\mathbb{R}^N}\frac{1}{|y-z|^{N-2}}W_2^{p-1}\sum_{j=1}^k\frac{1}{(1+|z-\overline{x}_j|)^{\frac{N-2}{2}+\tau}}dz
\leq C\sum_{j=1}^k\frac{1}{(1+|y-\overline{x}_j|)^{\frac{N-2}{2}+\tau+\sigma}}\ \text{ for } y_3\geq 0,
\end{align*}
and
\begin{align*}
\int_{\mathbb{R}^N}\frac{1}{|y-z|^{N-2}}W_2^{p-1}\sum_{j=1}^k\frac{1}{(1+|z-\underline{x}_j|)^{\frac{N-2}{2}+\tau}}dz
\leq C\sum_{j=1}^k\frac{1}{(1+|y-\underline{x}_j|)^{\frac{N-2}{2}+\tau+\sigma}}\ \text{ for } y_3\leq 0.
\end{align*}
\end{lemma}
\begin{proof}
By \cite[Lemma B.3]{weiy}, there exists $\tau_1>\frac{N-2-m}{N-2}$ such that for $z\in \Omega_1^+$,
\begin{align*}
\sum_{j=2}^k\frac{1}{(1+|z-\overline{x}_j|)^{N-2}}\leq \frac{C}{(1+|z-\overline{x}_1|)^{N-2-\tau_1}}.
\end{align*}
Moreover, from \cite[Lemma B.1]{duanmw}, there exists $\alpha\in (1,N-2)$ such that
\begin{align*}
\sum_{j=1}^k\frac{1}{(1+|z-\overline{x}_j|)^{N-2}}\leq \frac{C}{(1+|z-\underline{x}_1|)^{N-2-\alpha}}\frac{1}{|\overline{x}_1-\underline{x}_1|^{\alpha}}.
\end{align*}
Thus,
\begin{align*}
&\int_{\mathbb{R}^N}\frac{1}{|y-z|^{N-2}}W_2^{p-1}\sum_{j=1}^k\frac{1}{(1+|z-\overline{x}_j|)^{\frac{N-2}{2}+\tau}}dz
\nonumber\\
=&\int_{\mathbb{R}^N}\frac{1}{|y-z|^{N-2}}
\left(\sum_{j=1}^kV_{\overline{x}_j,\Lambda}+\sum_{j=1}^kV_{\underline{x}_j,\Lambda}\right)^{p-1}
\sum_{j=1}^k\frac{1}{(1+|z-\overline{x}_j|)^{\frac{N-2}{2}+\tau}}dz
\nonumber\\
\leq &\int_{\mathbb{R}^N}\frac{1}{|y-z|^{N-2}}
\left(\frac{C}{(1+|z-\overline{x}_1|)^{N-2-\tau_1}}
+\frac{C}{(1+|z-\overline{x}_1|)^{N-2-\alpha}}\frac{1}{|\overline{x}_1-\underline{x}_1|^{\alpha}}\right)^{p-1}
\nonumber\\
&\times\sum_{j=1}^k\frac{1}{(1+|z-\overline{x}_j|)^{\frac{N-2}{2}+\tau}}dz.
\end{align*}
Since for $(N-2-\tau_1)(p-1)+\frac{N-2}{2}+\tau-\tau_1-2\in (0,N-2)$ and
$(N-2-\tau_1)(p-1)-\tau_1-2>0$, that is
$\tau_1\in [\frac{N-2-m}{N-2},N-2-\frac{N}{p}]$ with $p$ satisfies
\begin{align*}
p>\frac{N(N-2)}{(N-2)^2-(N-2-m)},
\end{align*}
we have that for $z\in \Omega_1^+$,
\begin{align*}
&\frac{1}{(1+|z-\overline{x}_1|)^{(N-2-\tau_1)(p-1)}}\sum_{j=1}^k\frac{1}{(1+|z-\overline{x}_j|)^{\frac{N-2}{2}+\tau}}
\nonumber\\
\leq &\frac{1}{(1+|z-\overline{x}_1|)^{(N-2-\tau_1)(p-1)+\frac{N-2}{2}+\tau}}
+\frac{1}{(1+|z-\overline{x}_1|)^{(N-2-\tau_1)(p-1)+\frac{N-2}{2}+\tau-\tau_1}}
\sum_{j=2}^k\frac{1}{|\overline{x}_1-\overline{x}_j|^{\tau_1}}.
\end{align*}
Therefore, there exists $\theta>0$, such that
\begin{align*}
&\int_{\Omega_1^+}\frac{1}{|y-z|^{N-2}}
\frac{1}{(1+|z-\overline{x}_1|)^{(N-2-\tau_1)(p-1)}}\sum_{j=1}^k\frac{1}{(1+|z-\overline{x}_j|)^{\frac{N-2}{2}+\tau}}
\nonumber\\
\leq &\int_{\Omega_1^+}\frac{1}{|y-z|^{N-2}}
\frac{C}{(1+|z-\overline{x}_1|)^{(N-2-\tau_1)(p-1)+\frac{N-2}{2}+\tau-\tau_1}}
\nonumber\\
\leq &\frac{C}{(1+|y-\overline{x}_1|)^{(N-2-\tau_1)(p-1)+\frac{N-2}{2}+\tau-\tau_1}}
\nonumber\\
\leq &\frac{C}{(1+|y-\overline{x}_1|)^{\frac{N-2}{2}+\tau+\theta}}.
\end{align*}
Thus,
\begin{align*}
&\int_{\mathbb{R}^N}\frac{1}{|y-z|^{N-2}}
\frac{1}{(1+|z-\overline{x}_1|)^{(N-2-\tau_1)(p-1)}}\sum_{j=1}^k\frac{1}{(1+|z-\overline{x}_j|)^{\frac{N-2}{2}+\tau}}
\nonumber\\
=&2\sum_{i}^k\int_{\Omega_i^+}\frac{1}{|y-z|^{N-2}}
\frac{1}{(1+|z-\overline{x}_i|)^{(N-2-\tau_1)(p-1)}}\sum_{j=1,j\neq i}^k\frac{1}{(1+|z-\overline{x}_j|)^{\frac{N-2}{2}+\tau}}
\nonumber\\
\leq &C\sum_{j=1}^k\frac{1}{(1+|y-\overline{x}_j|)^{\frac{N-2}{2}+\tau+\sigma}}.
\end{align*}
Since $(N-2-\alpha)(p-1)-\alpha-2>0$, it holds
\begin{align*}
&\left(\frac{C}{(1+|z-\underline{x}_1|)^{N-2-\alpha}}\frac{1}{|\overline{x}_1-\underline{x}_1|^{\alpha}}\right)^{p-1}
\sum_{j=1}^k\frac{1}{(1+|z-\overline{x}_j|)^{\frac{N-2}{2}+\tau}}
\nonumber\\
\leq &\left(\frac{C}{(1+|z-\overline{x}_1|)^{(N-2-\alpha)(p-1)+\frac{N-2}{2}+\tau}}
+\frac{C}{(1+|z-\overline{x}_1|)^{(N-2-\alpha)(p-1)-\alpha+\frac{N-2}{2}+\tau}}
\sum_{j=2}^k\frac{1}{|\overline{x}_1-\overline{x}_j|^{\alpha}}
\right)
\nonumber\\
&\times \frac{1}{|\overline{x}_1-\underline{x}_1|^{\alpha(p-1)}}.
\end{align*}
Similarly, we get
\begin{align*}
&\int_{\mathbb{R}^N}\frac{1}{|y-z|^{N-2}}
\left(\frac{C}{(1+|z-\underline{x}_1|)^{N-2-\alpha}}\frac{1}{|\overline{x}_1-\underline{x}_1|^{\alpha}}\right)^{p-1}
\sum_{j=1}^k\frac{1}{(1+|z-\overline{x}_j|)^{\frac{N-2}{2}+\tau}}
\nonumber\\
\leq &\frac{C}{(1+|y-\overline{x}_1|)^{\frac{N-2}{2}+\tau+\theta}}.
\end{align*}
Therefore, we have proved the result for $y_3\geq 0$ and the case of $y_3\leq 0$ can be handled in a same way.
\end{proof}

\begin{proposition}\label{iw}
We have
\begin{align*}
I(W_1,W_2)
=&kA_1-\frac{k}{\Lambda^{N-2}}\left[\frac{B_4k^{N-2}}{(r\sqrt{1-h^2})^{N-2}}
+\frac{B_5k}{r^{N-2}h^{N-3}\sqrt{1-h^2}}\right]
\nonumber\\
&+k\left[\frac{A_2}{\mu^{m_2}\Lambda^{m_2}}+\frac{\bar{A}_2}{\mu^{m_1}\Lambda^{m_1}}\right]
+k\left[\frac{A_3}{\mu^{m_2}\Lambda^{m_2-2}}+\frac{\bar{A}_3}{\mu^{m_1}\Lambda^{m_1-2}}\right](\mu-r)^2
\nonumber\\
&+k\frac{C(r,\Lambda)}{\mu^m}(\mu-r)^{\theta+2}+k\frac{C(r,\Lambda)}{\mu^{m+\theta}}
+kO\left(\frac{1}{k^{\frac{m(N-2)}{N-2-m}+\frac{2(N-3)}{N-1}+\theta}}\right),
\end{align*}
where $A_1$, $B_4$, $B_5$, $A_2$, $\bar{A}_2$, $A_3$, $\bar{A}_3$, $m$, $\theta$ are defined in (\ref{ca1}),(\ref{cb4}), (\ref{ca2}), (\ref{ca3}), (\ref{cm}).
\end{proposition}
\begin{proof}
Since
\begin{align*}
I(W_1,W_2)=&\int_{\mathbb{R}^N}\nabla W_1\nabla W_2
-\left(\frac{1}{p+1}\int_{\mathbb{R}^N}K_1\left(\frac{|y|}{\mu}\right)|W_2|^{p+1}
+\frac{1}{q+1}\int_{\mathbb{R}^N}K_2\left(\frac{|y|}{\mu}\right)|W_1|^{q+1}\right)
\nonumber\\
:=&I_1-I_2,
\end{align*}
we can estimate the integral in two parts. For the first part $I_1$, by using the symmetry, we compute that
\begin{align}\label{i1}
I_1=&\int_{\mathbb{R}^N}\nabla W_1\nabla W_2
\nonumber\\
=&\sum_{j=1}^k\sum_{i=1}^k\int_{\mathbb{R}^N}-\Delta\left(U_{\overline{x}_j,\Lambda}+U_{\underline{x}_j,\Lambda}\right)
\left(V_{\overline{x}_j,\Lambda}+V_{\underline{x}_j,\Lambda}\right)
\nonumber\\
=&\sum_{j=1}^k\sum_{i=1}^k\int_{\mathbb{R}^N}\left(V_{\overline{x}_j,\Lambda}^p+V_{\underline{x}_j,\Lambda}^p\right)
\left(V_{\overline{x}_j,\Lambda}+V_{\underline{x}_j,\Lambda}\right)
\nonumber\\
=&2k\sum_{j=1}^k\int_{\mathbb{R}^N}\left(V_{\overline{x}_1,\Lambda}^pV_{\overline{x}_j,\Lambda}
+V_{\underline{x}_1,\Lambda}^pV_{\overline{x}_j,\Lambda}\right)
\nonumber\\
=&2k\int_{\mathbb{R}^N}\left(V_{\overline{x}_1,\Lambda}^{p+1}
+\sum_{j=2}^kV_{\overline{x}_1,\Lambda}^pV_{\overline{x}_j,\Lambda}\right)
+2k\int_{\mathbb{R}^N}\sum_{j=1}^kV_{\underline{x}_1,\Lambda}^pV_{\overline{x}_j,\Lambda}
\nonumber\\
=&2k\int_{\mathbb{R}^N}V_{0,1}^{p+1}
+2k\int_{\mathbb{R}^N}V_{\overline{x}_1,\Lambda}^p\left(\sum_{j=2}^kV_{\overline{x}_j,\Lambda}
+\sum_{j=1}^kV_{\underline{x}_j,\Lambda}\right).
\end{align}
For the second term $I_2$, define $I_2:=I_{21}+I_{22}$, we shall estimate $I_{21}$ and $I_{22}$.
If $q+1\leq 3$, then
\begin{align*}
I_{22}=&\frac{1}{q+1}\int_{\mathbb{R}^N}K_2\left(\frac{|y|}{\mu}\right)|W_1|^{q+1}
\nonumber\\
=&\frac{1}{q+1}\int_{\mathbb{R}^N}K_2\left(\frac{|y|}{\mu}\right)
\left|\sum_{j=1}^kU_{\overline{x}_j,\Lambda}+\sum_{j=1}^kU_{\underline{x}_j,\Lambda}\right|^{q+1}
\nonumber\\
=&\frac{2k}{q+1}\int_{\Omega_1^+}K_2\left(\frac{|y|}{\mu}\right)\Bigg\{U_{\overline{x}_1,\Lambda}^{q+1}
+(q+1)U_{\overline{x}_1,\Lambda}^{q}\left(\sum_{j=2}^kU_{\overline{x}_j,\Lambda}
+\sum_{j=1}^kU_{\underline{x}_j,\Lambda}\right)
\nonumber\\
&+O\left(U_{\overline{x}_1,\Lambda}^{\frac{q+1}{2}}\left(\sum_{j=2}^kU_{\overline{x}_j,\Lambda}
+\sum_{j=1}^kU_{\underline{x}_j,\Lambda}\right)^{\frac{q+1}{2}}\right)
\Bigg\}.
\end{align*}
If $q+1>3$, then
\begin{align*}
I_{22}=&\frac{1}{q+1}\int_{\mathbb{R}^N}K_2\left(\frac{|y|}{\mu}\right)|W_1|^{q+1}
\nonumber\\
=&\frac{1}{q+1}\int_{\mathbb{R}^N}K_2\left(\frac{|y|}{\mu}\right)
\left|\sum_{j=1}^kU_{\overline{x}_j,\Lambda}+\sum_{j=1}^kU_{\underline{x}_j,\Lambda}\right|^{q+1}
\nonumber\\
=&\frac{2k}{q+1}\int_{\Omega_1^+}K_2\left(\frac{|y|}{\mu}\right)\Bigg\{U_{\overline{x}_1,\Lambda}^{q+1}
+(q+1)U_{\overline{x}_1,\Lambda}^{q}\left(\sum_{j=2}^kU_{\overline{x}_j,\Lambda}
+\sum_{j=1}^kU_{\underline{x}_j,\Lambda}\right)
\nonumber\\
&+O\left(U_{\overline{x}_1,\Lambda}^{\frac{q+1}{2}}\left(\sum_{j=2}^kU_{\overline{x}_j,\Lambda}
+\sum_{j=1}^kU_{\underline{x}_j,\Lambda}\right)^{\frac{q+1}{2}}
+U_{\overline{x}_1,\Lambda}^{q-1}\left(\sum_{j=2}^kU_{\overline{x}_j,\Lambda}
+\sum_{j=1}^kU_{\underline{x}_j,\Lambda}\right)^{2}\right)
\Bigg\}.
\end{align*}
By \cite[Lemma B.1]{duanmw}, for $y\in \Omega_1^+$, choose $\alpha=\frac{2}{q+1}\varepsilon_0$,
\begin{align*}
\left(\sum_{j=2}^kU_{\overline{x}_j,\Lambda}
+\sum_{j=1}^kU_{\underline{x}_j,\Lambda}\right)
\leq \frac{C}{(1+|y-\overline{x}_1|)^{\alpha}}\left(\frac{k}{r}\right)^{N-2-\alpha}
=\frac{C}{(1+|y-\overline{x}_1|)^{\frac{2}{q+1}\varepsilon_0}}\left(\frac{k}{r}\right)^{N-2-\frac{2}{q+1}\varepsilon_0},
\end{align*}
where $\varepsilon_0$ can be small enough.

Then by $\frac{q+1}{2}(N-2)\geq N$,
\begin{align*}
&\int_{\Omega_1^+}K_2\left(\frac{|y|}{\mu}\right)U_{\overline{x}_1,\Lambda}^{\frac{q+1}{2}}\left(\sum_{j=2}^kU_{\overline{x}_j,\Lambda}
+\sum_{j=1}^kU_{\underline{x}_j,\Lambda}\right)^{\frac{q+1}{2}}
\nonumber\\
=&O\left(\int_{\Omega_1^+}K_2\left(\frac{|y|}{\mu}\right)U_{\overline{x}_1,\Lambda}^{\frac{q+1}{2}}
\frac{1}{(1+|y-\overline{x}_1|)^{\varepsilon_0}}\left(\frac{k}{r}\right)^{\frac{q+1}{2}(N-2)-\varepsilon_0}
\right)
\nonumber\\
=&O\left(\int_{\Omega_1^+}K_2\left(\frac{|y|}{\mu}\right)
\frac{1}{(1+|y-\overline{x}_1|)^{\frac{q+1}{2}(N-2)+\varepsilon_0}}\left(\frac{k}{r}\right)^{\frac{q+1}{2}(N-2)-\varepsilon_0}
\right)
=O\left(\left(\frac{k}{r}\right)^{N-\varepsilon_0}\right).
\end{align*}
If $q+1>3$, by
\begin{align*}
\left(\sum_{j=2}^kU_{\overline{x}_j,\Lambda}
+\sum_{j=1}^kU_{\underline{x}_j,\Lambda}\right)
\leq \frac{C}{(1+|y-\overline{x}_1|)^{\frac{1}{2}\varepsilon_0}}\left(\frac{k}{r}\right)^{N-2-\frac{1}{2}\varepsilon_0},
\end{align*}
we have
\begin{align*}
&\int_{\Omega_1^+}K_2\left(\frac{|y|}{\mu}\right)U_{\overline{x}_1,\Lambda}^{q-1}
\left(\sum_{j=2}^kU_{\overline{x}_j,\Lambda}
+\sum_{j=1}^kU_{\underline{x}_j,\Lambda}\right)^{2}
\nonumber\\
=&O\left(\int_{\Omega_1^+}K_2\left(\frac{|y|}{\mu}\right)U_{\overline{x}_1,\Lambda}^{q-1}
\frac{1}{(1+|y-\overline{x}_1|)^{\varepsilon_0}}\left(\frac{k}{r}\right)^{2(N-2)-\varepsilon_0}
\right)
=O\left(\left(\frac{k}{r}\right)^{N-\varepsilon_0}\right).
\end{align*}
Moreover,
\begin{align*}
\int_{\Omega_1^+}K_2\left(\frac{|y|}{\mu}\right)U_{\overline{x}_1,\Lambda}^{q+1}
=&\int_{\Omega_1^+}U_{\overline{x}_1,\Lambda}^{q+1}
+\int_{\Omega_1^+}\left(K_2\left(\frac{|y|}{\mu}\right)-1\right)U_{\overline{x}_1,\Lambda}^{q+1}
\nonumber\\
=&\int_{\mathbb{R}^N}U_{0,1}^{q+1}
+\int_{\Omega_1^+}\left(K_2\left(\frac{|y|}{\mu}\right)-1\right)U_{\overline{x}_1,\Lambda}^{q+1}
+O\left(\int_{\mathbb{R}^N\setminus \Omega_1^+}U_{\overline{x}_1,\Lambda}^{q+1}\right)
\nonumber\\
=&\int_{\mathbb{R}^N}U_{0,1}^{q+1}
+\int_{\Omega_1^+}\left(K_2\left(\frac{|y|}{\mu}\right)-1\right)U_{\overline{x}_1,\Lambda}^{q+1}
+O\left(\left(\frac{k}{\mu}\right)^{N}\right)
\nonumber\\
=&\int_{\mathbb{R}^N}U_{0,1}^{q+1}
+\left\{\int_{\Omega_1^+\cap \{y:|\frac{|y|}{\mu}-1|\geq \delta\}}
+\int_{\Omega_1^+\cap \{y:|\frac{|y|}{\mu}-1|\leq \delta\}}\right\}
\left(K_2\left(\frac{|y|}{\mu}\right)-1\right)U_{\overline{x}_1,\Lambda}^{q+1}
\nonumber\\
&+O\left(\left(\frac{k}{\mu}\right)^{N}\right).
\end{align*}
When $|\frac{|y|}{\mu}-1|\geq \delta$,
$|y-\overline{x}_1|\geq ||y|-|\overline{x}_1||\geq ||y|-\mu|-|\mu-|\overline{x}_1||\geq \frac{1}{2}\mu\delta$, then
\begin{align*}
\int_{\Omega_1^+\cap \{y:|\frac{|y|}{\mu}-1|\geq \delta\}}\left(K_2\left(\frac{|y|}{\mu}\right)-1\right)U_{\overline{x}_1,\Lambda}^{q+1}
=O\left(\frac{1}{\mu^{N-\varepsilon_0}}\right).
\end{align*}
When $|\frac{|y|}{\mu}-1|\leq \delta$, using the condition $(K_2)$,
\begin{align*}
&\int_{\Omega_1^+\cap \{y:|\frac{|y|}{\mu}-1|\leq \delta\}}\left(K_2\left(\frac{|y|}{\mu}\right)-1\right)U_{\overline{x}_1,\Lambda}^{q+1}
\nonumber\\
=&-\frac{c_2}{\mu^{m_2}}\int_{\Omega_1^+\cap \{y:|\frac{|y|}{\mu}-1|\leq \delta\}}
\left||y|-\mu\right|^{m_2}U_{\overline{x}_1,\Lambda}^{q+1}
+O\left(\frac{1}{\mu^{m_2+\theta_2}}\int_{\Omega_1^+\cap \{y:|\frac{|y|}{\mu}-1|\leq \delta\}}
\left||y|-\mu\right|^{m_2+\theta_2}U_{\overline{x}_1,\Lambda}^{q+1}
\right)
\nonumber\\
=&-\frac{c_2}{\mu^{m_2}}\int_{\mathbb{R}^N}
\left||y|-\mu\right|^{m_2}U_{\overline{x}_1,\Lambda}^{q+1}
+O\left(\int_{\mathbb{R}^N\setminus B_{\frac{\mu}{k}}(\overline{x}_1)}\left[\left(\frac{|y|}{\mu}\right)^{m_2}+1\right]
U_{\overline{x}_1,\Lambda}^{q+1}
\right)
\nonumber\\
&+O\left(\frac{1}{\mu^{m_2+\theta_2}}\int_{\Omega_1^+\cap \{y:|\frac{|y|}{\mu}-1|\leq \delta\}}
\left||y|-\mu\right|^{m_2+\theta_2}U_{\overline{x}_1,\Lambda}^{q+1}
\right)
\nonumber\\
=&-\frac{c_2}{\mu^{m_2}}\int_{\mathbb{R}^N}
\left||y+\overline{x}_1|-\mu\right|^{m_2}U_{0,\Lambda}^{q+1}
+O\left(\left(\frac{k}{\mu}\right)^{N-\varepsilon_0}\right)
\nonumber\\
&+O\left(\frac{1}{\mu^{m_2+\theta_2}}\int_{\Omega_1^+\cap \{y:|\frac{|y|}{\mu}-1|\leq \delta\}}
\left||y|-\mu\right|^{m_2+\theta_2}U_{\overline{x}_1,\Lambda}^{q+1}
\right)
\nonumber\\
=&-\frac{c_2}{\mu^{m_2}}\left[\frac{1}{\Lambda^{m_2}}\int_{\mathbb{R}^N}|y_1|^{m_2}U_{0,1}^{q+1}
+\frac{m_2(m_2-1)}{2\Lambda^{m_2-2}}\int_{\mathbb{R}^N}|y_1|^{m_2-2}U_{0,1}^{q+1}(\mu-r)^2
+C(r,\Lambda)(\mu-r)^{2+\theta_2}
\right]
\nonumber\\
&+O\left(\left(\frac{k}{\mu}\right)^{N-\varepsilon_0}\right)
+O\left(\frac{1}{\mu^{m_2+\theta_2}}\int_{\Omega_1^+\cap \{y:|\frac{|y|}{\mu}-1|\leq \delta\}}
\left||y|-\mu\right|^{m_2+\theta_2}U_{\overline{x}_1,\Lambda}^{q+1}
\right),
\end{align*}
where
\begin{align*}
\frac{1}{\mu^{m_2+\theta_2}}\int_{\Omega_1^+\cap \{y:|\frac{|y|}{\mu}-1|\leq \delta\}}
\left||y|-\mu\right|^{m_2+\theta_2}U_{\overline{x}_1,\Lambda}^{q+1}
=&\frac{1}{\mu^{m_2+\theta_2}}\int_{\mathbb{R}^N}
\left||y|-\mu\right|^{m_2+\theta_2}U_{\overline{x}_1,\Lambda}^{q+1}
+O\left(\left(\frac{k}{\mu}\right)^{N-\varepsilon_0}\right)
\nonumber\\
=&\frac{C(r,\Lambda)}{\mu^{m_2+\theta_2}}+O\left(\left(\frac{k}{\mu}\right)^{N-\varepsilon_0}\right).
\end{align*}
Here $C(r,\Lambda)$ denote functions which are independent of $h$ and can be absorbed in $O(1)$.

Consequently,
\begin{align*}
\int_{\Omega_1^+}K_2\left(\frac{|y|}{\mu}\right)U_{\overline{x}_1,\Lambda}^{q+1}
=&\int_{\mathbb{R}^N}U_{0,1}^{q+1}
-\frac{c_2}{\mu^{m_2}\Lambda^{m_2}}\int_{\mathbb{R}^N}|y_1|^{m_2}U_{0,1}^{q+1}
\nonumber\\
&-\frac{1}{2}m_2(m_2-1)\frac{c_2}{\mu^{m_2}\Lambda^{m_2-2}}\int_{\mathbb{R}^N}|y_1|^{m_2-2}U_{0,1}^{q+1}(\mu-r)^2
\nonumber\\
&+\frac{C(r,\Lambda)}{\mu^{m_2}}(\mu-r)^{2+\theta_2}+\frac{C(r,\Lambda)}{\mu^{m_2+\theta_2}}
+O\left(\left(\frac{k}{\mu}\right)^{N-\varepsilon_0}\right).
\end{align*}
It holds
\begin{align*}
&(q+1)\int_{\Omega_1^+}K_2\left(\frac{|y|}{\mu}\right)U_{\overline{x}_1,\Lambda}^{q}\left(\sum_{j=2}^kU_{\overline{x}_j,\Lambda}
+\sum_{j=1}^kU_{\underline{x}_j,\Lambda}\right)
\nonumber\\
=&(q+1)\int_{\Omega_1^+}U_{\overline{x}_1,\Lambda}^{q}\left(\sum_{j=2}^kU_{\overline{x}_j,\Lambda}
+\sum_{j=1}^kU_{\underline{x}_j,\Lambda}\right)
\nonumber\\
&+(q+1)\int_{\Omega_1^+}\left(K_2\left(\frac{|y|}{\mu}\right)-1\right)U_{\overline{x}_1,\Lambda}^{q}\left(\sum_{j=2}^kU_{\overline{x}_j,\Lambda}
+\sum_{j=1}^kU_{\underline{x}_j,\Lambda}\right)
\nonumber\\
=&(q+1)\int_{\mathbb{R}^N}U_{\overline{x}_1,\Lambda}^{q}\left(\sum_{j=2}^kU_{\overline{x}_j,\Lambda}
+\sum_{j=1}^kU_{\underline{x}_j,\Lambda}\right)
-(q+1)\int_{\mathbb{R}^N\setminus \Omega_1^+}U_{\overline{x}_1,\Lambda}^{q}\left(\sum_{j=2}^kU_{\overline{x}_j,\Lambda}
+\sum_{j=1}^kU_{\underline{x}_j,\Lambda}\right)
\nonumber\\
&+(q+1)\int_{\Omega_1^+}\left(K_2\left(\frac{|y|}{\mu}\right)-1\right)U_{\overline{x}_1,\Lambda}^{q}\left(\sum_{j=2}^kU_{\overline{x}_j,\Lambda}
+\sum_{j=1}^kU_{\underline{x}_j,\Lambda}\right),
\end{align*}
then we can estimate the integral in three parts. For the second part in the above equality, let $\bar{d}_j=|\overline{x}_1-\overline{x}_j|$, $j=2,\cdots, k$ and
$\bar{d}_2=|\overline{x}_1-\overline{x}_2|=2r\sqrt{1-h^2}\sin\frac{\pi}{k}=O\left(\frac{r}{k}\right)$, by \cite[Lemma B.3]{duanmw}, we compute that
\begin{align*}
\sum_{j=2}^k\int_{\mathbb{R}^N\setminus \Omega_1^+}U_{\overline{x}_1,\Lambda}^{q}U_{\overline{x}_j,\Lambda}
=&\sum_{j=2}^k\left[\int_{(\mathbb{R}^N\setminus \Omega_1^+)\cap B_{\frac{\bar{d}_j}{2}}(\overline{x}_1)}
+\int_{(\mathbb{R}^N\setminus \Omega_1^+)\setminus B_{\frac{\bar{d}_j}{2}}(\overline{x}_1)}
\right]
U_{\overline{x}_1,\Lambda}^{q}U_{\overline{x}_j,\Lambda}
\nonumber\\
=&\sum_{j=2}^k\int_{(\mathbb{R}^N\setminus \Omega_1^+)\cap B_{\frac{\bar{d}_j}{2}}(\overline{x}_1)}U_{\overline{x}_1,\Lambda}^{q}U_{\overline{x}_j,\Lambda}
+O\left(\sum_{j=2}^k\frac{1}{|\overline{x}_1-\overline{x}_j|^{N-\varepsilon_0}}\right)
\nonumber\\
\leq &C\sum_{j=2}^k\int_{B_{\frac{\bar{d}_j}{2}}(\overline{x}_1)\setminus B_{\frac{\bar{d}_2}{2}}(\overline{x}_1)}U_{\overline{x}_1,\Lambda}^{q}U_{\overline{x}_j,\Lambda}
+O\left(\sum_{j=2}^k\frac{1}{|\overline{x}_1-\overline{x}_j|^{N-\varepsilon_0}}\right)
\nonumber\\
\leq &C\sum_{j=2}^k\frac{1}{|\overline{x}_1-\overline{x}_j|^{N-2}}\int_{B_{\frac{\Lambda\bar{d}_j}{2}}(\overline{x}_1)\setminus B_{\frac{\Lambda\bar{d}_2}{2}}(\overline{x}_1)}\frac{1}{(1+|z|)^{(N-2)q}}
\nonumber\\
&+O\left(\sum_{j=2}^k\frac{1}{|\overline{x}_1-\overline{x}_j|^{N-\varepsilon_0}}\right)
\nonumber\\
\leq &C\sum_{j=2}^k\frac{1}{|\overline{x}_1-\overline{x}_j|^{N-2}}\bar{d}_2^{-2}
+O\left(\sum_{j=2}^k\frac{1}{|\overline{x}_1-\overline{x}_j|^{N-\varepsilon_0}}\right)
\nonumber\\
=&\sum_{j=2}^k\frac{1}{|\overline{x}_1-\overline{x}_j|^{N-2}}O\left(\left(\frac{k}{r}\right)^2\right)
+O\left(\sum_{j=2}^k\frac{1}{|\overline{x}_1-\overline{x}_j|^{N-\varepsilon_0}}\right)
\nonumber\\
=&O\left(\left(\frac{k}{r}\right)^{N-\varepsilon_0}\right).
\end{align*}
Similarly, we get
\begin{align*}
\sum_{j=1}^k\int_{\mathbb{R}^N\setminus \Omega_1^+}U_{\overline{x}_1,\Lambda}^{q}U_{\underline{x}_j,\Lambda}
=O\left(\left(\frac{k}{r}\right)^{N}\right).
\end{align*}
In conclusion, The second part we get is as follows
\begin{align*}
(q+1)\int_{\mathbb{R}^N\setminus \Omega_1^+}U_{\overline{x}_1,\Lambda}^{q}\left(\sum_{j=2}^kU_{\overline{x}_j,\Lambda}
+\sum_{j=1}^kU_{\underline{x}_j,\Lambda}\right)
=O\left(\left(\frac{k}{r}\right)^{N-\varepsilon_0}\right).
\end{align*}
As to the third part, we have
\begin{align*}
&(q+1)\int_{\Omega_1^+}\left(K_2\left(\frac{|y|}{\mu}\right)-1\right)U_{\overline{x}_1,\Lambda}^{q}\left(\sum_{j=2}^kU_{\overline{x}_j,\Lambda}
+\sum_{j=1}^kU_{\underline{x}_j,\Lambda}\right)
\nonumber\\
=&(q+1)\left\{\int_{\Omega_1^+\cap  \left\{y: \left|\frac{|y|}{\mu}-1\right|\geq \delta\right\}}
+\int_{\Omega_1^+\cap  \left\{y: \left|\frac{|y|}{\mu}-1\right|\leq \delta\right\}}
\right\}
\left(K_2\left(\frac{|y|}{\mu}\right)-1\right)U_{\overline{x}_1,\Lambda}^{q}\left(\sum_{j=2}^kU_{\overline{x}_j,\Lambda}
+\sum_{j=1}^kU_{\underline{x}_j,\Lambda}\right).
\end{align*}
For $\left|\frac{|y|}{\mu}-1\right|\geq \delta$ and $y\in \Omega_1^+$, we get $|y-\overline{x}_1|\geq ||y|-\mu|-|\mu-|\overline{x}_1||\geq \frac{1}{2}\delta \mu$ and
\begin{align*}
\left(\sum_{j=2}^kU_{\overline{x}_j,\Lambda}
+\sum_{j=1}^kU_{\underline{x}_j,\Lambda}\right)
\leq C \left(\frac{k}{r}\right)^{\alpha}\frac{1}{|y-\overline{x}_1|^{N-2-\alpha}},
\end{align*}
where $\alpha\in \left(\frac{N-2-m}{N-2}, \frac{N-2}{2}\right)$.
Then
\begin{align*}
&\int_{\Omega_1^+\cap  \left\{y: \left|\frac{|y|}{\mu}-1\right|\geq \delta\right\}}
\left(K_2\left(\frac{|y|}{\mu}\right)-1\right)U_{\overline{x}_1,\Lambda}^{q}\left(\sum_{j=2}^kU_{\overline{x}_j,\Lambda}
+\sum_{j=1}^kU_{\underline{x}_j,\Lambda}\right)
\nonumber\\
\leq &C\int_{\Omega_1^+\cap  \left\{y: \left|\frac{|y|}{\mu}-1\right|\geq \delta\right\}}
\left(\frac{k}{r}\right)^{\alpha}
\frac{1}{|y-\overline{x}_1|^{q(N-2)+N-2-\alpha}}
\nonumber\\
\leq &C\left(\frac{k}{r}\right)^{\alpha}\left(\frac{1}{\mu}\right)^{N-\alpha-\varepsilon_0}
\int_{\Omega_1^+\cap  \left\{y: \left|\frac{|y|}{\mu}-1\right|\geq \delta\right\}}
\frac{1}{|y-\overline{x}_1|^{q(N-2)-2+\varepsilon_0}}
\leq C\left(\frac{k}{r}\right)^{N}.
\end{align*}
For $\left|\frac{|y|}{\mu}-1\right|\leq \delta$,
\begin{align*}
&\int_{\Omega_1^+\cap  \left\{y: \left|\frac{|y|}{\mu}-1\right|\leq \delta\right\}}
\left(K_2\left(\frac{|y|}{\mu}\right)-1\right)U_{\overline{x}_1,\Lambda}^{q}\left(\sum_{j=2}^kU_{\overline{x}_j,\Lambda}
+\sum_{j=1}^kU_{\underline{x}_j,\Lambda}\right)
\nonumber\\
\leq &C\frac{1}{\mu^{m_2}}\int_{\Omega_1^+\cap  \left\{y: \left|\frac{|y|}{\mu}-1\right|\leq \delta\right\}}
\left||y|-\mu\right|^{m_2}U_{\overline{x}_1,\Lambda}^{q}\left(\sum_{j=2}^kU_{\overline{x}_j,\Lambda}
+\sum_{j=1}^kU_{\underline{x}_j,\Lambda}\right)
\nonumber\\
\leq &C\frac{1}{\mu^{m_2}}\Bigg\{\int_{\Omega_1^+\cap  \left\{y: \left|\frac{|y|}{\mu}-1\right|\leq \delta\right\}\cap
\left\{y:|y-\overline{x}_1|\geq \delta_1\frac{\mu}{k}\right\}}
+\int_{\Omega_1^+\cap  \left\{y: \left|\frac{|y|}{\mu}-1\right|\leq \delta\right\}\cap
\left\{y:|y-\overline{x}_1|\leq \delta_1\frac{\mu}{k}\right\}}
\Bigg\}
\nonumber\\
&\left||y|-\mu\right|^{m_2}U_{\overline{x}_1,\Lambda}^{q}\left(\sum_{j=2}^kU_{\overline{x}_j,\Lambda}
+\sum_{j=1}^kU_{\underline{x}_j,\Lambda}\right),
\end{align*}
where $\delta_1$ is a small constant.

When $|y-\overline{x}_1|\leq \delta_1\frac{\mu}{k}$, $||y|-\mu|\leq ||y|-|\overline{x}_1||+||\overline{x}_1|-\mu|\leq \delta_2\frac{\mu}{k}$ while $\delta_2$ small. Then
\begin{align*}
\frac{C}{\mu^{m_2}}\left||y|-\mu\right|^{m_2}
\leq \frac{C}{\mu^{m_2}}\left(\delta_2\frac{\mu}{k}\right)^{m_2}=\frac{C}{k^{m_2}}.
\end{align*}
Thereby,
\begin{align*}
&\frac{C}{\mu^{m_2}}\int_{\Omega_1^+\cap  \left\{y: \left|\frac{|y|}{\mu}-1\right|\leq \delta\right\}\cap
\left\{y:|y-\overline{x}_1|\leq \delta_1\frac{\mu}{k}\right\}}
\left||y|-\mu\right|^{m_2}U_{\overline{x}_1,\Lambda}^{q}\left(\sum_{j=2}^kU_{\overline{x}_j,\Lambda}
+\sum_{j=1}^kU_{\underline{x}_j,\Lambda}\right)
\nonumber\\
\leq &\frac{C}{k^{m_2}}\int_{\mathbb{R}^N}U_{\overline{x}_1,\Lambda}^{q}\left(\sum_{j=2}^kU_{\overline{x}_j,\Lambda}
+\sum_{j=1}^kU_{\underline{x}_j,\Lambda}\right)
\leq \frac{C}{k^{m_2}}\left(\frac{k}{r}\right)^{N-2}.
\end{align*}
When $|y-\overline{x}_1|\geq \delta_1\frac{\mu}{k}$, by
\begin{align*}
\left(\sum_{j=2}^kU_{\overline{x}_j,\Lambda}
+\sum_{j=1}^kU_{\underline{x}_j,\Lambda}\right)
\leq C \left(\frac{k}{r}\right)^{\alpha}\frac{1}{|y-\overline{x}_1|^{N-2-\alpha}},
\end{align*}
where $\alpha\in \left(\frac{N-2-m}{N-2}, \frac{N-2}{2}\right)$. Therefore,
\begin{align*}
&\frac{C}{\mu^{m_2}}\int_{\Omega_1^+\cap  \left\{y: \left|\frac{|y|}{\mu}-1\right|\leq \delta\right\}\cap
\left\{y:|y-\overline{x}_1|\geq \delta_1\frac{\mu}{k}\right\}}
\left||y|-\mu\right|^{m_2}U_{\overline{x}_1,\Lambda}^{q}\left(\sum_{j=2}^kU_{\overline{x}_j,\Lambda}
+\sum_{j=1}^kU_{\underline{x}_j,\Lambda}\right)
\nonumber\\
\leq &C\int_{\Omega_1^+\cap  \left\{y: \left|\frac{|y|}{\mu}-1\right|\leq \delta\right\}\cap
\left\{y:|y-\overline{x}_1|\geq \delta_1\frac{\mu}{k}\right\}}
\left(\frac{k}{r}\right)^{\alpha}\frac{1}{|y-\overline{x}_1|^{q(N-2)+N-2-\alpha}}
=O\left(\left(\frac{k}{r}\right)^{N-\varepsilon_0}\right).
\end{align*}
Above all, the third part we get is as follows
\begin{align*}
(q+1)\int_{\Omega_1^+}\left(K_2\left(\frac{|y|}{\mu}\right)-1\right)U_{\overline{x}_1,\Lambda}^{q}\left(\sum_{j=2}^kU_{\overline{x}_j,\Lambda}
+\sum_{j=1}^kU_{\underline{x}_j,\Lambda}\right)
=O\left(\left(\frac{k}{r}\right)^{N-\varepsilon_0}\right)
+O\left(\frac{1}{k^{m_2}}\left(\frac{k}{r}\right)^{N-2}\right).
\end{align*}
Therefore,
\begin{align}\label{i22}
I_{22}
=&\frac{2k}{q+1}\int_{\Omega_1^+}K_2\left(\frac{|y|}{\mu}\right)\Bigg\{U_{\overline{x}_1,\Lambda}^{q+1}
+(q+1)U_{\overline{x}_1,\Lambda}^{q}\left(\sum_{j=2}^kU_{\overline{x}_j,\Lambda}
+\sum_{j=1}^kU_{\underline{x}_j,\Lambda}\right)\Bigg\}
+kO\left(\left(\frac{k}{r}\right)^{N-\varepsilon_0}\right)
\nonumber\\
=&\frac{2k}{q+1}\Bigg\{\int_{\mathbb{R}^N}U_{0,1}^{q+1}
-\frac{c_2}{\mu^{m_2}\Lambda^{m_2}}\int_{\mathbb{R}^N}|y_1|^{m_2}U_{0,1}^{q+1}
-\frac{c_2m_2(m_2-1)}{2\mu^{m_2}\Lambda^{m_2-2}}\int_{\mathbb{R}^N}|y_1|^{m_2-2}U_{0,1}^{q+1}(\mu-r)^2
\nonumber\\
&+\frac{C(r,\Lambda)}{\mu^{m_2}}(\mu-r)^{2+\theta_2}+\frac{C(r,\Lambda)}{\mu^{m_2+\theta_2}}
+O\left(\left(\frac{k}{\mu}\right)^{N-\varepsilon_0}\right)
\Bigg\}
\nonumber\\
&+\frac{2k}{q+1} \Bigg\{(q+1)\int_{\mathbb{R}^N}U_{\overline{x}_1,\Lambda}^{q}\left(\sum_{j=2}^kU_{\overline{x}_j,\Lambda}
+\sum_{j=1}^kU_{\underline{x}_j,\Lambda}\right)
+O\left(\frac{1}{k^{m_2}}\left(\frac{k}{r}\right)^{N-2}\right)
\Bigg\}
\nonumber\\
=&\frac{2k}{q+1}\Bigg\{\int_{\mathbb{R}^N}U_{0,1}^{q+1}
+(q+1)\int_{\mathbb{R}^N}U_{\overline{x}_1,\Lambda}^{q}\left(\sum_{j=2}^kU_{\overline{x}_j,\Lambda}
+\sum_{j=1}^kU_{\underline{x}_j,\Lambda}\right)
\nonumber\\
&-\frac{c_2}{\mu^{m_2}\Lambda^{m_2}}\int_{\mathbb{R}^N}|y_1|^{m_2}U_{0,1}^{q+1}
-\frac{c_2m_2(m_2-1)}{2\mu^{m_2}\Lambda^{m_2-2}}\int_{\mathbb{R}^N}|y_1|^{m_2-2}U_{0,1}^{q+1}(\mu-r)^2
\Bigg\}
\nonumber\\
&+k\frac{C(r,\Lambda)}{\mu^{m_2}}(\mu-r)^{2+\theta_2}+k\frac{C(r,\Lambda)}{\mu^{m_2+\theta_2}}
+kO\left(\left(\frac{k}{\mu}\right)^{N-\varepsilon_0}\right)
+kO\left(\frac{1}{k^{m_2}}\left(\frac{k}{r}\right)^{N-2}\right).
\end{align}
Similarly,
\begin{align}\label{i21}
I_{21}
=&\frac{2k}{p+1}\Bigg\{\int_{\mathbb{R}^N}V_{0,1}^{p+1}
+(p+1)\int_{\mathbb{R}^N}V_{\overline{x}_1,\Lambda}^{p}\left(\sum_{j=2}^kV_{\overline{x}_j,\Lambda}
+\sum_{j=1}^kV_{\underline{x}_j,\Lambda}\right)
\nonumber\\
&-\frac{c_1}{\mu^{m_1}\Lambda^{m_1}}\int_{\mathbb{R}^N}|y_1|^{m_1}V_{0,1}^{p+1}
-\frac{c_1m_1(m_1-1)}{2\mu^{m_1}\Lambda^{m_1-2}}\int_{\mathbb{R}^N}|y_1|^{m_1-2}V_{0,1}^{p+1}(\mu-r)^2
\Bigg\}
\nonumber\\
&+k\frac{C(r,\Lambda)}{\mu^{m_1}}(\mu-r)^{2+\theta_1}+k\frac{C(r,\Lambda)}{\mu^{m_1+\theta_1}}
+kO\left(\left(\frac{k}{\mu}\right)^{N-\varepsilon_0}\right)
+kO\left(\frac{1}{k^{m_1}}\left(\frac{k}{r}\right)^{N-2}\right).
\end{align}
By (\ref{i1}), (\ref{i21}), (\ref{i22}) and Lemma \ref{a1}, we have
\begin{align*}
I(W_1,W_2)=&I_1-I_2
\nonumber\\
=&2k\left[\left(1-\frac{1}{p+1}\right)\int_{\mathbb{R}^N}V_{0,1}^{p+1}
-\frac{1}{q+1}\int_{\mathbb{R}^N}U_{0,1}^{q+1}
\right]
\nonumber\\
&-2k\int_{\mathbb{R}^N}U_{\overline{x}_1,\Lambda}^{q}\left(\sum_{j=2}^kU_{\overline{x}_j,\Lambda}
+\sum_{j=1}^kU_{\underline{x}_j,\Lambda}\right)
\nonumber\\
&+2k\Bigg\{\frac{c_1}{(p+1)\mu^{m_1}\Lambda^{m_1}}\int_{\mathbb{R}^N}|y_1|^{m_1}V_{0,1}^{p+1}
+\frac{c_2}{(q+1)\mu^{m_2}\Lambda^{m_2}}\int_{\mathbb{R}^N}|y_1|^{m_2}U_{0,1}^{q+1}
\Bigg\}
\nonumber\\
&+2k\Bigg\{\frac{c_1m_1(m_1-1)}{2(p+1)\mu^{m_1}\Lambda^{m_1-2}}\int_{\mathbb{R}^N}|y_1|^{m_1-2}V_{0,1}^{p+1}(\mu-r)^2
\nonumber\\
&+\frac{c_2m_2(m_2-1)}{2(q+1)\mu^{m_2}\Lambda^{m_2-2}}\int_{\mathbb{R}^N}|y_1|^{m_2-2}U_{0,1}^{q+1}(\mu-r)^2
\Bigg\}
+kO\left(\left(\frac{k}{\mu}\right)^{N-\varepsilon_0}\right)
\nonumber\\
&+\sum_{i=1}^2\left[k\frac{C(r,\Lambda)}{\mu^{m_i}}(\mu-r)^{2+\theta_i}+k\frac{C(r,\Lambda)}{\mu^{m_i+\theta_i}}
+kO\left(\frac{1}{k^{m_i}}\left(\frac{k}{r}\right)^{N-2}\right)
\right]
\nonumber\\
=&kA_1-\frac{k}{\Lambda^{N-2}}\left[\frac{B_4k^{N-2}}{(r\sqrt{1-h^2})^{N-2}}
-\frac{B_5k}{r^{N-2}h^{N-3}\sqrt{1-h^2}}\right]
\nonumber\\
&+k\left[\frac{A_2}{\mu^{m_2}\Lambda^{m_2}}+\frac{\bar{A}_2}{\mu^{m_1}\Lambda^{m_1}}\right]
+k\left[\frac{A_3}{\mu^{m_2}\Lambda^{m_2-2}}+\frac{\bar{A}_3}{\mu^{m_1}\Lambda^{m_1-2}}\right](\mu-r)^2
\nonumber\\
&+k\frac{C(r,\Lambda)}{\mu^m}(\mu-r)^{\theta+2}+k\frac{C(r,\Lambda)}{\mu^{m+\theta}}
+kO\left(\frac{1}{k^{\frac{m(N-2)}{N-2-m}+\frac{2(N-3)}{N-1}+\theta}}\right),
\end{align*}
where
\begin{align}\label{ca1}
A_1=2\left[\left(1-\frac{1}{p+1}\right)\int_{\mathbb{R}^N}V_{0,1}^{p+1}
-\frac{1}{q+1}\int_{\mathbb{R}^N}U_{0,1}^{q+1}\right],
\end{align}
\begin{align}\label{cb4}
B_4=2B_0B_1, \ B_5=2B_0B_2,
\end{align}
\begin{align}\label{ca2}
A_2=\frac{2c_2}{q+1}\int_{\mathbb{R}^N}|y_1|^{m_2}U_{0,1}^{q+1},\
\bar{A}_2=\frac{2c_1}{p+1}\int_{\mathbb{R}^N}|y_1|^{m_1}V_{0,1}^{p+1},
\end{align}
\begin{align}\label{ca3}
A_3=\frac{c_2m_2(m_2-1)}{q+1}\int_{\mathbb{R}^N}|y_1|^{m_2-2}U_{0,1}^{q+1},\
\bar{A}_3=\frac{c_1m_1(m_1-1)}{p+1}\int_{\mathbb{R}^N}|y_1|^{m_1-2}V_{0,1}^{p+1},
\end{align}
and
\begin{align}\label{cm}
m=\min\{m_1,m_2\},\ \theta=\min\{\theta_1,\theta_2\}.
\end{align}
\end{proof}

We next need to calculate $\frac{\partial I(W_1,W_2)}{\partial \Lambda}$ and $\frac{\partial I(W_1,W_2)}{\partial h}$.

\begin{proposition}\label{iwlamda}
We have
\begin{align*}
\frac{\partial I(W_1,W_2)}{\partial \Lambda}
=&\frac{k(N-2)}{\Lambda^{N-1}}\left[\frac{B_4k^{N-2}}{(r\sqrt{1-h^2})^{N-2}}
+\frac{B_5k}{r^{N-2}h^{N-3}\sqrt{1-h^2}}\right]
\nonumber\\
&-k\left[\frac{A_2m_2}{\mu^{m_2}\Lambda^{m_2+1}}+\frac{\bar{A}_2m_1}{\mu^{m_1}\Lambda^{m_1+1}}\right]
-k\left[\frac{A_3(m_2-2)}{\mu^{m_2}\Lambda^{m_2-1}}+\frac{\bar{A}_3(m_1-2)}{\mu^{m_1}\Lambda^{m_1-1}}\right](\mu-r)^2
\nonumber\\
&
+kO\left(\frac{1}{k^{\frac{m(N-2)}{N-2-m}+\theta}}\right),
\end{align*}
where $B_4$, $B_5$, $A_2$, $\bar{A}_2$, $A_3$, $\bar{A}_3$, $m$, $\theta$ are same positive constants in Proposition \ref{iw}.
\end{proposition}
\begin{proof}
Since
\begin{align*}
I(W_1,W_2)
=&\int_{\mathbb{R}^N}\nabla W_1\nabla W_2
-\left(\frac{1}{p+1}\int_{\mathbb{R}^N}K_1\left(\frac{|y|}{\mu}\right)|W_2|^{p+1}
+\frac{1}{q+1}\int_{\mathbb{R}^N}K_2\left(\frac{|y|}{\mu}\right)|W_1|^{q+1}\right)
\nonumber\\
=&2k\int_{\mathbb{R}^N}V_{0,1}^{p+1}
+2k\int_{\mathbb{R}^N}V_{\overline{x}_1,\Lambda}^{p}\left(\sum_{j=2}^kV_{\overline{x}_j,\Lambda}
+\sum_{j=1}^kV_{\underline{x}_j,\Lambda}\right)
\nonumber\\
&-\left(\frac{1}{p+1}\int_{\mathbb{R}^N}K_1\left(\frac{|y|}{\mu}\right)|W_2|^{p+1}
+\frac{1}{q+1}\int_{\mathbb{R}^N}K_2\left(\frac{|y|}{\mu}\right)|W_1|^{q+1}\right),
\end{align*}
we have
\begin{align*}
\frac{\partial I(W_1,W_2)}{\partial \Lambda}
=&2k\frac{\partial }{\partial \Lambda} \int_{\mathbb{R}^N}V_{\overline{x}_1,\Lambda}^{p}\left(\sum_{j=2}^kV_{\overline{x}_j,\Lambda}
+\sum_{j=1}^kV_{\underline{x}_j,\Lambda}\right)
\nonumber\\
&-2k\int_{\Omega_1^+}\left[K_1\left(\frac{|y|}{\mu}\right)W_2^{p}\frac{\partial W_2}{\partial \Lambda}
+K_2\left(\frac{|y|}{\mu}\right)W_1^{q}\frac{\partial W_1}{\partial \Lambda}
\right].
\end{align*}
For $y\in \Omega_1^+$, we see that
\begin{align*}
&\left|\frac{\partial }{\partial \Lambda}\left[W_1^{q+1}-U_{\overline{x}_1,\Lambda}^{q+1}-
(q+1)U_{\overline{x}_1,\Lambda}^{q}\left(\sum_{j=2}^kU_{\overline{x}_j,\Lambda}
+\sum_{j=1}^kU_{\underline{x}_j,\Lambda}\right)\right]
\right|
\nonumber\\
\leq &
\begin{cases}
CU_{\overline{x}_1,\Lambda}^{\frac{q+1}{2}}\left(\sum_{j=2}^kU_{\overline{x}_j,\Lambda}
+\sum_{j=1}^kU_{\underline{x}_j,\Lambda}\right)^{\frac{q+1}{2}},\ & \text{ if } q+1\leq 3,\\
CU_{\overline{x}_1,\Lambda}^{\frac{q+1}{2}}\left(\sum_{j=2}^kU_{\overline{x}_j,\Lambda}
+\sum_{j=1}^kU_{\underline{x}_j,\Lambda}\right)^{\frac{q+1}{2}}
+CU_{\overline{x}_1,\Lambda}^{q-1}\left(\sum_{j=2}^kU_{\overline{x}_j,\Lambda}
+\sum_{j=1}^kU_{\underline{x}_j,\Lambda}\right)^2, \ & \text{ if } q+1>3.
\end{cases}
\end{align*}
Then
\begin{align*}
&(q+1)\int_{\Omega_1^+}K_2\left(\frac{|y|}{\mu}\right)W_1^{q}\frac{\partial W_1}{\partial \Lambda}
\nonumber\\
=&\int_{\Omega_1^+}K_2\left(\frac{|y|}{\mu}\right)
\frac{\partial W_1^{q+1}}{\partial \Lambda}
\nonumber\\
=&\int_{\Omega_1^+}K_2\left(\frac{|y|}{\mu}\right)
\frac{\partial U_{\overline{x}_1,\Lambda}^{q+1}}{\partial \Lambda}
+(q+1)\int_{\Omega_1^+}K_2\left(\frac{|y|}{\mu}\right)\frac{\partial }{\partial \Lambda}
\left(U_{\overline{x}_1,\Lambda}^{q}\left(\sum_{j=2}^kU_{\overline{x}_j,\Lambda}
+\sum_{j=1}^kU_{\underline{x}_j,\Lambda}\right)
\right)
\nonumber\\
&+O\left(\int_{\Omega_1^+}U_{\overline{x}_1,\Lambda}^{\frac{q+1}{2}}\left(\sum_{j=2}^kU_{\overline{x}_j,\Lambda}
+\sum_{j=1}^kU_{\underline{x}_j,\Lambda}\right)^{\frac{q+1}{2}}
\right)
\nonumber\\
&+
\begin{cases}
0, \ &\text{ if } q+1\leq 3,\\
O\left(\int_{\Omega_1^+}U_{\overline{x}_1,\Lambda}^{q-1}\left(\sum_{j=2}^kU_{\overline{x}_j,\Lambda}
+\sum_{j=1}^kU_{\underline{x}_j,\Lambda}\right)^2
\right) , \ &\text{ if } q+1>3.
\end{cases}
\end{align*}
Using the proof of Proposition \ref{iw}, we get
\begin{align*}
\int_{\Omega_1^+}U_{\overline{x}_1,\Lambda}^{\frac{q+1}{2}}\left(\sum_{j=2}^kU_{\overline{x}_j,\Lambda}
+\sum_{j=1}^kU_{\underline{x}_j,\Lambda}\right)^{\frac{q+1}{2}}
=O\left(\left(\frac{k}{r}\right)^{N-\varepsilon_0}\right)
\end{align*}
and
\begin{align*}
\int_{\Omega_1^+}U_{\overline{x}_1,\Lambda}^{q-1}\left(\sum_{j=2}^kU_{\overline{x}_j,\Lambda}
+\sum_{j=1}^kU_{\underline{x}_j,\Lambda}\right)^2
=O\left(\left(\frac{k}{r}\right)^{N-\varepsilon_0}\right).
\end{align*}
On the other hand,
\begin{align*}
\int_{\Omega_1^+}K_2\left(\frac{|y|}{\mu}\right)
\frac{\partial U_{\overline{x}_1,\Lambda}^{q+1}}{\partial \Lambda}
=&\int_{\Omega_1^+}
\frac{\partial U_{\overline{x}_1,\Lambda}^{q+1}}{\partial \Lambda}
+\int_{\Omega_1^+}\left(K_2\left(\frac{|y|}{\mu}\right)-1\right)
\frac{\partial U_{\overline{x}_1,\Lambda}^{q+1}}{\partial \Lambda}
\nonumber\\
=&\int_{\mathbb{R}^N}\frac{\partial U_{\overline{x}_1,\Lambda}^{q+1}}{\partial \Lambda}
+\int_{\Omega_1^+}\left(K_2\left(\frac{|y|}{\mu}\right)-1\right)
\frac{\partial U_{\overline{x}_1,\Lambda}^{q+1}}{\partial \Lambda}
+O\left(\left(\frac{k}{r}\right)^{N}\right)
\nonumber\\
=&\int_{\Omega_1^+}\left(K_2\left(\frac{|y|}{\mu}\right)-1\right)
\frac{\partial U_{\overline{x}_1,\Lambda}^{q+1}}{\partial \Lambda}
+O\left(\left(\frac{k}{r}\right)^{N}\right)
\nonumber\\
=&\left\{\int_{\Omega_1^+\cap \{y:|\frac{|y|}{\mu}-1|\geq \delta\}}
+\int_{\Omega_1^+\cap \{y:|\frac{|y|}{\mu}-1|\leq \delta\}}\right\}
\left(K_2\left(\frac{|y|}{\mu}\right)-1\right)
\frac{\partial U_{\overline{x}_1,\Lambda}^{q+1}}{\partial \Lambda}
+O\left(\left(\frac{k}{r}\right)^{N}\right).
\end{align*}
For $|\frac{|y|}{\mu}-1|\geq \delta$, $|y-\overline{x}_1|\geq ||y|-\mu|-|\mu-|\overline{x}_1||\geq \frac{1}{2}\delta \mu$, then
\begin{align*}
\int_{\Omega_1^+\cap \{y:|\frac{|y|}{\mu}-1|\geq \delta\}}\left(K_2\left(\frac{|y|}{\mu}\right)-1\right)
\frac{\partial U_{\overline{x}_1,\Lambda}^{q+1}}{\partial \Lambda}
\leq \frac{C}{\mu^{N-\varepsilon_0}}.
\end{align*}
For $|\frac{|y|}{\mu}-1|\leq \delta$, by the decay property of $K_2$,
\begin{align*}
&\int_{\Omega_1^+\cap \{y:|\frac{|y|}{\mu}-1|\leq \delta\}}\left(K_2\left(\frac{|y|}{\mu}\right)-1\right)
\frac{\partial U_{\overline{x}_1,\Lambda}^{q+1}}{\partial \Lambda}
\nonumber\\
=&-\frac{c_2}{\mu^{m_2}}\int_{\Omega_1^+\cap \{y:|\frac{|y|}{\mu}-1|\leq \delta\}}
\left||y|-\mu\right|^{m_2}\frac{\partial U_{\overline{x}_1,\Lambda}^{q+1}}{\partial \Lambda}
+O\left(\frac{1}{\mu^{m_2+\theta_2}}\int_{\Omega_1^+\cap \{y:|\frac{|y|}{\mu}-1|\leq \delta\}}
\left||y|-\mu\right|^{m_2+\theta_2}\frac{\partial U_{\overline{x}_1,\Lambda}^{q+1}}{\partial \Lambda}\right)
\nonumber\\
=&-\frac{c_2}{\mu^{m_2}}\int_{\mathbb{R}^N}
\left||y|-\mu\right|^{m_2}\frac{\partial U_{\overline{x}_1,\Lambda}^{q+1}}{\partial \Lambda}
+O\left(\int_{\mathbb{R}^N\setminus B_{\frac{\mu}{k}}(\overline{x}_1)}\left(\frac{|y|^{m_2}}{\mu^{m_2}}+1\right)
\frac{\partial U_{\overline{x}_1,\Lambda}^{q+1}}{\partial \Lambda}
\right)
\nonumber\\
&+O\left(\frac{1}{\mu^{m_2+\theta_2}}\int_{\Omega_1^+\cap \{y:|\frac{|y|}{\mu}-1|\leq \delta\}}
\left||y|-\mu\right|^{m_2+\theta_2}\frac{\partial U_{\overline{x}_1,\Lambda}^{q+1}}{\partial \Lambda}\right)
\nonumber\\
=&-\frac{c_2}{\mu^{m_2}}\int_{\mathbb{R}^N}
\left||y|-\mu\right|^{m_2}\frac{\partial U_{\overline{x}_1,\Lambda}^{q+1}}{\partial \Lambda}
+O\left(\left(\frac{k}{\mu}\right)^{N-\varepsilon_0}\right)
\nonumber\\
&+O\left(\frac{1}{\mu^{m_2+\theta_2}}\int_{\Omega_1^+\cap \{y:|\frac{|y|}{\mu}-1|\leq \delta\}}
\left||y|-\mu\right|^{m_2+\theta_2}\frac{\partial U_{\overline{x}_1,\Lambda}^{q+1}}{\partial \Lambda}\right)
\nonumber\\
=&-\frac{c_2}{\mu^{m_2}}\int_{\mathbb{R}^N}
\left||y+\overline{x}_1|-\mu\right|^{m_2}\frac{\partial U_{0,\Lambda}^{q+1}}{\partial \Lambda}
+O\left(\left(\frac{k}{\mu}\right)^{N-\varepsilon_0}\right)
\nonumber\\
&+O\left(\frac{1}{\mu^{m_2+\theta_2}}\int_{\Omega_1^+\cap \{y:|\frac{|y|}{\mu}-1|\leq \delta\}}
\left||y|-\mu\right|^{m_2+\theta_2}\frac{\partial U_{\overline{x}_1,\Lambda}^{q+1}}{\partial \Lambda}\right).
\end{align*}
By Taylor expansion,
\begin{align*}
&\int_{\mathbb{R}^N}
\left||y+\overline{x}_1|-\mu\right|^{m_2}\frac{\partial U_{0,\Lambda}^{q+1}}{\partial \Lambda}
\nonumber\\
=&\int_{\mathbb{R}^N}
\left|y_1\right|^{m_2}\frac{\partial U_{0,\Lambda}^{q+1}}{\partial \Lambda}
+\frac{1}{2}m_2(m_2-1)\int_{\mathbb{R}^N}
\left|y_1\right|^{m_2-2}\frac{\partial U_{0,\Lambda}^{q+1}}{\partial \Lambda}(\mu-r)^2
+C(r,\Lambda)(\mu-r)^{2+\theta_2}
\nonumber\\
=&-\frac{m_2}{\Lambda^{m_2+1}}\int_{\mathbb{R}^N}
\left|y_1\right|^{m_2}U_{0,1}^{q+1}
-\frac{m_2(m_2-1)(m_2-2)}{2\Lambda^{m_2-1}}\int_{\mathbb{R}^N}
\left|y_1\right|^{m_2-2}U_{0,1}^{q+1}(\mu-r)^2
\nonumber\\
&+\frac{\partial C(r,\Lambda)}{\partial \Lambda}(\mu-r)^{2+\theta_2}.
\end{align*}
Moreover,
\begin{align*}
&O\left(\frac{1}{\mu^{m_2+\theta_2}}\int_{\Omega_1^+\cap \{y:|\frac{|y|}{\mu}-1|\leq \delta\}}
\left||y|-\mu\right|^{m_2+\theta_2}\frac{\partial U_{\overline{x}_1,\Lambda}^{q+1}}{\partial \Lambda}\right)
\nonumber\\
=&O\left(\frac{1}{\mu^{m_2+\theta_2}}\int_{\mathbb{R}^N}
\left||y|-\mu\right|^{m_2+\theta_2}\frac{\partial U_{\overline{x}_1,\Lambda}^{q+1}}{\partial \Lambda}\right)
+O\left(\left(\frac{k}{\mu}\right)^{N-\varepsilon_0}\right)
\nonumber\\
=&\frac{\partial C(r,\Lambda)}{\partial \Lambda}\frac{1}{\mu^{m_2+\theta_2}}
+O\left(\left(\frac{k}{\mu}\right)^{N-\varepsilon_0}\right).
\end{align*}
Therefore,
\begin{align*}
&\int_{\Omega_1^+}K_2\left(\frac{|y|}{\mu}\right)
\frac{\partial U_{\overline{x}_1,\Lambda}^{q+1}}{\partial \Lambda}
\nonumber\\
=&\frac{c_2m_2}{\Lambda^{m_2+1}\mu^{m_2}}\int_{\mathbb{R}^N}
\left|y_1\right|^{m_2}U_{0,1}^{q+1}
+\frac{c_2m_2(m_2-1)(m_2-2)}{2\Lambda^{m_2-1}\mu^{m_2}}\int_{\mathbb{R}^N}
\left|y_1\right|^{m_2-2}U_{0,1}^{q+1}(\mu-r)^2
\nonumber\\
&+\frac{\partial C(r,\Lambda)}{\partial \Lambda}\frac{(\mu-r)^{2+\theta_2}}{\mu^{m_2}}
+\frac{\partial C(r,\Lambda)}{\partial \Lambda}\frac{1}{\mu^{m_2+\theta_2}}
+O\left(\left(\frac{k}{\mu}\right)^{N-\varepsilon_0}\right).
\end{align*}
Using a proof similar to Proposition \ref{iw}, it holds
\begin{align*}
&\int_{\Omega_1^+}K_2\left(\frac{|y|}{\mu}\right)\frac{\partial }{\partial \Lambda}\left(U_{\overline{x}_1,\Lambda}^{q}\left(\sum_{j=2}^kU_{\overline{x}_j,\Lambda}
+\sum_{j=1}^kU_{\underline{x}_j,\Lambda}\right)
\right)
\nonumber\\
=&\frac{\partial }{\partial \Lambda}\int_{\mathbb{R}^N}U_{\overline{x}_1,\Lambda}^{q}\left(\sum_{j=2}^kU_{\overline{x}_j,\Lambda}
+\sum_{j=1}^kU_{\underline{x}_j,\Lambda}\right)
\nonumber\\
&+\frac{\partial }{\partial \Lambda}\int_{\mathbb{R}^N\setminus \Omega_1^+}U_{\overline{x}_1,\Lambda}^{q}\left(\sum_{j=2}^kU_{\overline{x}_j,\Lambda}
+\sum_{j=1}^kU_{\underline{x}_j,\Lambda}\right)
\nonumber\\
&+\frac{\partial }{\partial \Lambda}
\int_{\Omega_1^+}\left(K_2\left(\frac{|y|}{\mu}\right)-1\right)
U_{\overline{x}_1,\Lambda}^{q}\left(\sum_{j=2}^kU_{\overline{x}_j,\Lambda}
+\sum_{j=1}^kU_{\underline{x}_j,\Lambda}\right)
\nonumber\\
=&\frac{\partial }{\partial \Lambda}\int_{\mathbb{R}^N}U_{\overline{x}_1,\Lambda}^{q}\left(\sum_{j=2}^kU_{\overline{x}_j,\Lambda}
+\sum_{j=1}^kU_{\underline{x}_j,\Lambda}\right)
+O\left(\left(\frac{k}{\mu}\right)^{N-\varepsilon_0}\right)
\nonumber\\
&+\frac{\partial }{\partial \Lambda}
\int_{\Omega_1^+}\left(K_2\left(\frac{|y|}{\mu}\right)-1\right)
U_{\overline{x}_1,\Lambda}^{q}\left(\sum_{j=2}^kU_{\overline{x}_j,\Lambda}
+\sum_{j=1}^kU_{\underline{x}_j,\Lambda}\right)
\nonumber\\
=&\frac{\partial }{\partial \Lambda}\int_{\mathbb{R}^N}U_{\overline{x}_1,\Lambda}^{q}\left(\sum_{j=2}^kU_{\overline{x}_j,\Lambda}
+\sum_{j=1}^kU_{\underline{x}_j,\Lambda}\right)
+O\left(\left(\frac{k}{\mu}\right)^{N-\varepsilon_0}\right)
+O\left(\frac{1}{k^{m_2}}\left(\frac{k}{r}\right)^{N-2}\right).
\end{align*}
Thus,
\begin{align*}
&\frac{1}{q+1}\int_{\Omega_1^+}K_2\left(\frac{|y|}{\mu}\right)
\frac{\partial W_1^{q+1}}{\partial \Lambda}
\nonumber\\
=&\frac{c_2m_2}{(q+1)\Lambda^{m_2+1}\mu^{m_2}}\int_{\mathbb{R}^N}
\left|y_1\right|^{m_2}U_{0,1}^{q+1}
+\frac{c_2m_2(m_2-1)(m_2-2)}{2(q+1)\Lambda^{m_2-1}\mu^{m_2}}\int_{\mathbb{R}^N}
\left|y_1\right|^{m_2-2}U_{0,1}^{q+1}(\mu-r)^2
\nonumber\\
&+\frac{\partial C(r,\Lambda)}{\partial \Lambda}\frac{(\mu-r)^{2+\theta_2}}{\mu^{m_2}}
+\frac{\partial C(r,\Lambda)}{\partial \Lambda}\frac{1}{\mu^{m_2+\theta_2}}
+O\left(\left(\frac{k}{\mu}\right)^{N-\varepsilon_0}\right)
\nonumber\\
&+\frac{\partial }{\partial \Lambda}\int_{\mathbb{R}^N}U_{\overline{x}_1,\Lambda}^{q}\left(\sum_{j=2}^kU_{\overline{x}_j,\Lambda}
+\sum_{j=1}^kU_{\underline{x}_j,\Lambda}\right)
+O\left(\frac{1}{k^{m_2}}\left(\frac{k}{r}\right)^{N-2}\right).
\end{align*}
In an analogy way, we have
\begin{align*}
&\frac{1}{p+1}\int_{\Omega_1^+}K_1\left(\frac{|y|}{\mu}\right)
\frac{\partial W_2^{p+1}}{\partial \Lambda}
\nonumber\\
=&\frac{c_1m_1}{(p+1)\Lambda^{m_1+1}\mu^{m_1}}\int_{\mathbb{R}^N}
\left|y_1\right|^{m_1}V_{0,1}^{p+1}
+\frac{c_1m_1(m_1-1)(m_1-2)}{2(p+1)\Lambda^{m_1-1}\mu^{m_1}}\int_{\mathbb{R}^N}
\left|y_1\right|^{m_1-2}V_{0,1}^{p+1}(\mu-r)^2
\nonumber\\
&+\frac{\partial C(r,\Lambda)}{\partial \Lambda}\frac{(\mu-r)^{2+\theta_1}}{\mu^{m_1}}
+\frac{\partial C(r,\Lambda)}{\partial \Lambda}\frac{1}{\mu^{m_1+\theta_1}}
+O\left(\left(\frac{k}{\mu}\right)^{N-\varepsilon_0}\right)
\nonumber\\
&+\frac{\partial }{\partial \Lambda}\int_{\mathbb{R}^N}V_{\overline{x}_1,\Lambda}^{p}\left(\sum_{j=2}^kV_{\overline{x}_j,\Lambda}
+\sum_{j=1}^kV_{\underline{x}_j,\Lambda}\right)
+O\left(\frac{1}{k^{m_1}}\left(\frac{k}{r}\right)^{N-2}\right).
\end{align*}
Furthermore,
\begin{align}\label{ilamdaw}
\frac{\partial I(W_1,W_2)}{\partial \Lambda}
=&2k\frac{\partial }{\partial \Lambda} \int_{\mathbb{R}^N}V_{\overline{x}_1,\Lambda}^{p}\left(\sum_{j=2}^kV_{\overline{x}_j,\Lambda}
+\sum_{j=1}^kV_{\underline{x}_j,\Lambda}\right)
\nonumber\\
&-2k\left(\frac{1}{p+1}\int_{\Omega_1^+}K_1\left(\frac{|y|}{\mu}\right)
\frac{\partial W_2^{p+1}}{\partial \Lambda}
+\frac{1}{q+1}\int_{\Omega_1^+}K_2\left(\frac{|y|}{\mu}\right)
\frac{\partial W_1^{q+1}}{\partial \Lambda}
\right)
\nonumber\\
=&-2k\frac{\partial }{\partial \Lambda}\int_{\mathbb{R}^N}U_{\overline{x}_1,\Lambda}^{q}\left(\sum_{j=2}^kU_{\overline{x}_j,\Lambda}
+\sum_{j=1}^kU_{\underline{x}_j,\Lambda}\right)
\nonumber\\
&-k\left[\frac{A_2m_2}{\mu^{m_2}\Lambda^{m_2+1}}+\frac{\bar{A}_2m_1}{\mu^{m_1}\Lambda^{m_1+1}}\right]
-k\left[\frac{A_3(m_2-2)}{\mu^{m_2}\Lambda^{m_2-1}}+\frac{\bar{A}_3(m_1-2)}{\mu^{m_1}\Lambda^{m_1-1}}\right](\mu-r)^2
\nonumber\\
&
+\sum_{i=1}^2\left[\frac{\partial C(r,\Lambda)}{\partial \Lambda}\frac{(\mu-r)^{2+\theta_i}}{\mu^{m_i}}
+\frac{\partial C(r,\Lambda)}{\partial \Lambda}\frac{1}{\mu^{m_i+\theta_i}}
+O\left(\frac{1}{k^{m_i}}\left(\frac{k}{r}\right)^{N-2}\right)
\right]
\nonumber\\
&+O\left(\left(\frac{k}{\mu}\right)^{N-\varepsilon_0}\right).
\end{align}
We finally estimate $\frac{\partial }{\partial \Lambda}\int_{\mathbb{R}^N}U_{\overline{x}_1,\Lambda}^{q}\left(\sum_{j=2}^kU_{\overline{x}_j,\Lambda}
+\sum_{j=1}^kU_{\underline{x}_j,\Lambda}\right)$ and that completes the proof.
For $j=2, \cdots, k$, by the decay property of $U'$, we deduce
\begin{align*}
&\frac{\partial }{\partial \Lambda}\int_{\mathbb{R}^N}U_{\overline{x}_1,\Lambda}^{q}U_{\overline{x}_j,\Lambda}
\nonumber\\
=&\frac{\partial }{\partial \Lambda}\int_{\mathbb{R}^N}U_{0,1}^{q}(z)U_{0,1}(z+\Lambda(\overline{x}_1-\overline{x}_j))
\nonumber\\
=&\int_{\mathbb{R}^N}U_{0,1}^{q}(z)
\left(\frac{\partial }{\partial \Lambda}U_{0,1}(z+\Lambda(\overline{x}_1-\overline{x}_j))
-\frac{\partial }{\partial \Lambda}U_{0,1}(\Lambda(\overline{x}_1-\overline{x}_j))
+\frac{\partial }{\partial \Lambda}U_{0,1}(\Lambda(\overline{x}_1-\overline{x}_j))
\right)
\nonumber\\
=&\frac{\partial }{\partial \Lambda}\int_{\mathbb{R}^N}U_{0,1}^{q}(z)U_{0,1}(\Lambda(\overline{x}_1-\overline{x}_j))
+\frac{\partial }{\partial \Lambda}\int_{\mathbb{R}^N}U_{0,1}^{q}(z)
\left(U_{0,1}(z+\Lambda(\overline{x}_1-\overline{x}_j))
-U_{0,1}(\Lambda(\overline{x}_1-\overline{x}_j))\right)
\nonumber\\
=&\int_{\mathbb{R}^N}U_{0,1}^{q}(z)U'_{0,1}(\Lambda(\overline{x}_1-\overline{x}_j))|\overline{x}_1-\overline{x}_j|
\nonumber\\
&+\int_{\mathbb{R}^N}U_{0,1}^{q}(z)\left[U'_{0,1}(z+\Lambda(\overline{x}_1-\overline{x}_j))
-U'_{0,1}(\Lambda(\overline{x}_1-\overline{x}_j))
\right]|\overline{x}_1-\overline{x}_j|
\nonumber\\
=&-\frac{B_0(N-2)}{\Lambda^{N-1}|\overline{x}_1-\overline{x}_j|^{N-2}}
+O\left(\frac{|\overline{x}_1-\overline{x}_j|}{(\Lambda|\overline{x}_1-\overline{x}_j|)^{N+1}}\right)
\nonumber\\
&+\int_{\mathbb{R}^N}U_{0,1}^{q}(z)\Bigg\{
\frac{-b_{N,p}(N-2)}{|z+\Lambda(\overline{x}_1-\overline{x}_j)|^{N-1}}
+\frac{b_{N,p}(N-2)}{|\Lambda(\overline{x}_1-\overline{x}_j)|^{N-1}}
+O\left(\frac{1}{|z+\Lambda(\overline{x}_1-\overline{x}_j)|^{N+1}}\right)
\nonumber\\
&+O\left(\frac{1}{|\Lambda(\overline{x}_1-\overline{x}_j)|^{N+1}}\right)
\Bigg\}
\nonumber\\
=&-\frac{B_0(N-2)}{\Lambda^{N-1}|\overline{x}_1-\overline{x}_j|^{N-2}}
+O\left(\frac{1}{|\Lambda(\overline{x}_1-\overline{x}_j)|^{N-1}}\right).
\end{align*}
Similarly, for $j=1,\cdots,k$,
\begin{align*}
\frac{\partial }{\partial \Lambda}\int_{\mathbb{R}^N}U_{\overline{x}_1,\Lambda}^{q}U_{\underline{x}_j,\Lambda}
=-\frac{B_0(N-2)}{\Lambda^{N-1}|\overline{x}_1-\underline{x}_j|^{N-2}}
+O\left(\frac{1}{|\Lambda(\overline{x}_1-\underline{x}_j)|^{N-1}}\right).
\end{align*}
Therefore, it follows from \cite[Lemma A.1]{duanmw} that
\begin{align*}
\frac{\partial }{\partial \Lambda}\int_{\mathbb{R}^N}U_{\overline{x}_1,\Lambda}^{q}\left(\sum_{j=2}^kU_{\overline{x}_j,\Lambda}
+\sum_{j=1}^kU_{\underline{x}_j,\Lambda}\right)
=&\sum_{j=2}^k\left(-\frac{B_0(N-2)}{\Lambda^{N-1}|\overline{x}_1-\overline{x}_j|^{N-2}}
+O\left(\frac{1}{|\Lambda(\overline{x}_1-\overline{x}_j)|^{N-1}}\right)\right)
\nonumber\\
&+\sum_{j=1}^k\left(-\frac{B_0(N-2)}{\Lambda^{N-1}|\overline{x}_1-\underline{x}_j|^{N-2}}
+O\left(\frac{1}{|\Lambda(\overline{x}_1-\underline{x}_j)|^{N-1}}\right)\right)
\nonumber\\
=&\frac{k(N-2)}{\Lambda^{N-1}}\left[\frac{B_4k^{N-2}}{(r\sqrt{1-h^2})^{N-2}}
+\frac{B_5k}{r^{N-2}h^{N-3}\sqrt{1-h^2}}\right]
\nonumber\\
&
+kO\left(\frac{1}{k^{\frac{m(N-2)}{N-2-m}+\theta}}\right).
\end{align*}
Then combining above identity and (\ref{ilamdaw}), we complete the proof.
\end{proof}

\begin{proposition}\label{iwh}
We have
\begin{align*}
\frac{\partial I(W_1,W_2)}{\partial h}
=&-\frac{k}{\Lambda^{N-2}}\left[\frac{(N-2)B_4k^{N-2}h}{r^{N-2}(\sqrt{1-h^2})^{N}}
-\frac{(N-3)B_5k}{r^{N-2}h^{N-2}\sqrt{1-h^2}}\right]
\nonumber\\
&+kO\left(\frac{1}{k^{\frac{m(N-2)}{N-2-m}+\frac{N-3}{N-1}+\theta}}\right),
\end{align*}
where $B_4$, $B_5$, $m$, $\theta$ are same positive constants in Proposition \ref{iw}.
\end{proposition}
\begin{proof}
We have
\begin{align}\label{h123}
\frac{\partial I(W_1,W_2)}{\partial h}
=&\frac{\partial }{\partial h}\int_{\mathbb{R}^N}\nabla W_1\nabla W_2
-\Bigg(\frac{1}{p+1}\frac{\partial }{\partial h}\int_{\mathbb{R}^N}K_1\left(\frac{|y|}{\mu}\right)|W_2|^{p+1}
\nonumber\\
&+\frac{1}{q+1}\frac{\partial }{\partial h}\int_{\mathbb{R}^N}K_2\left(\frac{|y|}{\mu}\right)|W_1|^{q+1}\Bigg)
\nonumber\\
=&2k\frac{\partial }{\partial h}\int_{\mathbb{R}^N}V_{\overline{x}_1,\Lambda}^{p}\left(\sum_{j=2}^kV_{\overline{x}_j,\Lambda}
+\sum_{j=1}^kV_{\underline{x}_j,\Lambda}\right)
\nonumber\\
&-\frac{1}{p+1}\int_{\mathbb{R}^N}K_1\left(\frac{|y|}{\mu}\right)\frac{\partial (W_2^{p+1})}{\partial h}
-\frac{1}{q+1}\int_{\mathbb{R}^N}K_2\left(\frac{|y|}{\mu}\right)\frac{\partial (W_1^{q+1})}{\partial h}
\nonumber\\
:=&H_1-H_2-H_3.
\end{align}
For $y\in \Omega_1^+$,
\begin{align*}
&\left|\frac{\partial }{\partial \Lambda}\left[W_1^{q+1}-U_{\overline{x}_1,\Lambda}^{q+1}-
(q+1)U_{\overline{x}_1,\Lambda}^{q}\left(\sum_{j=2}^kU_{\overline{x}_j,\Lambda}
+\sum_{j=1}^kU_{\underline{x}_j,\Lambda}\right)\right]
\right|
\nonumber\\
\leq &
\begin{cases}
CrU_{\overline{x}_1,\Lambda}^{\frac{q+1}{2}}\left(\sum_{j=2}^kU_{\overline{x}_j,\Lambda}
+\sum_{j=1}^kU_{\underline{x}_j,\Lambda}\right)^{\frac{q+1}{2}},\ & \text{ if } q+1\leq 3,\\
Cr\left[U_{\overline{x}_1,\Lambda}^{\frac{q+1}{2}}\left(\sum_{j=2}^kU_{\overline{x}_j,\Lambda}
+\sum_{j=1}^kU_{\underline{x}_j,\Lambda}\right)^{\frac{q+1}{2}}
+U_{\overline{x}_1,\Lambda}^{q-1}\left(\sum_{j=2}^kU_{\overline{x}_j,\Lambda}
+\sum_{j=1}^kU_{\underline{x}_j,\Lambda}\right)^2\right], \ & \text{ if } q+1>3.
\end{cases}
\end{align*}
We have that for $q+1>3$, similar to Proposition \ref{iw},
\begin{align*}
H_3=&\frac{2k}{q+1}\int_{\Omega_1^+}K_2\left(\frac{|y|}{\mu}\right)\frac{\partial (W_1^{q+1})}{\partial h}
\nonumber\\
=&\frac{2k}{q+1}\int_{\Omega_1^+}K_2\left(\frac{|y|}{\mu}\right)\frac{\partial}{\partial h}\Bigg\{U_{\overline{x}_1,\Lambda}^{q+1}+(q+1)U_{\overline{x}_1,\Lambda}^{q}\left(\sum_{j=2}^kU_{\overline{x}_j,\Lambda}
+\sum_{j=1}^kU_{\underline{x}_j,\Lambda}\right)
\Bigg\}
\nonumber\\
&+krO\left(\int_{\Omega_1^+}K_2\left(\frac{|y|}{\mu}\right)U_{\overline{x}_1,\Lambda}^{\frac{q+1}{2}}\left(\sum_{j=2}^kU_{\overline{x}_j,\Lambda}
+\sum_{j=1}^kU_{\underline{x}_j,\Lambda}\right)^{\frac{q+1}{2}}
\right)
\nonumber\\
&+krO\left(\int_{\Omega_1^+}K_2\left(\frac{|y|}{\mu}\right)U_{\overline{x}_1,\Lambda}^{q-1}\left(\sum_{j=2}^kU_{\overline{x}_j,\Lambda}
+\sum_{j=1}^kU_{\underline{x}_j,\Lambda}\right)^2
\right)
\nonumber\\
=&\frac{2k}{q+1}\int_{\Omega_1^+}K_2\left(\frac{|y|}{\mu}\right)\frac{\partial}{\partial h}
\Bigg\{U_{\overline{x}_1,\Lambda}^{q+1}+(q+1)U_{\overline{x}_1,\Lambda}^{q}\left(\sum_{j=2}^kU_{\overline{x}_j,\Lambda}
+\sum_{j=1}^kU_{\underline{x}_j,\Lambda}\right)
+k^2O\left(\left(\frac{k}{r}\right)^{N-\varepsilon_0}\right).
\end{align*}
Since
\begin{align*}
\frac{2k}{q+1}\int_{\Omega_1^+}K_2\left(\frac{|y|}{\mu}\right)\frac{\partial U_{\overline{x}_1,\Lambda}^{q+1}}{\partial h}
=&\frac{2k}{q+1}\int_{\Omega_1^+}\frac{\partial U_{\overline{x}_1,\Lambda}^{q+1}}{\partial h}
+\frac{2k}{q+1}\int_{\Omega_1^+}\left(K_2\left(\frac{|y|}{\mu}\right)-1\right)\frac{\partial U_{\overline{x}_1,\Lambda}^{q+1}}{\partial h}
\nonumber\\
=&\frac{2k}{q+1}\int_{\mathbb{R}^N}\frac{\partial U_{\overline{x}_1,\Lambda}^{q+1}}{\partial h}
+\frac{2k}{q+1}\int_{\mathbb{R}^N\setminus \Omega_1^+}\frac{\partial U_{\overline{x}_1,\Lambda}^{q+1}}{\partial h}
\nonumber\\
&+\frac{2k}{q+1}\int_{\Omega_1^+}\left(K_2\left(\frac{|y|}{\mu}\right)-1\right)\frac{\partial U_{\overline{x}_1,\Lambda}^{q+1}}{\partial h}
\nonumber\\
=&\frac{2k}{q+1}\int_{\mathbb{R}^N\setminus \Omega_1^+}\frac{\partial U_{\overline{x}_1,\Lambda}^{q+1}}{\partial h}
+\frac{2k}{q+1}\int_{\Omega_1^+}\left(K_2\left(\frac{|y|}{\mu}\right)-1\right)\frac{\partial U_{\overline{x}_1,\Lambda}^{q+1}}{\partial h},
\end{align*}
where
\begin{align*}
\int_{\mathbb{R}^N\setminus \Omega_1^+}\frac{\partial U_{\overline{x}_1,\Lambda}^{q+1}}{\partial h}
\leq C\left(\frac{k}{r}\right)^{N+1-\varepsilon_0}\int_{\mathbb{R}^N\setminus \Omega_1^+}
\frac{r}{(1+|y-\overline{x}_1|)^{(q+1)(N-2)-(N+1-\varepsilon_0)}}
=kO\left(\left(\frac{k}{r}\right)^{N-\varepsilon_0}\right).
\end{align*}
Moreover,
\begin{align*}
\int_{\Omega_1^+}\left(K_2\left(\frac{|y|}{\mu}\right)-1\right)\frac{\partial U_{\overline{x}_1,\Lambda}^{q+1}}{\partial h}
=\left\{\int_{\Omega_1^+\cap \{y:|\frac{|y|}{\mu}-1|\geq \delta\}}
+\int_{\Omega_1^+\cap \{y:|\frac{|y|}{\mu}-1|\leq \delta\}}\right\}
\left(K_2\left(\frac{|y|}{\mu}\right)-1\right)\frac{\partial U_{\overline{x}_1,\Lambda}^{q+1}}{\partial h}.
\end{align*}
For $|\frac{|y|}{\mu}-1|\geq \delta$, we have $|y-\overline{x}_1|\geq ||y|-\mu|-|\mu-|\overline{x}_1||\geq \frac{1}{2}\delta\mu$. Furthermore,
\begin{align*}
&\int_{\Omega_1^+\cap \{y:|\frac{|y|}{\mu}-1|\geq \delta\}}
\left(K_2\left(\frac{|y|}{\mu}\right)-1\right)\frac{\partial U_{\overline{x}_1,\Lambda}^{q+1}}{\partial h}
\nonumber\\
\leq &C\int_{\Omega_1^+\cap {y:|\frac{|y|}{\mu}-1|\geq \delta}}
\frac{r}{(1+|y-\overline{x}_1|)^{q(N-2)+(N-1)}}
\nonumber\\
\leq &C\int_{\Omega_1^+\cap {y:|\frac{|y|}{\mu}-1|\geq \delta}}\frac{1}{\mu^{N+1-\varepsilon_0}}
\frac{r}{(1+|y-\overline{x}_1|)^{q(N-2)+(N-1)-(N+1-\varepsilon_0)}}
=kO\left(\left(\frac{k}{r}\right)^{N-\varepsilon_0}\right).
\end{align*}
If $|\frac{|y|}{\mu}-1|\leq \delta$, it follows from the condition $(K_2)$ that
\begin{align*}
&\int_{\Omega_1^+\cap \{y:|\frac{|y|}{\mu}-1|\leq \delta\}}
\left(K_2\left(\frac{|y|}{\mu}\right)-1\right)\frac{\partial U_{\overline{x}_1,\Lambda}^{q+1}}{\partial h}
\nonumber\\
=&-\frac{c_2}{\mu^{m_2}}\int_{\Omega_1^+\cap \{y:|\frac{|y|}{\mu}-1|\leq \delta\}}||y|-\mu|^{m_2}
\frac{\partial U_{\overline{x}_1,\Lambda}^{q+1}}{\partial h}
+O\left(\frac{1}{\mu^{m_2+\theta_2}}\int_{\Omega_1^+\cap \{y:|\frac{|y|}{\mu}-1|\leq \delta\}}||y|-\mu|^{m_2+\theta_2}
\frac{\partial U_{\overline{x}_1,\Lambda}^{q+1}}{\partial h}
\right)
\nonumber\\
=&-\frac{c_2}{\mu^{m_2}}\int_{\mathbb{R}^N}||y|-\mu|^{m_2}
\frac{\partial U_{\overline{x}_1,\Lambda}^{q+1}}{\partial h}
+O\left(\int_{\mathbb{R}^N\setminus B_{\frac{r}{k}}(\overline{x}_1)}\left(\left(\frac{|y|}{\mu}\right)^{m_2}+1\right)
\frac{\partial U_{\overline{x}_1,\Lambda}^{q+1}}{\partial h}
\right)
\nonumber\\
&
+O\left(\frac{1}{\mu^{m_2+\theta_2}}\int_{\Omega_1^+\cap \{y:|\frac{|y|}{\mu}-1|\leq \delta\}}||y|-\mu|^{m_2+\theta_2}
\frac{\partial U_{\overline{x}_1,\Lambda}^{q+1}}{\partial h}
\right)
=kO\left(\left(\frac{k}{\mu}\right)^{N-\varepsilon_0}\right).
\end{align*}
Consequently,
\begin{align*}
\int_{\Omega_1^+}\left(K_2\left(\frac{|y|}{\mu}\right)-1\right)\frac{\partial U_{\overline{x}_1,\Lambda}^{q+1}}{\partial h}=kO\left(\left(\frac{k}{\mu}\right)^{N-\varepsilon_0}\right).
\end{align*}
Furthermore,
\begin{align*}
\frac{2k}{q+1}\int_{\Omega_1^+}K_2\left(\frac{|y|}{\mu}\right)\frac{\partial U_{\overline{x}_1,\Lambda}^{q+1}}{\partial h}
=k^2O\left(\left(\frac{k}{\mu}\right)^{N-\varepsilon_0}\right).
\end{align*}
In the same way, we also get
\begin{align*}
&\int_{\Omega_1^+}K_2\left(\frac{|y|}{\mu}\right)\frac{\partial}{\partial h}\Bigg\{
U_{\overline{x}_1,\Lambda}^{q}\left(\sum_{j=2}^kU_{\overline{x}_j,\Lambda}
+\sum_{j=1}^kU_{\underline{x}_j,\Lambda}\right)
\Bigg\}
\nonumber\\
=&\int_{\Omega_1^+}\frac{\partial}{\partial h}\Bigg\{
U_{\overline{x}_1,\Lambda}^{q}\left(\sum_{j=2}^kU_{\overline{x}_j,\Lambda}
+\sum_{j=1}^kU_{\underline{x}_j,\Lambda}\right)
\Bigg\}
\nonumber\\
&+\int_{\Omega_1^+}\left(K_2\left(\frac{|y|}{\mu}\right)-1\right)\frac{\partial}{\partial h}\Bigg\{
U_{\overline{x}_1,\Lambda}^{q}\left(\sum_{j=2}^kU_{\overline{x}_j,\Lambda}
+\sum_{j=1}^kU_{\underline{x}_j,\Lambda}\right)
\Bigg\}
\nonumber\\
=&\int_{\mathbb{R}^N}\frac{\partial}{\partial h}\Bigg\{
U_{\overline{x}_1,\Lambda}^{q}\left(\sum_{j=2}^kU_{\overline{x}_j,\Lambda}
+\sum_{j=1}^kU_{\underline{x}_j,\Lambda}\right)
\Bigg\}
\nonumber\\
&+\int_{\mathbb{R}^N\setminus \Omega_1^+}\frac{\partial}{\partial h}\Bigg\{
U_{\overline{x}_1,\Lambda}^{q}\left(\sum_{j=2}^kU_{\overline{x}_j,\Lambda}
+\sum_{j=1}^kU_{\underline{x}_j,\Lambda}\right)
\Bigg\}
\nonumber\\
&+\int_{\Omega_1^+}\left(K_2\left(\frac{|y|}{\mu}\right)-1\right)\frac{\partial}{\partial h}\Bigg\{
U_{\overline{x}_1,\Lambda}^{q}\left(\sum_{j=2}^kU_{\overline{x}_j,\Lambda}
+\sum_{j=1}^kU_{\underline{x}_j,\Lambda}\right)
\Bigg\}
\nonumber\\
=&\int_{\mathbb{R}^N}\frac{\partial}{\partial h}\Bigg\{
U_{\overline{x}_1,\Lambda}^{q}\left(\sum_{j=2}^kU_{\overline{x}_j,\Lambda}
+\sum_{j=1}^kU_{\underline{x}_j,\Lambda}\right)
\Bigg\}
+kO\left(\left(\frac{k}{r}\right)^{N-\varepsilon_0}\right)
\nonumber\\
&+\int_{\mathbb{R}^N\setminus \Omega_1^+}\frac{\partial}{\partial h}\Bigg\{
U_{\overline{x}_1,\Lambda}^{q}\left(\sum_{j=2}^kU_{\overline{x}_j,\Lambda}
+\sum_{j=1}^kU_{\underline{x}_j,\Lambda}\right)
\Bigg\}.
\end{align*}

From the fact that $|y-\overline{x}_1|\geq C\frac{\mu}{k}$ for $y\in \mathbb{R}^N\setminus\Omega_1^+$,
\begin{align*}
&\frac{\partial}{\partial h}\int_{\mathbb{R}^N\setminus \Omega_1^+}U_{\overline{x}_1,\Lambda}^{q}U_{\overline{x}_j,\Lambda}
\nonumber\\
=&q\int_{\mathbb{R}^N\setminus \Omega_1^+}U_{\overline{x}_1,\Lambda}^{q-1}\frac{\partial U_{\overline{x}_1,\Lambda}}{\partial h}U_{\overline{x}_j,\Lambda}
+\int_{\mathbb{R}^N\setminus \Omega_1^+}U_{\overline{x}_1,\Lambda}^{q}\frac{\partial U_{\overline{x}_j,\Lambda}}{\partial h}
\nonumber\\
\leq &C\int_{\mathbb{R}^N\setminus \Omega_1^+}\frac{r}{(1+|y-\overline{x}_1|)^{(q-1)(N-2)+(N+1)}}
\frac{1}{(1+|y-\overline{x}_j|)^{N-2}}
\nonumber\\
&+C\int_{\mathbb{R}^N\setminus \Omega_1^+}\frac{1}{(1+|y-\overline{x}_1|)^{q(N-2)}}
\frac{r}{(1+|y-\overline{x}_j|)^{N-1}}
\nonumber\\
\leq &C\left(\frac{k}{\mu}\right)^{N+1-\varepsilon_0}\int_{\mathbb{R}^N\setminus \Omega_1^+}\frac{r}{(1+|y-\overline{x}_1|)^{q(N-2)-N+\varepsilon_0}}
\frac{1}{(1+|y-\overline{x}_j|)^{N-2}}
\nonumber\\
&+C\left(\frac{k}{\mu}\right)^{N+1-\varepsilon_0}\int_{\mathbb{R}^N\setminus \Omega_1^+}\frac{1}{(1+|y-\overline{x}_1|)^{q(N-2)-N-1+\varepsilon_0}}
\frac{r}{(1+|y-\overline{x}_j|)^{N-1}}
\nonumber\\
\leq &Ck\left(\frac{k}{\mu}\right)^{N-\varepsilon_0}\int_{\mathbb{R}^N\setminus \Omega_1^+}
\frac{1}{|\overline{x}_1-\overline{x}_j|^{\alpha}}\Bigg(\frac{1}{(1+|y-\overline{x}_1|)^{(q+1)(N-2)-N+\varepsilon_0-\alpha}}
\nonumber\\
&+\frac{1}{(1+|y-\overline{x}_j|)^{(q+1)(N-2)-N+\varepsilon_0-\alpha}}
\Bigg)
\nonumber\\
&+Ck\left(\frac{k}{\mu}\right)^{N-\varepsilon_0}\int_{\mathbb{R}^N\setminus \Omega_1^+}
\frac{1}{|\overline{x}_1-\overline{x}_j|^{\beta}}\Bigg(\frac{1}{(1+|y-\overline{x}_1|)^{(q+1)(N-2)-N+\varepsilon_0-\beta}}
\nonumber\\
&+\frac{1}{(1+|y-\overline{x}_j|)^{(q+1)(N-2)-N+\varepsilon_0-\beta}}
\Bigg)
=kO\left(\left(\frac{k}{\mu}\right)^{N-\varepsilon_0}\right),
\end{align*}
where for $q>\frac{N+2}{N-2}+\frac{N-2-m}{(N-2)^2}$, there exist $\alpha$, $\beta$ satisfying
\begin{align*}
\frac{N-2-m}{N-2}<\alpha<(q+1)(N-2)-2N,\ \frac{N-2-m}{N-2}<\beta<(q+1)(N-2)-2N,
\end{align*}
such that $(q+1)(N-2)-N+\varepsilon_0-\alpha>N$ and $(q+1)(N-2)-N+\varepsilon_0-\beta>N$.

Similarly,
$\frac{\partial}{\partial h}\int_{\mathbb{R}^N\setminus \Omega_1^+}U_{\overline{x}_1,\Lambda}^{q}U_{\underline{x}_j,\Lambda}=kO\left(\left(\frac{k}{\mu}\right)^{N-\varepsilon_0}\right)$ for $j=1,\cdots,k$. Hence,
\begin{align*}
\int_{\Omega_1^+}K_2\left(\frac{|y|}{\mu}\right)\frac{\partial}{\partial h}\Bigg\{
U_{\overline{x}_1,\Lambda}^{q}\left(\sum_{j=2}^kU_{\overline{x}_j,\Lambda}
+\sum_{j=1}^kU_{\underline{x}_j,\Lambda}\right)
\Bigg\}
=&\int_{\mathbb{R}^N}\frac{\partial}{\partial h}\Bigg\{
U_{\overline{x}_1,\Lambda}^{q}\left(\sum_{j=2}^kU_{\overline{x}_j,\Lambda}
+\sum_{j=1}^kU_{\underline{x}_j,\Lambda}\right)
\Bigg\}
\nonumber\\
&+kO\left(\left(\frac{k}{\mu}\right)^{N-\varepsilon_0}\right).
\end{align*}
In conclusion,
\begin{align}\label{h3}
H_3=-2k\int_{\mathbb{R}^N}\frac{\partial}{\partial h}\Bigg\{
U_{\overline{x}_1,\Lambda}^{q}\left(\sum_{j=2}^kU_{\overline{x}_j,\Lambda}
+\sum_{j=1}^kU_{\underline{x}_j,\Lambda}\right)
\Bigg\}
+k^2O\left(\left(\frac{k}{\mu}\right)^{N-\varepsilon_0}\right).
\end{align}
We can also get
\begin{align}\label{h2}
H_2=-2k\int_{\mathbb{R}^N}\frac{\partial}{\partial h}\Bigg\{
V_{\overline{x}_1,\Lambda}^{p}\left(\sum_{j=2}^kV_{\overline{x}_j,\Lambda}
+\sum_{j=1}^kV_{\underline{x}_j,\Lambda}\right)
\Bigg\}
+k^2O\left(\left(\frac{k}{\mu}\right)^{N-\varepsilon_0}\right).
\end{align}
Then from (\ref{h123}),
\begin{align}\label{h1234}
\frac{\partial I(W_1,W_2)}{\partial h}
=-2k\int_{\mathbb{R}^N}\frac{\partial}{\partial h}\Bigg\{
U_{\overline{x}_1,\Lambda}^{q}\left(\sum_{j=2}^kU_{\overline{x}_j,\Lambda}
+\sum_{j=1}^kU_{\underline{x}_j,\Lambda}\right)
\Bigg\}
+k^2O\left(\left(\frac{k}{\mu}\right)^{N-\varepsilon_0}\right).
\end{align}
Next, we estimate $\int_{\mathbb{R}^N}\frac{\partial}{\partial h}\Bigg\{
U_{\overline{x}_1,\Lambda}^{q}\left(\sum_{j=2}^kU_{\overline{x}_j,\Lambda}
+\sum_{j=1}^kU_{\underline{x}_j,\Lambda}\right)\Bigg\}$. For $j=2,\cdots, k$, we have
\begin{align*}
&\frac{\partial}{\partial h}\int_{\mathbb{R}^N}U_{\overline{x}_1,\Lambda}^{q}U_{\overline{x}_j,\Lambda}
\nonumber\\
=&\frac{\partial}{\partial h}\int_{\mathbb{R}^N}U_{0,1}^{q}(z)U_{0,1}(z+\Lambda(\overline{x}_1-\overline{x}_j))
\nonumber\\
=&\int_{\mathbb{R}^N}U_{0,1}^{q}(z)\frac{\partial U_{0,1}(z+\Lambda(\overline{x}_1-\overline{x}_j))}{\partial h}
\nonumber\\
=&\int_{\mathbb{R}^N}U_{0,1}^{q}(z)\frac{\partial U_{0,1}(\Lambda(\overline{x}_1-\overline{x}_j))}{\partial h}
+\int_{\mathbb{R}^N}U_{0,1}^{q}(z)\Bigg(\frac{\partial U_{0,1}(z+\Lambda(\overline{x}_1-\overline{x}_j))}{\partial h}
-\frac{\partial U_{0,1}(\Lambda(\overline{x}_1-\overline{x}_j))}{\partial h}
\Bigg)
\nonumber\\
=&\Lambda\int_{\mathbb{R}^N}U_{0,1}^{q}(z)U'_{0,1}(\Lambda(\overline{x}_1-\overline{x}_j))\frac{\partial |\overline{x}_1-\overline{x}_j|}{\partial h}
\nonumber\\
&+\int_{\mathbb{R}^N}U_{0,1}^{q}(z)\Bigg(\frac{\partial U_{0,1}(z+\Lambda(\overline{x}_1-\overline{x}_j))}{\partial h}
-\frac{\partial U_{0,1}(\Lambda(\overline{x}_1-\overline{x}_j))}{\partial h}
\Bigg).
\end{align*}
It follows from $|\overline{x}_1-\overline{x}_j|^2=4(1-h^2)r^2\sin^2r$ that
\begin{align*}
\frac{\partial |\overline{x}_1-\overline{x}_j|}{\partial h}
=\frac{1}{2}\frac{1}{|\overline{x}_1-\overline{x}_j|}\left(-8hr^2\sin^2r\right)
=-\frac{4hr^2\sin^2r}{|\overline{x}_1-\overline{x}_j|}.
\end{align*}
Then by the decay property of $U'$,
\begin{align*}
&\Lambda\int_{\mathbb{R}^N}U_{0,1}^{q}(z)U'_{0,1}(\Lambda(\overline{x}_1-\overline{x}_j))\frac{\partial |\overline{x}_1-\overline{x}_j|}{\partial h}
\nonumber\\
=&\Lambda\int_{\mathbb{R}^N}U_{0,1}^{q}(z)\left(\frac{-b_{N,p}(N-2)}{\Lambda^{N-1}|\overline{x}_1-\overline{x}_j|^{N-1}}
+O\left(\frac{1}{|\overline{x}_1-\overline{x}_j|^{N-1+\kappa_0}}\right)
\right)
\left(-\frac{4hr^2\sin^2r}{|\overline{x}_1-\overline{x}_j|}\right)
\nonumber\\
=&\int_{\mathbb{R}^N}U_{0,1}^{q}(z)\frac{b_{N,p}(N-2)}{\Lambda^{N-2}|\overline{x}_1-\overline{x}_j|^{N-2}}
\frac{4hr^2\sin^2r}{|\overline{x}_1-\overline{x}_j|^2}
+O\left(\frac{h}{(1-h^2)|\overline{x}_1-\overline{x}_j|^{N-2+\kappa_0}}\right)
\nonumber\\
=&\int_{\mathbb{R}^N}U_{0,1}^{q}(z)\frac{b_{N,p}(N-2)}{\Lambda^{N-2}|\overline{x}_1-\overline{x}_j|^{N-2}}
\frac{h}{1-h^2}
+O\left(\frac{h}{(1-h^2)|\overline{x}_1-\overline{x}_j|^{N-2+\kappa_0}}\right)
\nonumber\\
=&\int_{\mathbb{R}^N}U_{0,1}^{q}(z)\frac{b_{N,p}(N-2)}{\Lambda^{N-2}|\overline{x}_1-\overline{x}_j|^{N-2}}
\frac{h}{1-h^2}
+O\left(\frac{1}{k^{\frac{m(N-2)}{N-2-m}+\frac{N-3}{N-1}+\theta}}\right),
\end{align*}
where $b_{N,p}$ is defined in \cite[Lemma 2.2]{kimm}.

Using the mean value Theorem, we deduce that there exists a $t\in (0,1)$ such that
\begin{align*}
&\int_{\mathbb{R}^N}U_{0,1}^{q}(z)\frac{\partial}{\partial h}\Bigg(
U_{0,1}(z+\Lambda(\overline{x}_1-\overline{x}_j))
-U_{0,1}(\Lambda(\overline{x}_1-\overline{x}_j))
\Bigg)dz
\nonumber\\
=&\int_{\mathbb{R}^N}U_{0,1}^{q}(z)\frac{\partial}{\partial h}\Bigg(
U'_{0,1}(tz+\Lambda(\overline{x}_1-\overline{x}_j))\cdot z
\Bigg)dz
\nonumber\\
=&\int_{\mathbb{R}^N}U_{0,1}^{q}\left(\frac{s-\Lambda(\overline{x}_1-\overline{x}_j)}{t}\right)
\frac{\partial U'_{0,1}(s)}{\partial h}\frac{s-\Lambda(\overline{x}_1-\overline{x}_j)}{t}\frac{1}{t^N}ds
\nonumber\\
=&\int_{\mathbb{R}^N}U_{0,1}^{q}\left(\frac{s-\Lambda(\overline{x}_1-\overline{x}_j)}{t}\right)
\frac{\partial U'_{0,1}(s)}{\partial s}\frac{\partial s}{\partial h}
\frac{s-\Lambda(\overline{x}_1-\overline{x}_j)}{t}\frac{1}{t^N}ds
\nonumber\\
=&\int_{\mathbb{R}^N}U_{0,1}^{q}\left(\frac{s-\Lambda(\overline{x}_1-\overline{x}_j)}{t}\right)
\frac{\partial U'_{0,1}(s)}{\partial s}
\frac{4hr^2\cos^2(\frac{j-1}{k}\pi)}{\Lambda(\overline{x}_1-\overline{x}_j)}
\frac{s-\Lambda(\overline{x}_1-\overline{x}_j)}{t}\frac{1}{t^N}ds
\nonumber\\
\leq &C\int_{\mathbb{R}^N}U_{0,1}^{q}\left(\frac{s-\Lambda(\overline{x}_1-\overline{x}_j)}{t}\right)
\frac{4hr^2\cos^2(\frac{j-1}{k}\pi)}{(\Lambda(\overline{x}_1-\overline{x}_j))^{N+1}}
\frac{s-\Lambda(\overline{x}_1-\overline{x}_j)}{t}ds
\nonumber\\
\leq &C\left(\int_{|s|\geq 2\Lambda(\overline{x}_1-\overline{x}_j)}
+\int_{|s|\leq 2\Lambda(\overline{x}_1-\overline{x}_j)}\right)
U_{0,1}^{q}\left(\frac{s-\Lambda(\overline{x}_1-\overline{x}_j)}{t}\right)
\frac{4hr^2\cos^2(\frac{j-1}{k}\pi)}{(\Lambda(\overline{x}_1-\overline{x}_j))^{N+1}}
\frac{s-\Lambda(\overline{x}_1-\overline{x}_j)}{t}ds
\nonumber\\
=&O\left(\frac{hr^2\cos^2(\frac{j-1}{k}\pi)}{(\Lambda(\overline{x}_1-\overline{x}_j))^{N+1}}
\frac{1}{(\Lambda(\overline{x}_1-\overline{x}_j))^{q(N-2)-N-1}}
\right)
\nonumber\\
=&O\left(\frac{hr^2}{(\Lambda(\overline{x}_1-\overline{x}_j))^{q(N-2)}}
\right)
=O\left(k\left(\frac{k}{r}\right)^{N-\varepsilon_0}hk\left(\frac{k}{r}\right)^{q(N-2)-2+\varepsilon_0-N}
\right)
=O\left(k\left(\frac{k}{r}\right)^{N-\varepsilon_0}\right).
\end{align*}
Here we use the fact that $q\geq \frac{N+2}{N-2}+\frac{2(N-2-m)}{m(N-1)(N-2)}-\frac{\varepsilon_0}{N-2}$. Consequently, Using the fact that
\begin{align*}
m\in
\begin{cases}
[2,N-2), \ & \text{ if } N=5, 6,\\
(\frac{(N-2)^2}{2N-3},N-2), \ & \text{ if } N\geq 7,
\end{cases}
\end{align*}
we get
\begin{align*}
\frac{\partial}{\partial h}\int_{\mathbb{R}^N}U_{\overline{x}_1,\Lambda}^{q}U_{\overline{x}_j,\Lambda}
=&\int_{\mathbb{R}^N}U_{0,1}^{q}(z)\frac{b_{N,p}(N-2)}{\Lambda^{N-2}|\overline{x}_1-\overline{x}_j|^{N-2}}
\frac{h}{1-h^2}
+O\left(\frac{1}{k^{\frac{m(N-2)}{N-2-m}+\frac{N-3}{N-1}+\theta}}\right)
\nonumber\\
&+O\left(k\left(\frac{k}{r}\right)^{N-\varepsilon_0}\right)
\nonumber\\
=&\int_{\mathbb{R}^N}U_{0,1}^{q}(z)\frac{b_{N,p}(N-2)}{\Lambda^{N-2}|\overline{x}_1-\overline{x}_j|^{N-2}}
\frac{h}{1-h^2}
+O\left(\frac{1}{k^{\frac{m(N-2)}{N-2-m}+\frac{N-3}{N-1}+\theta}}\right).
\end{align*}
In the same way,
\begin{align*}
\frac{\partial}{\partial h}\int_{\mathbb{R}^N}U_{\overline{x}_1,\Lambda}^{q}U_{\underline{x}_j,\Lambda}
=\int_{\mathbb{R}^N}U_{0,1}^{q}(z)\frac{b_{N,p}(N-2)}{\Lambda^{N-2}|\overline{x}_1-\underline{x}_j|^{N-2}}
\frac{h}{1-h^2}
+O\left(\frac{1}{k^{\frac{m(N-2)}{N-2-m}+\frac{N-3}{N-1}+\theta}}\right).
\end{align*}
Hence, it follows from \cite[Lemma A.1]{duanmw} that
\begin{align*}
&\int_{\mathbb{R}^N}\frac{\partial}{\partial h}\Bigg\{
U_{\overline{x}_1,\Lambda}^{q}\left(\sum_{j=2}^kU_{\overline{x}_j,\Lambda}
+\sum_{j=1}^kU_{\underline{x}_j,\Lambda}\right)\Bigg\}
\nonumber\\
=&-\frac{k}{\Lambda^{N-2}}\left[\frac{(N-2)B_4k^{N-2}h}{r^{N-2}(\sqrt{1-h^2})^{N}}
-\frac{(N-3)B_5k}{r^{N-2}h^{N-2}\sqrt{1-h^2}}\right]
+O\left(\frac{1}{k^{\frac{m(N-2)}{N-2-m}+\frac{N-3}{N-1}+\theta}}\right).
\end{align*}
Then by (\ref{h1234}) and above equation, we conclude the proof.
\end{proof}

\end{document}